\documentclass[a4paper, reqno]{amsart}

\usepackage{lmodern,microtype}
\usepackage{xr-hyper}
\usepackage[dvipsnames,svgnames,x11names,hyperref]{xcolor}
\usepackage[colorlinks,citecolor=Mahogany,linkcolor=Mahogany,urlcolor=Mahogany,filecolor=Mahogany]{hyperref}

\usepackage{amsmath,amstext,amssymb,amsthm,amsfonts,mathtools,mathrsfs,stmaryrd,indentfirst,bbm,units,verbatim,enumerate}
\usepackage{caption}
\numberwithin{equation}{section}

\newcommand{\cat}[1]{\mathsf{#1}}

\newcommand{\mr}[1]{{\rm #1}}

\usepackage{tikz,tikz-cd}

\newcommand{\CircNum}[1]{\ooalign{\hfil\raise .00ex\hbox{\scriptsize #1}\hfil\crcr\mathhexbox20D}}

\newcommand{\bA}{\mathbb{A}}

\newcommand{\bC}{\mathbb{C}}

\newcommand{\bE}{\mathbb{E}}
\newcommand{\bF}{\mathbb{F}}

\newcommand{\bL}{\mathbb{L}}

\newcommand{\bN}{\mathbb{N}}

\newcommand{\bQ}{\mathbb{Q}}
\newcommand{\bR}{\mathbb{R}}

\newcommand{\bZ}{\mathbb{Z}}

\newcommand{\gA}{\bold{A}}
\newcommand{\gB}{\bold{B}}
\newcommand{\gC}{\bold{C}}
\newcommand{\gD}{\bold{D}}
\newcommand{\gE}{\bold{E}}

\newcommand{\gI}{\bold{I}}

\newcommand{\gK}{\bold{K}}
\newcommand{\gL}{\bold{L}}
\newcommand{\gM}{\bold{M}}
\newcommand{\gN}{\bold{N}}

\newcommand{\gR}{\bold{R}}

\newcommand{\gT}{\bold{T}}

\newcommand{\gX}{\bold{X}}

\newcommand{\cC}{\mathcal{C}}

\newcommand{\cT}{\mathcal{T}}

\newcommand{\bk}{\mathbbm{k}}

\newcommand\lra{\longrightarrow}
\newcommand\lla{\longleftarrow}

\newcommand\End{\mathrm{End}}

\newcommand\colim{\operatorname*{colim}}
\newcommand\hocolim{\operatorname*{hocolim}}

\newcommand\Ker{\operatorname*{Ker}}

\newcommand{\hcoker}{/\!\!/}

\newcommand{\fS}{\mathfrak{S}}

\newcommand{\N}{\mathbb{N}}
\newcommand{\Z}{\mathbb{Z}}

\newcommand{\Hom}{\mathrm{Hom}}

\newcommand{\Aut}{\mathrm{Aut}}

\newcommand{\Simp}{\mathrm{Simp}}

\newcommand{\Cok}{\mathrm{Cok}}

\renewcommand{\epsilon}{\varepsilon}

\newcommand{\nin}{\not\in}

\newcommand{\Alg}{\cat{Alg}}

\newcommand{\R}{A} 

\mathchardef\ordinarycolon\mathcode`\:
\mathcode`\:=\string"8000
\begingroup \catcode`\:=\active
  \gdef:{\mathrel{\mathop\ordinarycolon}}
\endgroup

\usepackage{amsthm}
\theoremstyle{plain}
\newtheorem{MainThm}{Theorem}
\newtheorem{MainCor}[MainThm]{Corollary}

\newtheorem{theorem}{Theorem}[section]

\newtheorem{proposition}[theorem]{Proposition}
\newtheorem{lemma}[theorem]{Lemma}

\newtheorem{corollary}[theorem]{Corollary}
\newtheorem{conjecture}[theorem]{Conjecture}

\theoremstyle{definition}
\newtheorem{definition}[theorem]{Definition}
\newtheorem{convention}[theorem]{Convention}

\theoremstyle{remark}

\newtheorem{remark}[theorem]{Remark}
\newtheorem{example}[theorem]{Example}
\newtheorem*{remark*}{Remark}

\title[$E_\infty$-cells and general linear groups of infinite fields]{$E_\infty$-cells and general linear groups\\of infinite fields}

\author{S{\o}ren Galatius}
\address{Department of Mathematics\\
	University of Copenhagen\\
	Denmark}
\email{galatius@math.ku.dk}

\author{Alexander Kupers}
\address{Department of Computer and Mathematical Sciences \\ University of Toronto Scarborough \\ 
1265 Military Trail, Toronto, ON M1C 1A4 \\
Canada}
\email{a.kupers@utoronto.ca}

\author{Oscar Randal-Williams}
\address{Centre for Mathematical Sciences\\
 Wilberforce Road\\
 Cambridge CB3 0WB\\
 UK}
\email{o.randal-williams@dpmms.cam.ac.uk}

\date{\today}

\begin{document}

\begin{abstract} We study the general linear groups of infinite fields (or more generally connected semi-local rings with infinite residue fields) from the perspective of $E_\infty$-algebras. We prove that there is a vanishing line of slope 2 for their $E_\infty$-homology, and analyse the groups on this line by determining all invariant bilinear forms on Steinberg modules. We deduce from this a number of consequences regarding the unstable homology of general linear groups, in particular answering questions of Rognes, Suslin, Mirzaii, and others.  
\end{abstract}

\vspace*{-4ex}

\maketitle

\thispagestyle{empty}
\enlargethispage{2ex}
\vspace{-7ex}

\renewcommand\contentsname{}

\setcounter{tocdepth}{1}
\tableofcontents

\section{Introduction}

The algebraic $K$-theory of a commutative ring $\R$ arises as the group-completion of an $E_\infty$-algebra
\[\gR \simeq \coprod_{[M]} B\mr{GL}(M),\]
where the coproduct is over isomorphism classes of finitely-generated projective $\R$-modules. In this paper we establish a fundamental new property of this $E_\infty$-algebra \emph{before} group-completion, in the case that $\R$ is an infinite field or more generally a semi-local ring with all residue fields infinite. This property concerns the grading of $\gR$ by rank, and to define this we must additionally assume that $\R$ is connected. Then there is a decomposition $\gR \simeq \coprod_{n \geq 0} \gR(n)$ with
\begin{equation*}
	\gR(n) \simeq \coprod_{\substack{[M]\\r(M)=n}} B\mr{GL}(M),
\end{equation*}
where the coproduct is now over isomorphism classes of finitely-generated projective $\R$-modules of rank $n$. As rank is additive under direct sum of $\R$-modules, this endows $\gR$ with the structure of an $\bN$-graded $E_\infty$-algebra (i.e.\ an $E_\infty$-algebra in the category $\cat{sSet}^\bN$ of functors from $\bN$, regarded as a category with only identity morphisms, to simplicial sets).

There is a homology theory for $E_\infty$-algebras, an $E_\infty$-version of Andr\'e--Quillen homology for simplicial commutative rings. We will refer to the corresponding homology groups as \emph{$E_\infty$-homology groups}; the reader should see Section \ref{sec:overview-e_k-cells} and \cite{e2cellsIv3} for more details.  It associates bigraded abelian groups $H^{E_\infty}_{n,d} (\gR)$ to the object $\gR \in \cat{Alg}_{E_\infty}(\smash{\cat{sSet}^\bN})$ described above, where the $n$-grading comes from the rank and the $d$-grading is the homological degree. In fact, these $E_\infty$-homology groups depend only on the $\bZ$-linearisation $\gR_\bZ$, considered as an $E_\infty$-algebra in the category $\smash{\cat{sMod}_\bZ^\bN}$ of $\bN$-graded simplicial $\bZ$-modules.

One interpretation of these $E_\infty$-homology groups is that they describe how to construct a minimal cellular (or even CW-) approximation to $\gR_\bZ$ within the category $\cat{Alg}_{E_\infty}(\smash{\cat{sMod}^\bN_\bZ})$, which in turn may be used to infer properties of $\pi_*(\gR_\bZ(n))$, the homology of general linear groups over $\R$. We have used a similar device to study an $E_2$-algebra constructed from mapping class groups of surfaces in \cite{e2cellsII} and to study an $E_\infty$-algebra constructed from general linear groups of finite fields in \cite{e2cellsIII}, but the main result of this paper is that in the situation described here the $E_\infty$-homology groups enjoy exceptionally strong vanishing properties:

\begin{MainThm}\label{thm:A}\ 
If $\R$ is a connected semi-local ring with infinite residue fields, then $H_{n,d}^{E_\infty}(\gR_\bZ)=0$ for $d < 2n-2$.
\end{MainThm}

In contrast, in all other situations where we know analogous vanishing results the range is of the form $d < n + \text{constant}$: the situation here is ``twice as good''. In view of the vanishing range in Theorem \ref{thm:A} it becomes important to evaluate the groups $H_{n,d}^{E_\infty}(\gR_\bZ)$ in the critical degrees $d = 2n-2$.

\begin{MainThm}\label{thm:B}\ 
If $\R$ is an infinite field, then $H_{n, 2n-2}^{E_\infty}(\gR_\bZ)$ is given by
\[\mr{Tor}_{1}^{\Gamma_\bZ[x]}(\bZ,\bZ)_n
\cong \begin{cases}
	\bZ \,\{x\} & \text{if $n=1$},\\
	\bZ/p \,\{\gamma_{p^k}(x)\} & \text{if $n = p^k$ with $p$ prime},\\
	0 & \text{otherwise},
\end{cases}\]
and in particular is finite for $n \geq 2$.
\end{MainThm}

\begin{remark*} Theorem~\ref{thm:B} would follow more generally for any connected semi-local ring with infinite residue fields, provided a suitable analogue of Theorem~\ref{thm:C} below held for such rings. We formulate this as Conjecture~\ref{conj:double-steinberg-local}.  In Section~\ref{sec:coinv-e_2-steinb} we prove this conjecture for projective modules of rank $\leq 3$, which is sufficient for some of our intended applications.\end{remark*}

\subsection{Rognes' connectivity conjecture} \label{sec:rognes-connec-intr}

These results are directly related to conjectures of Rognes about his spectrum-level rank filtration of the algebraic $K$-theory spectrum $\gK(\R)$ of a commutative ring $A$ \cite{rognesrank}.  Let us for notational reasons assume $A$ is connected and has the property that all finitely-generated projective modules are free.  Rognes' rank filtration is then an ascending exhaustive filtration
\[* \subset F_0 \gK(\R) \subset F_1 \gK(\R) \subset F_2 \gK(\R) \subset \cdots \subset \gK(\R)\]
by subspectra, and he identified the filtration quotients as
\[\frac{F_n\gK(\R)}{F_{n-1}\gK(\R)} \simeq \gD(\R^n) \hcoker \mr{GL}_n(\R),\]
the homotopy orbits of $\mr{GL}_n(\R)$ acting on a certain spectrum $\gD(\R^n)$ called the \emph{stable building}. Assuming further that $\R$ is a ring with ``many units'' (e.g.\ a semi-local ring with infinite residue fields) we will show that this is related to $E_\infty$-homology by an isomorphism
\begin{equation}\label{eq:EInftyIsStableBuilding}
	H^{E_\infty}_{n,d}(\gR_\bZ) \cong H_{d}(\gD(\R^{n}) \hcoker \mr{GL}_{n}(\R);\bZ).
\end{equation}

Based on calculations for $n \leq 3$, Rognes has conjectured that, for $\R$ a Euclidean domain or local ring, $\gD(\R^n)$ is $(2n-3)$-connected (the ``Connectivity Conjecture'', \cite[Conjecture 12.3]{rognesrank}, \cite[Conjecture 4.6.1]{RognesMot}, \cite[Conjecture 1.2]{Rognes16notes}). He also conjectured that the $\mr{GL}_n(\R)$-coinvariants of $H_{2n-2}(\gD(\R^n))$ are torsion for $n \geq 2$ \cite[Conjecture 4.6.3]{RognesMot}.

It seems that for many of the intended applications it suffices to know these results after taking homotopy orbits by $\mr{GL}_n(\R)$, and in view of \eqref{eq:EInftyIsStableBuilding} this is what we have established in Theorem \ref{thm:A}, at least in the semi-local ring case, and Theorem \ref{thm:B}, at least in the infinite field case. Since this paper appeared, Miller, Patzt, and Wilson have proved the connectivity conjecture for all fields \cite[Theorem 1.2]{MillerPatztWilson}.

\subsection{An invariant pairing on the Steinberg module}
\label{sec:bilin-form-steinb} 
The proof of Theorem~\ref{thm:B} involves a result about Steinberg modules.

When $A = \bF$ is a field and $V$ is a non-zero finite-dimensional $\bF$-vector space, the Solomon--Tits theorem asserts that the nerve of the partially ordered set of proper non-zero linear subspaces of $V$ has the homotopy type of a wedge of spheres of dimension $\dim(V) - 2$.  This nerve is the \emph{Tits building} $T(V)$, and the \emph{Steinberg module} is the reduced homology group
\begin{equation*}
  \mr{St}(V) \coloneqq \tilde{H}_{\dim(V)-2}({T}(V);\bZ).
\end{equation*}
As the top-dimensional homology of a simplicial complex, $\mr{St}(V)$ is a subgroup of the augmented simplicial chains $\smash{\widetilde{C}}_{\dim(V) - 2}(T(V))$, which is free abelian on the set of complete flags in $V$.  We therefore obtain a pairing on $\mr{St}(V)$ by restricting the pairing on simplicial chains in which the set of complete flags form an orthonormal basis. The resulting pairing
\begin{equation*}
  \langle -, -\rangle \colon \mr{St}(V) \otimes \mr{St}(V) \lra \bZ
\end{equation*}
is symmetric, bilinear, $\mr{GL}(V)$-invariant, and positive definite, and we will show that it is universal among bilinear $\mr{GL}(V)$-invariant pairings:

\begin{MainThm}\label{thm:C}
  The induced map on coinvariants $(\mr{St}(V) \otimes \mr{St}(V))_{\mr{GL}(V)} \to \bZ$
  is an isomorphism.
\end{MainThm}
There are natural homomorphisms $\mr{St}(V) \otimes \mr{St}(W) \to \mr{St}(V \oplus W)$, giving $n \mapsto (\mr{St}(\bF^n) \otimes \mr{St}(\bF^n))_{\mr{GL}_n(\bF)} \cong \bZ$ the structure of a graded-commutative ring; in Section~\ref{sec:products} we show that the $\bZ$'s in each degree assemble to a divided power algebra. Theorem \ref{thm:C} provides the input for Theorem \ref{thm:B} through the identification of the $E_2$-homology groups as
\[H^{E_2}_{n,d}(\gR_\bZ) \cong H_{d-2(n-1)}(\mr{GL}_n(\R);\mr{St}(\R^n) \otimes \mr{St}(\R^n))\]
when $\R$ is an infinite field, showing that $H^{E_2}_{n,2n-2}(\gR_\bZ) \cong \bZ$ for all $n \geq 1$. Proving Theorem \ref{thm:B} amounts to understanding how this changes when passing from $E_2$- to $E_\infty$-homology. The reader may recognise the formula in Theorem \ref{thm:B} as the indecomposables of a divided power algebra.

We prove Theorem~\ref{thm:C} in Section~\ref{sec:coinvariants-of}, and in Section~\ref{sec:indec-steinb-modul} we discuss several consequences, including the following.

\begin{MainCor}\label{thmcor:irred}
  The $\bk[\mr{GL}(V)]$-module $\bk \otimes_\bZ \mr{St}(V)$ is indecomposable, for any connected commutative ring $\bk$.
\end{MainCor}
Recall a module over a ring is irreducible if it contains no non-zero proper submodules, and indecomposable if it contains no non-zero proper summands. Since this paper first appeared, Putman and Snowden \cite[Theorem A]{PutmanSnowden} have proved that $\bk \otimes_\bZ \mr{St}(V)$ is irreducible for any field $\bk$.

\subsection{Applications to the homology of general linear groups}
\label{sec:homol-gener-line}

Theorems \ref{thm:A} and \ref{thm:B}, concerning $E_\infty$-homology, have many consequences for the unstable homology of general linear groups, especially when combined with explicit information about such homology groups in low degree and low rank. We will prove a selection of such results, chosen to display different kinds of behaviour and whose proofs exhibit various techniques for working with cellular $E_\infty$-algebras.

Firstly, in Section \ref{sec:NS} we will show how to use our methods to recover a result of Suslin \cite[Theorem 3.4]{SuslinCharClass}, Nesterenko and Suslin \cite[Theorems 2.7, 3.25]{SN} and Guin \cite[Th{\'e}or{\`e}me 2]{Guin}, asserting that for a connected semi-local ring $\R$ with infinite residue fields we have
\[H_{*}(\mr{GL}_n(\R), \mr{GL}_{n-1}(\R);\bZ) = 0 \qquad \text{for $*<n$},\] 
as well as an isomorphism
\[H_{n}(\mr{GL}_n(\R), \mr{GL}_{n-1}(\R);\bZ) \cong K_n^M(\R)\]
between relative group homology and Milnor $K$-theory, which we recall is the graded ring generated by $K_1^M(\R) = \R^\times$ subject to the relations $a \cdot b = 0 \in K_2^M(\R)$ when $a,b\in\R^\times$ satisfy $a + b = 1$. This can be thought of as an analogue of a ``secondary homological stability'' result, saying the relative homology of the stabilisation map not only vanishes in bidegrees $(n,d)$ with $d<n$ but also for $d=n$ is presented in rank $\leq 2$.

We will then extend Nesterenko and Suslin's theorem in the following way. In Section \ref{sec:NSplus1} we explain how
\[\bigoplus_{n \geq 1} H_{n+1}(\mr{GL}_n(\R),\mr{GL}_{n-1}(\R);\bQ)\]
may be made into a module over $K_*^M(\R)_\bQ \coloneqq \bQ \otimes_\bZ K_*^M(\R)$, and we then show how to generate this module efficiently. Our answer is expressed in terms of the third Harrison homology (see \cite{HarrisonCoh}) of the graded-commutative ring $K_*^M(\R)_\bQ$.

\begin{MainThm}\label{thm:Harrison}
For any connected semi-local ring $\R$ with all residue fields infinite, there is a natural homomorphism of graded $\bQ$-vector spaces
\begin{equation*}
\bigoplus_{n \geq 0} \mr{Harr}_{3}(K_*^M(\R)_\bQ)_n \lra \bQ \otimes_{K_*^M(\R)_\bQ} \bigoplus_{n \geq 0}  H_{n+1}(\mr{GL}_n(A),\mr{GL}_{n-1};\bQ),
\end{equation*}
which is an isomorphism for $n \geq 5$. If $\R$ is an infinite field then this map is an isomorphism for $n \geq 4$.
\end{MainThm}

\sloppy For readers unfamiliar with Harrison homology let us mention that it is a summand of Hochschild homology \cite[Theorem 1.1]{Barr}. Expressing the latter as Tor-groups, our result implies the degree $n$ part of the graded vector space $\mr{Tor}_3^{K_*^M(\R)_\bQ}(\bQ,\bQ)$ surjects onto the $K_*^M(\R)_\bQ$-module indecomposables of $\bigoplus_{n \geq 1} H_{n+1}(\mr{GL}_n(\R),\mr{GL}_{n-1}(\R);\bQ)$.

Under various concrete assumptions on the ring $\R$ our methods can give stronger results about the relative homology $H_*(\mr{GL}_n(\R),\mr{GL}_{n-1}(\R))$, going far beyond the Nesterenko--Suslin line $* = n$.  The following are three examples of such results:

\begin{MainThm}\label{thm:stab-special-assumptions}
  \mbox{}
  \begin{enumerate}[(i)]
    \item If $\R$ is a connected semi-local ring with all residue fields infinite and such that $K_2(\R)_\bQ=0$ then
    \[H_d(\mr{GL}_n(\R),\mr{GL}_{n-1}(\R);\bQ)=0\]
    in degrees $d < \tfrac{4n-1}{3}$.
		
		 \item \label{enum:p-div} If $p$ is a prime number and $\R$ is a connected semi-local ring with all residue fields infinite and such that $\R^\times \otimes \bZ/p=0$ then
    \[H_d(\mr{GL}_n(\R),\mr{GL}_{n-1}(\R);\bZ/p) = 0\]
    in degrees $d < \frac{3}{2}n$.
		
    \item \label{enum:alg-closed} If $\bF$ is an algebraically closed field then
    \[H_{d}(\mr{GL}_n(\bF), \mr{GL}_{n-1}(\bF);\bZ/p)=0\]
    in degrees $d < \tfrac{5}{3}n$, for all primes $p$.
  \end{enumerate}
\end{MainThm}

In particular, part \eqref{enum:alg-closed} of this theorem implies that for algebraically closed fields $H_{n+1}(\mr{GL}_n(\bF), \mr{GL}_{n-1}(\bF);\bZ/p)=0$ for all $n \geq 2$ and all primes $p$. This relates to the ``higher pre-Bloch groups'' suggested by Loday \cite[Section 4.4]{lodaycomp} and denoted $\mathfrak{p}_n(\bF)$ by Mirzaii \cite{mirzaiiclosed}. These groups are expected to be related to $H_{n+1}(\mr{GL}_n(\bF))$ in a way that specialises to the relationship between $H_3(\mr{GL}_2(\bF))$ and the pre-Bloch group $\mathfrak{p}(\bF) = \mathfrak{p}_2(\bF)$.  As we explain in Section \ref{sec:MirzYagConj}, our results imply
\[\mathfrak{p}_n(\bF) \otimes \bZ/p = \begin{cases}
\bZ/p & \text{$n$ odd,}\\
0 & \text{$n$ even,}
\end{cases}\]
which resolves \cite[Conjecture 3.5]{mirzaiiclosed}.  We also resolve a closely related conjecture made earlier by \cite[Conjecture 0.2]{yagunov}.

Finally, in Section~\ref{sec:weak-inject-conj} we prove the following result, which implies a new case of Suslin's ``injectivity conjecture.''
\begin{MainThm}\label{mainthm:WeakInj}
  If $\bF$ is an infinite field and $\bk$ is a field in which $(n-1)!$ is invertible then the stabilisation map
  \[H_n(\mr{GL}_{n-1}(\bF);\bk) \lra H_n(\mr{GL}_n(\bF);\bk)\]
  is injective.
\end{MainThm}

Suslin asked more generally whether the stabilisation map $H_i(\mr{GL}_{n-1}(\bF);\bQ) \to H_i(\mr{GL}_n(\bF);\bQ)$ might be injective for all infinite fields and all $i$ \cite[Problem 4.13]{Sah}, \cite[Remark 7.7]{borelyang}, \cite[Conjecture 2]{DeJeu}, \cite[Conjecture 1]{MirzaiiInj}. Our argument completes an approach of Mirzaii \cite{MirzaiiInj}.

In a different direction, in Section \ref{sec:steinberg-vanishing} we apply our methods to analyse the homology of the Steinberg module, in particular showing it vanishes in low degrees:
\begin{MainThm}\label{thm:SteinbergHomology}
If $\R$ is a connected semi-local ring with infinite residue fields, then $H_{d}(\mr{GL}_n(\R) ; \mr{St}(\R^n)) = 0$ for $d<\tfrac{1}{2}(n-1)$.
\end{MainThm}

Analogous results for fields have been obtained by Ash--Putman--Sam \cite[Theorem 1.1]{ashputmansam} and Miller--Nagpal--Patzt \cite[Theorem 7.1]{millernagpalpatzt}.

\subsection*{Acknowledgements}

We thank John Rognes for comments on an earlier version.

AK and SG were supported by the European Research Council (ERC) under the European Union's Horizon 2020 research and innovation programme (grant agreement No.\ 682922).  SG was also supported by the EliteForsk prize and by the Danish National Research Foundation through the Copenhagen Centre for Geometry and Topology (DNRF151).
Part of this work was carried out while SG held visiting positions at Columbia University, and he thanks the department for its hospitality.
AK was also supported by the Danish National Research Foundation through the Centre for Symmetry and Deformation (DNRF92), by NSF grant DMS-1803766, the Natural Sciences and Engineering Research Council of Canada (NSERC) [funding reference number 512156 and 512250], as well as the Research Competitiveness Fund of the University of Toronto at Scarborough. ORW was partially supported by EPSRC grant EP/M027783/1, by the ERC under the European Union's Horizon 2020 research and innovation programme (grant agreement No.\ 756444), and by a Philip Leverhulme Prize from the Leverhulme Trust.


\section{The Steinberg module and its tensor square} \label{sec:steinberg}

We will first prove Theorem~\ref{thm:C} and some of its consequences, as the methods here are more elementary and do not use $E_\infty$-algebras. We begin by recalling some definitions.

\begin{definition}\label{defn:TitsBldg}
  The \emph{Tits building} $\cT(V)$ is the set of proper nontrivial subspaces $W \subset V$, partially ordered by inclusion.  We let $T_\bullet(V) = N_\bullet(\cT(V),\subset)$ and write $T(V) = |T_\bullet(V)|$ for its thin geometric realisation.
\end{definition}

The following is proven in \cite{solomon} for finite fields, and in \cite[Theorem 2.2]{Garland} for all fields. See \cite[Corollary 1]{KahnSun} for a proof using techniques similar to those in Section \ref{sec:local-semi-local}.

\begin{theorem}[Solomon--Tits]\label{thm:SolomonTits}
  The space $T(V)$ has the homotopy type of a wedge of $(\dim(V)-2)$-spheres.  The \emph{Steinberg module} is the $\bZ[\mr{GL}(V)]$-module
  \begin{equation*}
    \mr{St}(V) \coloneqq \widetilde{H}_{\dim(V)-2}(T(V);\bZ),
  \end{equation*}
  which restricts to a free module of rank 1 over the subring $\bZ[U]$, where $U \subset \mr{GL}(V)$ denotes the subgroup of upper unitriangular matrices with respect to a basis of $V$.\qed
\end{theorem}

\begin{example}As a special case, $T(\bF^1) = \varnothing$ and $\mr{St}(\bF^1) = \bZ$ with trivial action of $\mr{GL}(\bF^1)$. Similarly, $T(\bF^2)$ is the set of lines in $\bF^2$ and $\mr{St}(\bF^2)$ is the kernel of the augmentation $\bZ\{T(\bF^2)\} \twoheadrightarrow \bZ$ with $\mr{GL}(\bF^2)$ acting by permutation on the set of lines.\end{example}

Let us describe some particular classes in the Steinberg module. For an ordered set $\gL = (L_1, L_2 , \ldots, L_n)$ of 1-dimensional subspaces giving a direct sum decomposition of $V$, there is a map $f_\gL \colon \mr{sd}(\partial \Delta^{n-1}) \to {T}(V)$ given by sending a flag of nonempty proper subsets of $\{1,2,\ldots,n\}$ to the corresponding flag of nonzero proper subspaces of $V$. The image $a_\gL$ of the fundamental class of the barycentric subdivision $\mr{sd}(\partial \Delta^{n-1})$ of the boundary of an $(n-1)$-simplex
under $(f_\gL)_*$ is called the \emph{apartment} associated to $\gL$, and is an element of $\mr{St}(V)$. It is well-known (e.g.~\cite[Theorem 8.5.2]{borelserre}) that the apartments span $\mr{St}(V)$.

\subsection{Pairings on Steinberg modules}
\label{sec:coinvariants-of}

Let us choose bases and identify $V \cong \bF^n$. In the introduction we have described a positive definite symmetric bilinear form
\[\langle -, -\rangle \colon \mr{St}(\bF^n) \otimes \mr{St}(\bF^n) \lra \bZ\]
which is $\mr{GL}_n(\bF)$-invariant, and therefore induces a map
\[(\mr{St}(\bF^n) \otimes \mr{St}(\bF^n))_{\mr{GL}_n(\bF)} \lra \bZ\]
on coinvariants. We must show that this is an isomorphism as long as $n \geq 1$. Let us assume for the moment that the coinvariants are cyclic, and explain how to deduce the rest of the theorem.

\begin{proof}[Proof of Theorem \ref{thm:C} assuming cyclicity]
It remains to show that $\langle -, -\rangle \colon \mr{St}(\bF^n) \otimes \mr{St}(\bF^n) \to \bZ$ is surjective for $n \geq 1$, which we do using the theory of apartments as described above. Consider the decomposition $\gL = (L_1, L_2, \ldots, L_n)$ with $L_i = \mr{span}(e_i)$, and the decomposition $\gL' = (L_1', L_2', \ldots, L_n')$ with $L'_1 = L_1$ and $L'_i = \mr{span}(e_i + e_{i-1})$ for $i=2,3,\ldots,  n$. The corresponding maps $f_\gL, f_{\gL'} \colon \mr{sd}(\partial \Delta^{n-1}) \to T(\bF^n)$ defining the apartments $a_\gL$ and $a_{\gL'}$ share precisely one $(n-2)$-simplex and hence we have $\langle a_\gL, a_{\gL'}\rangle=1$.
\end{proof}

To prove that the coinvariants are indeed cyclic we shall use
the 
Lee--Szczarba presentation (\cite[Section 3]{leeszczarba}, see also Section~\ref{sec:local-semi-local} below):
\begin{equation*}
\bZ X_1 \overset{\partial}\lra \bZ X_0 \lra \mr{St}(\bF^n) \lra 0.
\end{equation*}
Here $X_0$ is the set of bases for $\bF^n$, where we shall identify $\phi = (\phi_1, \ldots, \phi_n)$ with the matrix $\phi \in \mr{GL}_n(\bF)$ whose $i$th column is $\phi_i$. Similarly, $X_1$ is the set of spanning $(n+1)$-tuples of non-zero vectors in $\bF^n$. Thus this presentation says that $\mr{St}(\bF^n)$ is generated over $\bZ$ by symbols $[\phi]$ with $\phi \in X_0$ subject to certain relations arising from $X_1$. The generator $[\phi]$ corresponds to the apartment class from the direct sum decomposition of $\bF^n$ into the spans of the columns of the matrix $\phi$.

In particular, for any $\phi = (\phi_1, \ldots, \phi_n) \in X_0$ we have
elements of $X_1$ given by
\begin{align*}
&(\phi_1, \ldots, \phi_i,\phi_{i+1}, \phi_i, \phi_{i+2}, \ldots,
\phi_n)\\
& (\phi_1, \ldots, \phi_i, \alpha \phi_i, \phi_{i+1}, \ldots, \phi_n)\\
&(\phi_1, \ldots, \phi_i, \phi_i + \phi_{i+1}, \phi_{i+1},\ldots, \phi_n).
\end{align*}
These give the following three relations:
\begin{enumerate}[(I)]
	\item Column permutations: the first of these three elements of $X_1$ gives rise to the relation that we may swap adjacent columns of $\phi$ if we simultaneously change the sign of the generator $[\phi]$.  By induction we get the relation
\begin{equation*}
[\phi_{\sigma(1)}, \ldots, \phi_{\sigma(n)}] =
\mathrm{sign}(\sigma)
[\phi_1, \ldots, \phi_n]
\end{equation*}
for any permutation $\sigma$.  
	\item Column scalings: the second element of $X_1$, in which
$\alpha \in \bF^\times$, allows us to multiply any column of the matrix
$\phi$ by a non-zero element.  
	\item Column addition: the third element of $X_1$ gives the relation
\begin{align*}
[\phi_1, \ldots, \phi_n] = {} & {} [\phi_1, \ldots, \phi_{i-1},
\phi_i+\phi_{i+1} ,\phi_{i+1}, \phi_{i+2}, \ldots, \phi_n]\\
&+ [\phi_1, \ldots,
\phi_{i-1}, \phi_i, \phi_i+\phi_{i+1},
\phi_{i+2}, \ldots, \phi_n].
\end{align*}
\end{enumerate}

Combining these, we see that the symbol $[\phi] \in \mr{St}(\bF^n)$
is subject to the usual column operations from linear algebra applied to
the matrix $\phi \in \mr{GL}_n(\bF)$, except that ``column addition'' between
the $i$th and $j$th column is symmetric in $i$ and $j$: we must add the
$i$th column to the $j$th and \emph{simultaneously} the $j$th to the
$i$th, and then take the formal sum of the two
resulting matrices.

\begin{proof}[Proof of cyclicity]
By tensoring together two copies of the Lee--Szczarba resolution, we obtain an exact sequence
	\begin{equation*}
	\bZ[(X_1 \times X_0) \amalg (X_0 \times X_1)] \xrightarrow{\partial
		\otimes 1 + 1 \otimes \partial} \bZ[X_0 \times X_0]
	\to \mr{St}(\bF^n) \otimes \mr{St}(\bF^n) \to 0
	\end{equation*}
	of left $\bZ[\mr{GL}_n(\bF)]$-modules.  Hence
	$\mr{St}(\bF^n) \otimes \mr{St}(\bF^n)$ is generated over $\bZ$ by
	symbols $[\phi] \otimes [\psi]$ with $\phi,\psi \in \mr{GL}_n(\bF)$, each
	subject to the ``column operation'' relations above.  In the
	coinvariants we additionally have the relation
	\begin{equation*}
	[\phi] \otimes [\psi] = [1] \otimes [\phi^{-1} \psi],
	\end{equation*}
	using $[1]$ as shorthand for $[\mr{id}_n]$, with $\mr{id}_n$ the identity $(n \times n)$-matrix.
	
	We deduce from this that
	$(\mr{St}(\bF^n) \otimes \mr{St}(\bF^n))_{\mr{GL}_n(\bF)}$ is generated by
	symbols $[1] \otimes [\phi]$ for matrices $\phi \in \mr{GL}_n(\bF)$,
	subject to the following two types of relations.  The first is the
	column operations described above.  The second is that we allow
	\emph{row} operations in a similar fashion (where ``row
	addition'' is symmetrised in the same way as before, resulting in
	the formal sum of the results of adding the $i$th row to the
	$j$th and the $j$th to the $i$th).  Actually, column addition in $\phi$ becomes row \emph{subtraction} in $\phi^{-1}\psi$, but row addition may be achieved by combining row subtraction with scaling of rows by $-1$.  The proof of cyclicity of the
	coinvariants is finished in the following two steps.
		
\vspace{1ex}
	
	\noindent\textbf{Claim 1.} $(\mr{St}(\bF^n) \otimes \mr{St}(\bF^n))_{\mr{GL}_n(\bF)}$
	is generated as an abelian group by symbols $[1] \otimes [\phi]$
	where $\phi \in \mr{GL}_n(\bF)$ is a matrix with entries $\phi_{i,i} = 1$
	for all $i = 1, \ldots, n$, $\phi_{i,i+1} \in \{0,1\}$ for all
	$i = 1, \ldots, n-1$, and all other entries zero. That is, it is a matrix in
	Jordan form with 1's on the diagonal, such as
	\[\begin{bmatrix} 1 & 1 & 0 & 0 \\
	0 & 1 & 0 & 0 \\
	0 & 0 & 1 & 1 \\
	0 & 0 & 0 & 1 \end{bmatrix}.\]
	
	\vspace{1ex}

	\noindent\textbf{Claim 2.} If $\phi \in \mr{GL}_n(\bF)$ is a matrix in Jordan form with 1's on the diagonal, then the class $[1] \otimes [\phi] \in (\mr{St}(\bF^n) \otimes \mr{St}(\bF^n))_{\mr{GL}_n(\bF)}$ is an integer multiple of $[1] \otimes [J_n]$,
	where $J_n$ is the matrix with entries $\phi_{i,i} = 1$ for all
	$i = 1, \ldots, n$, $\phi_{i,i+1} = 1$ for all $i = 1, \ldots, n-1$,
	and all other entries zero. That is, $J_n$ is a single Jordan block of
	size $n \times n$, such as
	\[J_4 = \begin{bmatrix} 1 & 1 & 0 & 0 \\
	0 & 1 & 1 & 0 \\
	0 & 0 & 1 & 1 \\
	0 & 0 & 0 & 1 \end{bmatrix}.\]
	
	\vspace{1ex}
	
	These claims together imply cyclicity.  They are proved separately below.
\end{proof}

\begin{proof}[Proof of Claim 1]
	We first explain how rewrite an arbitrary generator
	$[1] \otimes [\phi]$ as an integral linear combination of generators
	in which the last row of $\phi$ is of the form $(0,\ldots, 0,1)$ and
	the last column is of the form $(0,\ldots, 0,\epsilon,1)$ with
	$\epsilon \in \{0,1\}$, such as
	\[\begin{bmatrix} * & * & * & 0 \\
	* & * & * & 0 \\
	* & * & * & \epsilon \\
	0 & 0 & 0 & 1 \end{bmatrix},\]
using all available column operations but only row operations which do not involve the last row.	

If the last row has a single non-zero entry, we may scale its column and then permute columns so that the last row is of the desired form. Otherwise we may use column scalings
	to arrange that the last row has one entry 1 and one entry $-1$, so that the
	sum of the corresponding columns has last entry 0.  Using column
	addition between these two columns we obtain a relation
	$[1] \otimes [\phi] = [1] \otimes [\phi'] + [1] \otimes [\phi'']$
	where $\phi'$ and $\phi''$ have strictly more zeros in the last row
	than $\phi$ did.  By induction we obtain a relation
	$[1] \otimes [\phi] = \sum_i [1] \otimes [\phi_i]$ where the last
	row of each $\phi_i$ has precisely one non-zero entry. Proceeding as in the first sentence of this paragraph 
	we see that the coinvariants are
	generated by those $[1] \otimes [\phi]$ where the last row of $\phi$
	is $(0,\ldots, 0,1)$.  
	
	If $\phi$ is of this form we may apply the
	same argument with the roles of columns and rows swapped: by
	induction on the number of zeros in the last column we may use
	\emph{row} operations among the first $(n-1)$ rows to decrease the
	number of non-zero entries in the last column.  Without using the
	last row we achieve that no more than two entries are non-zero,
	i.e.\ we have a relation
	$[1] \otimes [\phi] = \sum_i [1] \otimes [\phi_i]$ where the last
	column of each $\phi_i$ has at most two non-zero entries.  One of
	these must be the last entry, so (after scaling and swapping among
	the first $(n-1)$ rows) the last column of each $\phi_i$ is either
	$(0,\ldots, 0,1)$ or (after scaling rows) of the form
	$(0,\ldots 0, 1,1)$.
	
	This argument takes care of the last row and last column, using all
	column operations but only row operations not involving the last
	row.  Hence it may be applied to the upper left $(n-1) \times (n-1)$
	block of $\phi \in \mr{GL}_n(\bF)$ without making modifications to the last row and
	column.  By induction we rewrite an arbitrary generator $[1] \otimes
	[\phi]$ as a linear combination of generators where $\phi$ is in the
	desired Jordan block form.
\end{proof}

\begin{proof}[Proof of Claim 2]
	If we write $J_{a,b} = \mathrm{diag}(J_a,J_b)$ with $a+b = n$, it
	suffices to explain how to use the relations to rewrite
	$[1] \otimes [J_{a,b}]$ as a multiple of $[1] \otimes [J_{a+b}]$.
	In the case where $\phi$ has more than two Jordan blocks we just
	repeatedly combine two of them into one.
	
	To this end, let us write $J_{a,b}(i)$ for the matrix obtained from
	$J_{a,b}$ by adding a 1 in the $(a+1)$st column and $i$th row, for
	$i = 1,\ldots, a$, such as
	\[J_{2,2} = \begin{bmatrix} 1 & 1 & 0 & 0 \\
	0 & 1 & 0 & 0 \\
	0 & 0 & 1 & 1 \\
	0 & 0 & 0 & 1 \end{bmatrix},\]
	and 
	\[J_{2,2}(1) = \begin{bmatrix} 1 & 1 & 1 & 0 \\
	0 & 1 & 0 & 0 \\
	0 & 0 & 1 & 1 \\
	0 & 0 & 0 & 1 \end{bmatrix}, \qquad J_{2,2}(2) = \begin{bmatrix} 1 & 1 & 0 & 0 \\
	0 & 1 & 1 & 0 \\
	0 & 0 & 1 & 1 \\
	0 & 0 & 0 & 1 \end{bmatrix} = J_4.\]
	Then clearly $J_{a,b}(a) = J_{a+b}$ and we shall
	also write $J_{a,b}(0) = J_{a,b}$.  
	
	We claim the relation
	\begin{equation}
	\label{eq:8}
	[1] \otimes [J_{a,b}(i)] = [1] \otimes [J_{a,b}(i+1)] +  [1]
	\otimes [J_{b+i,a-i}(i+1)]
	\end{equation}
	holds for all $i = 0, \ldots, a-1$.  To see this, first scale rows
	$i+1, \ldots ,a$ and columns $i+1,\ldots ,a$ of $J_{a,b}(i)$ by $-1$, with
	the effect of changing the sign of the entry in the $i$th row and
	$(i+1)$st column.  Then perform a column addition operation between
	the $(i+1)$st and the $(a+1)$st columns.  The result is two terms, the
	first of which becomes $J_{a,b}(i+1)$ after scaling rows
	$1,\ldots, i$ and columns $1,\ldots, i$ by $-1$, the second of which
	becomes $J_{b+i,a-i}(i+1)$ after conjugating by the matrix
	\begin{equation*}
	\begin{bmatrix}
	 \mr{id}_i & 0 & 0\\
	0 & 0 &  \mr{id}_{a-i}\\
	0 &  \mr{id}_b & 0
	\end{bmatrix},
	\end{equation*}
	where $\mr{id}_k$ denotes the identity $(k \times k)$-matrix. 
	Since this is a permutation matrix the conjugation may be achieved
	by permuting rows and columns.
	
	We leave it to the reader to verify
	that these arguments apply also in the case $i=0$ (for example by
	taking the above proof for $J_{a+1,b}(1)$ and deleting the first row
	and column, which are not moved during the proof).  We have finished
	the proof of the relation~(\ref{eq:8}), which by induction implies
	that $[1] \otimes [J_{a,b}]$ is a multiple of
	$[1] \otimes [J_{a+b}]$.  (In fact the multiple can easily be seen
	to be the binomial coefficient $\binom{a+b}a$.)
\end{proof}

\subsection{Anisotropy and indecomposability}
\label{sec:indec-steinb-modul}

For any commutative ring $\bk$ we have an induced $\bk[\mr{GL}(V)]$-module $\mr{St}_\bk(V) = \bk \otimes_\bZ \mr{St}(V)$, and we extend the pairing above to a $\bk$-bilinear pairing
\[\langle-, -\rangle_\bk\colon \mr{St}_\bk(V) \times \mr{St}_\bk(V) \lra \bk.\]

Let us record some properties of this pairing.
\begin{theorem}\label{thmmaincor:B}\mbox{}
  \begin{enumerate}[(i)]
  \item\label{it:Bi} The pairing induces an isomorphism $(\mr{St}_\bk(V) \otimes_\bk \mr{St}_\bk(V))_{\mr{GL}(V)} \overset{\sim}\to \bk$.
  \item\label{it:Bii} The set of $\mr{GL}(V)$-invariant $\bk$-bilinear pairings $b \colon \mr{St}_\bk(V) \times \mr{St}_\bk(V) \to \bk$ forms a free $\bk$-module of rank 1, with the pairing $\langle-,-\rangle_\bk$ as basis.  In particular all invariant forms are symmetric.
  \item\label{it:Biii} If $\bk$ is an ordered field the pairing is positive definite.  More generally, if $\bk$ is a ring in which 0 is not a non-trivial sum of squares, then the pairing is anisotropic.  In either case, the restriction to any $\bk$-linear submodule $A \subset \mr{St}_\bk(V)$ also induces an injection $A \to A^\vee$.
  \item\label{it:adjInj} If $\bF$ is infinite then the adjoint $\mr{St}_\bk(V) \to \mr{St}_\bk(V)^\vee = \Hom_\bk(\mr{St}_\bk(V),\bk)$ is injective.
  \end{enumerate}
\end{theorem}
\begin{remark}
In (\ref{it:adjInj}) the assumption that $\bF$ is infinite cannot be removed: one may verify by hand that the conclusion is false for $\bF=\bF_2$, $V=\bF_2^2$, and $\bk=\bZ/3$.
\end{remark}
 
Recall that a symmetric $\bk$-bilinear pairing $b\colon M \times M \to \bk$ on a $\bk$-module $M$ is called non-degenerate if the adjoint $M \to M^\vee = \Hom_\bk(M,\bk)$ is injective.  It is called \emph{anisotropic} if $b(x,x) \neq 0$ for any $x \in M\setminus \{0\}$.  Anisotropy implies non-degeneracy, but the stronger notion has the advantage of passing to submodules: in fact, anisotropy is equivalent to the restrictions $M' \times M' \to \bk$ being non-degenerate for all submodules $M' \subset M$.  If $\bk$ is an ordered field, then any positive definite form is anisotropic (but anisotropy is more general, e.g., the quadratic form $x^2 - 2 y^2$ over the ordered field $\bQ$ is anisotropic but not positive definite).

\begin{proof}[Proof of Theorem~\ref{thmmaincor:B}]
  It is clear that
  \begin{equation*}
    (\mr{St}_\bk(V) \otimes_\bk \mr{St}_\bk(V))_{\mr{GL}(V)} = \bk \otimes_\bZ (\mr{St}(V) \otimes \mr{St}(V))_{\mr{GL}(V)},
  \end{equation*}
  so (\ref{it:Bi}) and (\ref{it:Bii}) follow from Theorem~\ref{thm:C}.

  For (\ref{it:Biii}), recall that $\mr{St}(V) = \mr{ker}[\partial \colon \widetilde{C}_{\dim(V) - 2}(T(V)) \to \widetilde{C}_{\dim(V) - 3}(T(V))]$ and that the pairing on $\mr{St}_\bk(V)$ is the restriction of the chain level pairing in which the set of full flags forms an orthonormal basis.  If $x = \sum_F a_F \cdot F$ is a finite sum of full flags $F$ with coefficients $a_F \in \bk$, then $\langle x, x\rangle_\bk = \sum_F a_F^2$.  If $x \neq 0$ then this self-pairing is positive if $\bk$ is an ordered field, and non-zero if no non-trivial sum of squares is zero in $\bk$.

  For (\ref{it:adjInj}) we first establish the following claim. Recall that top-dimensional simplices of $T(V)$ correspond to full flags in $V$; we write $F$ for a full flag $0 \subset F_1 \subset F_2 \subset \cdots \subset F_n = V$ where $\dim(F_i) = i$. The apartment corresponding to a splitting of $V$ into 1-dimensional subspaces $L_1, L_2, \ldots, L_n$ consists of those full flags which may be
    obtained as partial sums of the $L_i$ in some order.

\vspace{1ex}

\noindent \textbf{Claim.} Assume $\bF$ is infinite. For a full flag $F$ and a finite set of full flags $\{F^\alpha\}_{\alpha \in I}$ distinct from $F$, there is an apartment of $T(V)$ having $F$ as a face and having no $F^\alpha$ as a face.
\begin{proof}[Proof of claim]
For each $r$ consider the set
\[V_r \coloneqq F_r \setminus \left (F_{r-1} \cup \bigcup_{\alpha \in I} F^\alpha_{r-1}\right).\]
This is the complement in $F_r$ of finitely-many proper subspaces, so is non-empty under our assumption that $\bF$ is infinite.

Choose elements $v_r \in V_r$, and let $L_r \coloneqq \mr{span}(v_r)$. As $v_r \nin F_{r-1}$ we have $F_r = L_1 \oplus \cdots \oplus L_r$, so the apartment given by this splitting has the flag $F$ as a face. If this apartment contained $F^\alpha$ as a face then, for some permutation $\sigma$, we would have $F^\alpha_r = L_{\sigma(1)} \oplus L_{\sigma(2)} \oplus \cdots \oplus L_{\sigma(r)}$ for each $r$. Thus $v_{\sigma(r)} \in F_r^\alpha$ for each $r$, so by definition of the sets $V_r$ we have $r \geq \sigma(r)$ for each $r$, but then $\sigma$ must be the identity permutation and so $F^\alpha=F$, a contradiction.
\end{proof}

To finish the proof of (\ref{it:adjInj}) let $x \in \mr{St}_\bk(V)$ be nonzero: it is then a finite non-trivial $\bk$-linear sum of full flags, and we let $F$ be a full flag having non-zero coefficient $k \in \bk$ and $\{F^\alpha\}_{\alpha \in I}$ be the remaining full flags arising in this sum. The claim provides an apartment $a$ containing $F$ but not containing any $F^\alpha$, but then the definition of the pairing shows that $\langle x, a \rangle = k \neq 0$ so $x$ is not in the kernel of the adjoint map.
\end{proof}

Let us explain how to use the bilinear pairing to prove that the $\bk[\mr{GL}(V)]$-module $\mr{St}_\bk(V)$ is indecomposable for any connected commutative ring $\bk$ and any finite-dimensional vector space $V$ over a field $\bF$.    The irreducibility/indecomposability question for Steinberg modules of infinite fields was also asked by A.~Putman \cite{Putman-MO}.  The connection between Theorem~\ref{thm:C} and irreducibility was pointed out to us by A.~Venkatesh. 

\begin{proof}[Proof of Corollary~\ref{thmcor:irred}]
  The $\bk$-linear map $\mr{St}_\bk(V) \to \mr{St}_\bk(V)^\vee = \Hom_\bk(\mr{St}_\bk(V),\bk)$ adjoint to the bilinear form is injective by Theorem \ref{thmmaincor:B} (ii).  Applying the left exact functor $\Hom_\bk(\mr{St}_\bk(V),-)$ gives an injective homomorphism
  \begin{align*}
    \Hom_\bk(\mr{St}_\bk(V),\mr{St}_\bk(V)) \hookrightarrow {} &{} \Hom_\bk(\mr{St}_\bk(V),\mr{St}_\bk(V)^\vee)\\
    & \cong  \Hom_\bk(\mr{St}_\bk(V) \otimes_\bk \mr{St}_\bk(V),\bk),
  \end{align*}
  sending an endomorphism $A$ to the bilinear form $x \otimes y \mapsto \langle Ax,y \rangle$.  This injective homomorphism is $\mr{GL}(V)$-equivariant when that group acts by conjugation in the domain and by dualising the diagonal action on $\mr{St}_\bk(V) \otimes_\bk \mr{St}_\bk(V)$ in the codomain.  Passing to fixed points we get an injective $\bk$-linear homomorphism
  \begin{equation*}
    \End_{\bk}(\mr{St}_\bk(V))^{\mr{GL}(V)} \to \Hom_\bk(\mr{St}_\bk(V) \otimes_\bk \mr{St}_\bk(V),\bk)^{\mr{GL}(V)}.
  \end{equation*}
  We have seen that the codomain is a free $\bk$-module of rank one, and since $1 \in \End_{\bk[\mr{GL}(V)]}(\mr{St}_\bk(V))$ is sent to $\langle-,-\rangle_\bk$, it follows that the identity gives an isomorphism
  \begin{equation*}
    \bk \to \End_\bk(\mr{St}_\bk(V))^{\mr{GL}(V)}.
  \end{equation*}
  In particular if the ring $\bk$ is connected (i.e.\ has no non-trivial idempotents) then there are no non-trivial $\mr{GL}(V)$-equivariant idempotent endomorphisms of $\mr{St}_\bk(V)$.
\end{proof}

When $\bk$ is a field in which no non-trivial sum of squares is zero, then $\mr{St}_\bk(V)$ cannot contain any submodules which are finite-dimensional over $\bk$.  Indeed, suppose for contradiction that $A \subset \mr{St}_\bk(V)$ were such a submodule and consider the composition
\begin{equation*}
  A \hookrightarrow \mr{St}_\bk(V) \to \mr{St}_\bk(V)^\vee \to A^\vee.
\end{equation*}
It is injective because the pairing is anisotropic, but $A$ and $A^\vee$ are vector spaces of the same finite dimension so it is an isomorphism, so $A$ is a $\bk[\mr{GL}(V)]$-linear summand of $\mr{St}_\bk(V)$.

\begin{remark}
  A similar conclusion holds when $\bk$ admits an involution whose fixed field $\bk^+ \subset \bk$ satisfies that no non-trivial sum of squares is zero, e.g.,\ for $\bk = \bC$.  The proof is similar, except the pairing should be extended to a sesquilinear pairing on $\mr{St}_\bk(V)$.
\end{remark}

\section{Overview of $E_k$-cells and $E_k$-homology}\label{sec:overview-e_k-cells} Let us briefly outline the theory developed \cite{e2cellsIv3}, which will play a role in the rest of the paper. We refer there for further details and for proofs and references, and we shall refer to things labelled X in \cite{e2cellsIv3} as $E_k$.X here.

This theory is designed to analyse the homology of (non-unital) $E_k$-algebras like
\begin{equation}\label{eq:DefR}
\coprod_{n \geq 1} B\mr{GL}_n(\bF).
\end{equation}
In fact, as we are only interested in the $\bk$-homology we may as well take the $\bk$-linear singular simplices on this space. Furthermore, as we wish to distinguish the contributions to homology of the different path components, we shall work with an additional $\bN$-grading, called \emph{rank}: we therefore work in the category $\cat{sMod}_\bk^\bN$ of $\bN$-graded simplicial $\bk$-modules, i.e.\ the category of functors $M \colon \bN \to \cat{sMod}_\bk$. This is given a symmetric monoidal structure by Day convolution, with $p$ simplices in rank $n$ given by
\[(M \otimes N)_p(n) = \bigoplus_{a+b=n} M_p(a) \otimes_\bk N_p(b).\]
The homology of a simplicial $\bk$-module means the homology of the associated chain complex, or equivalently the homotopy groups of the underlying simplicial set. An object $X \in \cat{sMod}_\bk^\bN$ hence has bigraded homology groups via $H_{n,d}(X) \coloneqq H_d(X(n))$.

The little $k$-cubes operad has a space of $r$-ary operations given by $r$-tuples of rectilinear embeddings of a $k$-cube into $k$-cube with disjoint interior, except when $r=0$ in which case it is empty. We may import this into $\cat{sMod}_\bk^\bN$ by first taking its $\bk$-linear singular simplicial set, and then placing the result in rank 0: we denote the result by $\cC_k$. 

A (non-unital) \emph{$E_k$-algebra} in $\cat{sMod}_\bk^\bN$ is then an algebra over the operad $\mathcal{C}_k$; we denote these by bold letters such as $\gR$ and write $\Alg_{E_k}(\smash{\cat{sMod}_\bk^\bN})$ for the category of $E_k$-algebras in $\smash{\cat{sMod}_\bk^\bN}$. We denote by $E_k(-)$ the monad on $\smash{\cat{sMod}_\bk^\bN}$ associated to $\cC_k$, and by $\mathbf{E}_k(X)$ the free (non-unital) $E_k$-algebra on $X$. We will mainly be concerned with considering the $\bk$-linearisation of \eqref{eq:DefR} as an $E_1$-, $E_2$-, or $E_\infty$-algebra in the category $\cat{sMod}_\bk^\bN$, and in particular describing cell structures on it.

To explain what we mean by this, let us denote by $\partial D^{n,d} \in \cat{sMod}_\bk^\bN$ the object given at $n \in \bN$ by $\bk[\partial \Delta^d]$, the $\bk$-linearisation of the simplicial set given by the boundary of the $d$-simplex, and by 0 otherwise; we write $D^{n,d}$ for the analogous construction with $\bk[\Delta^d]$. By adjunction, a morphism $\partial D^{n,d} \to \gR$ in $\cat{sMod}_\bk^\bN$ extends to a morphism $\mathbf{E}_k(\partial D^{n,d}) \to \gR$ of $E_k$-algebras, using which we may form the pushout
\begin{equation*}
\begin{tikzcd}\gE_k(\partial D^{n,d}) \rar \dar & \gR \dar \\
\gE_k(D^{n,d}) \rar & \gR \cup^{E_k} \gD^{n,d}.\end{tikzcd}
\end{equation*}
in the category $\Alg_{E_k}(\cat{sMod}_\bk^\bN)$. This is what it means to attach a cell to $\gR$. A \emph{cellular} $E_k$-algebra is one obtained by iterated cell attachments starting with the zero object, and a \emph{cellular approximation} to $\gR$ is a weak equivalence $\gC \overset{\sim}\to \gR$ from a cellular object. 

In order to control cell structures we will use a homology theory for $E_k$-algebras. The \emph{indecomposables} of $\gR \in \Alg_{E_k}(\cat{sMod}_\bk^\bN)$ are defined by the exact sequence
\[\bigoplus_{n \geq 2} \cC_k(n) \otimes \gR^{\otimes n} \lra \gR \lra Q^{E_k}(\gR) \lra 0\]
in $\cat{sMod}_\bk^\bN$, with the leftmost map given by the $E_k$-algebra structure of $\gR$. That is, we collapse all elements of $\gR$ which can be obtained by applying at least $2$-ary operations. This construction is the left adjoint to the functor $Z^{E_k} \colon \cat{sMod}_\bk^\bN \to \Alg_{E_k}(\cat{sMod}_\bk^\bN)$ which considers any object as an $E_k$-algebra in the trivial way.
The construction $Q^{E_k} \colon \Alg_{E_k}(\cat{sMod}_\bk^\bN) \to \cat{sMod}_\bk^\bN$ is not homotopy invariant, but has a left derived functor $\smash{Q^{E_k}_\bL}$ called the \emph{derived indecomposables}. (In the notation of \cite{e2cellsIv3} we have actually defined the relative indecomposables, but because $\cC_\infty(1) \simeq *$ at the level of derived functors there is no difference between this and the absolute indecomposables, cf.\ equation (11.3) in Section $E_k$.11.3.) We then define the \emph{$E_k$-homology} groups of $\gR$ by the formula 
\[H^{E_k}_{n,d}(\gR) \coloneqq H_{n,d}(Q^{E_k}_\bL(\gR)) = H_d(Q^{E_k}_\bL(\gR)(n)).\]

One does not typically try to compute $E_k$-homology directly from the definition. Instead, one uses a description in terms of a $k$-fold bar construction (instances have been given by Getzler--Jones, Basterra--Mandell, Francis, and Fresse). More precisely, if $\gR^+$ denotes the unitalisation of the non-unital $E_k$-algebra $\gR$, and $\epsilon \colon \gR^+ \to \bk$ its corresponding augmentation, then in Section $E_k$.13.1 we have described an equivalence
\[\bk \oplus \Sigma^k Q^{E_k}_\bL(\gR) \simeq B^{E_k}(\gR^+, \epsilon),\]
where the right-hand side denotes a certain model for the $k$-fold bar construction of the augmented $E_k$-algebra $\epsilon \colon \gR^+ \to \bk$. An advantage of this point of view is that the $k$-fold bar construction can be calculated iteratively, which makes it easy to relate information about $E_{k-1}$-homology to information about $E_k$-homology.

Derived indecomposables are related to cellular approximation as follows. Being a left adjoint, $Q^{E_k}$ preserves pushouts, and by direct calculation $Q^{E_k}(\gE_k(X)) \cong X$. A cellular $E_k$-algebra $\gC$ is cofibrant, so we may consider its underived indecomposables $Q^{E_k}(\gC)$: this is then a cellular object in $\cat{sMod}_\bk^\bN$, with one cell for each $E_k$-cell of $\gC$. The groups $H_{*,*}^{E_k}(\gR)$ therefore give a lower bound on the cells needed for any cellular approximation of $\gR$. In fact there is a Hurewicz theorem for $E_k$-homology (Corollary $E_k$.11.12), which means it can be used to construct minimal $E_k$-cell structures. In particular if $\gR(0) \simeq 0$ and $H_{n,d}^{E_k}(\gR)=0$ for $d < f(n)$, then $\gR$ admits a cellular approximation 
\[\gC \overset{\sim}\lra \gR\]
built only using cells of bidegrees $(n,d)$ such that $d \geq f(n)$. This $\gC$ may be chosen to be CW rather than just cellular, meaning it comes with a lift to an $E_k$-algebra in filtered objects of $\cat{sMod}_\bk^\bN$ with good properties (not a filtered object in $E_k$-algebras), see Section $E_k$.6.3.

An overview of how we shall use these ideas in the rest of the paper is as follows. In Section \ref{sec:contructing-r} we will explain how to construct an $E_\infty$-algebra $\gR$ in $\cat{sSet}^\bN$ with
\[\gR(n) \simeq B\mr{GL}_n(\R),\]
which for any commutative ring $\bk$ we may $\bk$-linearise to give $\gR_\bk \in \Alg_{E_\infty}(\cat{sMod}_\bk^\bN)$ having $H_{n,d}(\gR_\bk) = H_d(\mr{GL}_n(\R);\bk)$. Our first goal is to establish vanishing lines for $H^{E_k}_{n,d}(\gR_\bk)$ for $k\in \{1,2,\infty\}$. To do so, in Section \ref{sec:buildings} we will describe the $k$-fold simplicial sets which arise in the $k$-fold bar construction model for derived $E_k$-indecomposables of this $\gR$, as well as Rognes' $k$-fold analogue of the Tits building, and conditions under which these are related. The main goal is to show that if the data $(\R,\bk)$ satisfies the ``Nesterenko--Suslin property" then Rognes' $k$-fold analogue of the Tits building may be used to calculate $H^{E_k}_{n,d}(\gR_\bk)$.

After this general theory, in Section \ref{sec:GLfield} we will specialise to the case that $\R$ is an infinite field (in which case $(\R, \bk)$ does indeed satisfy the Nesterenko--Suslin property for any coefficients $\bk$). In this case we will simply observe that Rognes' $2$-fold analogue of the Tits building is simply the smash-square of the $1$-fold analogue, so that it is twice as highly-connected as the Tits building. This means that there is a much steeper vanishing line for the $E_2$-homology (and hence $E_\infty$-homology) of $\gR_\bk$ than one can formally deduce from the vanishing line for its $E_1$-homology. Finally, we analyse completely the $E_2$- and $E_\infty$-homology along this vanishing line, which in particular will prove Theorem \ref{thm:B}. In Section \ref{sec:local-semi-local} we will explain the analogues of most of these results in the case that $\R$ is a connected semi-local ring with all residue fields infinite.

We then come to applications. Here we shall use most of the technical tools developed in \cite{e2cellsIv3}, and will not try to summarise them all here. 

\subsection{Glossary and some notation from \cite{e2cellsIv3}}
\label{sec:notat-other-recoll}

For the reader's convenience we collect some notation from op.cit., but we refer there for more details.

\begin{enumerate}[$\bullet$]
	\item Many arguments take place in the categories $\cat{sSet}^\bN$, $\cat{sSet}^\bN_*$, and $\cat{sMod}_\bk^\bN$, of functors from $\bN$ (regarded as a category with only identity morphisms) to the category of simplicial sets, pointed simplicial sets, and simplicial $\bk$-modules, respectively.
	\item Given an algebra $\gR \in \Alg_{E_\infty}(\cat{sSet}^\bN)$ and a commutative ring $\bk$, we denote by $\gR_\bk \in \Alg_{E_\infty}(\cat{sMod}_\bk^\bN)$ the $\bk$-linearisation of $\gR$.
	\item Objects $X \in \cat{sMod}_\bk^\bN$ have homology groups which are bigraded $\bk$-modules, defined by $H_{n,d}(X) = \pi_d(X(n))$.  
	\item The object $D^{n,d} \in \cat{sSet}^\bN$ is given at $n \in \bN$ by $\Delta^d$ and $\varnothing$ otherwise. Similarly, $\partial D^{n,d} \in \cat{sSet}^\bN$ is given at $n \in \bN$ by $\partial \Delta^{d}$ and $\varnothing$ otherwise. Finally, $S^{n,d}  = D^{n,d}/\partial D^{n,d} \in \cat{sSet}_*^\bN$. 
	\item On homology, smashing with $S^{n,d}$ has the effect of shifting bidegrees: for $X \in \cat{sMod}_\bk^\bN$, $H_{n+n',d+d'}(S^{n',d'} \wedge X) = H_{n,d}(X)$ if $n \geq 0$ and $d \geq 0$ and otherwise vanishes.
	\item The unitalisation of an $E_k$-algebra $\gR$ is denoted $\gR^+$, see Section $E_k$.4.4. By ``$E_k$-algebra'' we always mean the non-unital type, unless otherwise specified.  
	\item If $\gR$ is an $E_1$-algebra, then $\overline{\gR}$ denotes a unital and associative algebra naturally weakly equivalent to $\gR^+$ as a unital $E_1$-algebra, see Proposition $E_k$.12.9 and the construction preceding it. If $\gR$ is an $E_k$-algebra, we regard it is an $E_1$-algebra before applying these constructions.
	\item Filtered objects in a category $\cat{C}$ are functors $\bZ_{\leq} \to \cat{C}$, where $\bZ_{\leq}$ denotes the category with object set $\bZ$ and morphism set $m \to n$ either a singleton or empty, depending on whether $m \leq n$ or $m > n$.  We often consider filtered objects in functor categories such as $\cat{C} = \cat{sMod}_\bk^\bN$, in which case filtered objects are functors $\bN \times \bZ_{\leq} \to \cat{sMod}_\bk$.
	\item To an unfiltered object $X$ is associated a filtered object $a_* X$ for each $a \in \bZ$, making the functor $X \mapsto a_* X$ left adjoint to evaluation at the object $a \in \bZ_{\leq}$.  Explicitly, $(a_* X)(n)$ is $X$ for $n \geq a$ and is the initial object for $n < a$.  See Section $E_k.5.3.2$.
	\item There is a spectral sequence associated to a cofibrant filtered object $X \in \cat{sMod}_\bk^{\bN \times \bZ_{\leq}}$.  In our grading conventions, see Theorem $E_k$.10.10, it has
	\begin{align*}
	&E^1_{n,p,q} = H_{n,p+q,p}(\mr{gr}(X)) = H_{n,p+q}(X(p),X(p-1)), \\
	&d^r \colon  E^r_{n,p,q} \to E^r_{n,p-r,q+r-1},
	\end{align*}
	and if it conditionally converges, it does so to $H_{n,p+q}(\mr{colim}\,X)$. This does not relate the different $n$ at all (disregarding any multiplicative structures) so may be regarded as one spectral sequence for each $n \in \bN$.  Then $(E^r_{n,*,*},d^r)_{r \geq 2}$ is graded as the usual homological Serre spectral sequence.
\end{enumerate}

\section{General linear groups as an $E_\infty$-algebra} \label{sec:contructing-r}

\begin{convention}All our rings $\R$ are commutative.\end{convention}

For any ring $\R$ there is a groupoid $\cat{P}_\R$ with objects the finitely-generated projective $\R$-modules, and morphisms given by the $\R$-module isomorphisms. We write $\mr{GL}(M)$ for the group of  automorphisms of an $\R$-module $M$. Direct sum $\oplus$ endows $\cat{P}_\R$ with a symmetric monoidal structure.

\begin{lemma}\label{lem:vect-cat-props} The symmetric monoidal groupoid $(\cat{P}_\R, \oplus, 0)$ has the following properties:
	\begin{enumerate}[(i)]
		\item\label{it:axiom171} for the monoidal unit $0$ we have $\mr{GL}(0) = \{e\}$, 
		\item\label{it:axiom172} the homomorphism $ -\oplus - \colon \mr{GL}(M) \times \mr{GL}(N) \to \mr{GL}(M \oplus N)$ induced by the monoidal structure is injective.
	\end{enumerate}
\end{lemma}

A finitely-generated projective $\R$-module $M$ has a rank $\mr{rk}_\mathfrak{p}(M)$ for each prime ideal $\mathfrak{p} \subset \R$, given by $\mr{rk}_\mathfrak{p}(M) \coloneqq \dim_{\kappa(\mathfrak{p})}(M \otimes_\R \kappa(\mathfrak{p}))$ with $\kappa(\mathfrak{p})$ the field of fractions of the integral domain $\R/\mathfrak{p}$,
and these form a locally constant function on $\mr{Spec}(\R)$ \cite[\href{https://stacks.math.columbia.edu/tag/00NV}{Tag 00NV}]{stacks-project}). We will therefore suppose that $\mr{Spec}(\R)$ is connected (and will usually say that ``$\R$ is connected"), so that there is a unique rank associated to each finitely-generated projective $\R$-module. This gives a functor
\[r \colon \cat{P}_\R \lra \bN.\]

We will use this category to construct a non-unital $E_\infty$-algebra $\gR \in \cat{Alg}_{E_\infty}(\cat{sSet}^\bN)$ with $\gR(0) = \varnothing$ and 
\begin{equation}\label{eq:ValuesOfR}
\gR(n) \simeq \coprod_{\substack{[M]\\r(M)=n}} B\mr{GL}(M),
\end{equation}
where the disjoint union is over isomorphism classes of rank $n$ projective $\R$-modules, following Section $E_k$.17.1.

We now proceed as in Section $E_k$.17.1, using that by Lemma \ref{lem:vect-cat-props} (\ref{it:axiom171}) above the groupoid $\cat{P}_\R$ satisfies Assumption $E_k$.17.1. There is a functor $\underline{\ast}_{>0}$ in $\cat{sSet}^{\cat{P}_\R}$ given by $0 \mapsto \varnothing$ and $M \mapsto \ast$ for $M \neq 0$, which has a unique non-unital commutative algebra structure and hence is in particular a non-unital $E_\infty$-algebra. As described in Section $E_k$.9.2 the category $\cat{Alg}_{E_k}(\cat{sSet}^{\cat{P}_\R})$ of $E_k$-algebras in $\cat{sSet}^{\cat{P}_\R}$ admits the projective model structure. By cofibrantly replacing $\underline{\ast}_{>0}$ we obtain a cofibrant non-unital $E_\infty$-algebra $\gT$ with
\[\gT(P) \simeq \begin{cases}
\varnothing & \text{ if $M=0$,}\\
\ast & \text{ if $M \neq 0$.}
\end{cases}\]

Precomposition by $r$ gives a functor $r^* \colon \cat{sSet}^\bN \to \cat{sSet}^{\cat{P}_\R}$ which admits a left adjoint $r_*$ by left Kan extension, and this is (strong) symmetric monoidal. As in Section $E_k$.4.3 it induces a functor between categories of $E_\infty$-algebras, i.e.\ it may be considered as a functor $r_* \colon \cat{Alg}_{E_\infty}(\cat{sSet}^{\cat{P}_\R}) \to \cat{Alg}_{E_\infty}(\cat{sSet}^\bN)$, which is the left adjoint in a Quillen adjunction. We may thus take its left derived functor $\bL r_*$, and define 
\[\gR \coloneqq r_*(\gT) \simeq \bL r_*(\underline{\ast}_{>0}) \in \cat{Alg}_{E_\infty}(\cat{sSet}^\bN),\]
which satisfies $\gR(0)= \varnothing$ and \eqref{eq:ValuesOfR} above.

As we have mentioned we will often only be interested in homology groups, in which case there is no loss in passing from simplicial sets to simplicial $\bk$-modules and considering instead $\gR_\bk \coloneqq \bk[\gR] \in \cat{Alg}_{E_\infty}(\cat{sMod}_\bk^\bN)$. This satisfies 
\[H_{n,d}(\gR_\bk) = \bigoplus_{\substack{[M]\\r(M)=n}} H_d(\mr{GL}(M);\bk).\]

\begin{remark}In practice we will usually work under ring-theoretic assumptions on $\R$ which imply that every finitely-generated projective $\R$-module is free (namely that $\R$ is semi-local and connected \cite[\href{https://stacks.math.columbia.edu/tag/02M9}{Tag 02M9}]{stacks-project}). In this case we simply have 
\[H_{n,d}(\gR_\bk) = H_d(\mr{GL}_n(\R);\bk).\]
We can mimic this for any ring $\R$ by considering the subcategory of $\cat{P}_\R$ consisting of free modules, which inherits a symmetric monoidality, and which again has a rank functor. Repeating the above construction gives a $\gR \in \cat{Alg}_{E_\infty}(\cat{sSet}^\bN)$ with $\gR(0) = \varnothing$ and $\gR(n) \simeq B\mr{GL}_n(\R)$. However the general constructions in Sections \ref{sec:buildings} and \ref{sec:local-semi-local} are most natural from the point of view of projective modules, and this is the perspective we shall take.\end{remark}

\section{Higher-dimensional buildings and their split analogues}\label{sec:buildings}

In this section we study various buildings of summands or flags of $\R$-modules. In Section \ref{sec:ns-prop} we relate buildings of summands to buildings of flags using the Nesterenko--Suslin property. Both are in Section \ref{sec:buidings-ek} related to $E_k$-homology of $\gR$.

\subsection{The $k$-dimensional building}

As usual we write $[p]$ for the finite linear order $0 < 1 < 2 < \cdots < p$, considered as a category.  Products $[p_1] \times \ldots \times [p_k]$ denote the product category (or equivalently, product poset), and use the shorter notation $[p_1, \ldots, p_k]$.  Given $a, b \in [p]$ with $a \leq b$ we shall write $[a \leq b]$ for the full subcategory on the set $\{a,b\}$, isomorphic as a poset to either $[0]$ or $[1]$.  Similarly, given $a , b \in [p_1, \ldots, p_k]$ with $a \leq b$ we have a subposet
\begin{equation*}
  [a_1 \leq b_1] \times \ldots \times [a_k \leq b_k] \hookrightarrow [p_1, \ldots, p_k]
\end{equation*}
which is a \emph{cube}, i.e.\ isomorphic as a poset to $[1]^{k'}$ for some $k' \leq k$.  We shall need to refer to the subposet
\begin{equation*}
  [a_1 \leq b_1] \times \ldots \times [a_k \leq b_k]\setminus \{b\} \hookrightarrow [p_1, \ldots, p_k],
\end{equation*}
the cube with its terminal element removed, the \emph{punctured cube}.

For a ring $\R$ and an $\R$-module $M$, let $\cat{Sub}(M)$ denote the set of submodules $P \subset M$ which are summands. By definition, $P$ is a summand if it admits a complement, i.e.~there exists a submodule $Q \subset M$ such that the natural map $P \oplus Q \to M$ is an isomorphism.  If $P \subset M$ is a summand in $M$, it is also a summand in any submodule $M' \subset M$ containing $P$; its complement is the kernel of the restriction to $M'$ of the projection $M \to P$.

We shall regard $\cat{Sub}(M)$ as a partially ordered set with respect to inclusion.  

\begin{definition}\label{defn:lattice}
A functor $\phi \colon [p_1, \ldots, p_k] \to \cat{Sub}(M)$ is called a \emph{lattice} if the natural map
\begin{equation}\label{eq:6}
  \colim_{[a_1 \leq b_1] \times \ldots \times [a_k \leq b_k] \setminus \{b\}} \phi \lra \phi(b)
\end{equation}
is a monomorphism onto a summand (in $M$, or equivalently in $\phi(b)$). (Here the colimit denotes colimit in the category of $\R$-modules, not in the poset $\cat{Sub}(M)$.)

A lattice $\phi \colon [p_1,\ldots, p_k] \to \cat{Sub}(M)$ is \emph{full} if it satisfies
\begin{enumerate}[(i)]
\item\label{item:1} $\phi(a_1,\ldots,a_k) = 0$ if $a_i = 0$ for some $i \in \{1, \ldots, k\}$,
\item $\phi(p_1,\ldots, p_k) = M$.
\end{enumerate}
\end{definition}

\begin{remark}
  Our definition differs from Rognes' \cite[Definition 2.3]{rognesrank}, in that he includes condition~(\ref{item:1}) in his notion of ``lattice''.  Nonetheless, our notions of $k$-dimensional building and stable building are isomorphic to Rognes' (compare our Definition~\ref{def:k-dim-building} with \cite[Definition 3.9]{rognesrank} and our Definition~\ref{def:stable-building} with \cite[Definition 10.8]{rognesrank}): they differ only in that Rognes discards those functors not satisfying~(\ref{item:1}) by passing to a pointed subset, while passing to a quotient set seems more natural to us.
\end{remark}

The following two properties are easily verified.

\begin{lemma}\label{lem:lattices}\mbox{}
\begin{enumerate}[(i)]
\item  For any lattice $\phi \colon [p_1,\ldots, p_k] \to \cat{Sub}(M)$ and any morphism $\theta: [q_1,\ldots,q_k] \to [p_1,\ldots,p_k]$ in $\Delta^{\times k}$, the functor
  \begin{equation*}
    \theta^* \phi \colon [q_1,\ldots, q_k] \lra \cat{Sub}(M)
  \end{equation*}
  is again a lattice.

\item  If $\phi$ is not full, then $\theta^*\phi$ is also not full.\qed
\end{enumerate}
\end{lemma}

\begin{definition}\label{def:k-dim-building}
  The \emph{$k$-dimensional building $D^k(M)_{\bullet, \ldots, \bullet}$} is the $k$-fold simplicial pointed set with $(p_1, p_2, \ldots, p_k)$-simplices defined by the pushout
  \begin{equation*}
    \begin{tikzcd}
      \{\text{non-full lattices}\} \rar[hook]\dar & \{\text{lattices $[p_1,\ldots,p_k] \to \cat{Sub}(M)$}\} \dar \\
      \{\ast\} \rar & D^k(M)_{p_1,\ldots,p_k}
    \end{tikzcd}
  \end{equation*}
The simplicial structure is given by the evident functoriality on $\Delta^k$, via Lemma \ref{lem:lattices}.

  We write $D^k(M)$ for the $k$-fold thin geometric realisation of this $k$-fold simplicial pointed set.
\end{definition}

There is a map of $(k+1)$-fold simplicial pointed spaces
\[S^1_\bullet \wedge D^k(M)_{\bullet, \ldots, \bullet} \lra D^{k+1}(M)_{\bullet, \ldots, \bullet}\]
given on the non-degenerate simplex of $S^1_\bullet$ by the map
\[D^k(M)_{p_1, p_2, \ldots, p_k} \lra D^{k+1}(M)_{1, p_1, p_2, \ldots, p_k}\]
induced by sending a lattice $\phi \colon [p_1,\ldots,p_k] \to \cat{Sub}(M)$ to the lattice $\phi' \colon [1, p_1, \ldots, p_k] \to \cat{Sub}(M)$ given by
\begin{align*}
  \phi'(0,a_1, \ldots, a_k) & = 0\\
  \phi'(1, a_1, \ldots, a_k) & = \phi(a_1, \ldots, a_k)
\end{align*}
Upon geometric realisation these provide the structure maps for a symmetric spectrum \cite[Proposition 3.8]{rognesrank}.

\begin{definition}\label{def:stable-building} The \emph{stable building $\mathbf{D}(M)$} is the spectrum $\{ D^k(M) \}_{k \in \bN}$ with structure maps as above. It comes equipped with an action of $\mr{GL}(M)$ by naturality.\end{definition}

We will use the following criterion to recognise full lattices, using the notation $\underline{p} \coloneqq \{1,\ldots,p\}$:

\begin{lemma}\label{lem:splittings-exist}
  Let $\phi \colon [p_1, \ldots, p_k] \to \cat{Sub}(M)$ be a functor. Then the following are equivalent:
  \begin{enumerate}[(i)]
  	\item $\phi$ is a full lattice,
  	\item the map \eqref{eq:6} is a monomorphism onto a summand whenever $a_i = b_i-1$, $\phi(a_1,\ldots,a_k) = 0$ if $a_i = 0$ for some $i \in \{1,\ldots,k\}$ and $\phi(p_1,\ldots,p_k) = M$,
  	\item there exist $M_i \in \cat{Sub}(M)$ for each $i = (i_1, \ldots, i_k) \in \underline{p_1} \times \ldots \times \underline{p_k}$ such that the $M_i$'s span $M$ and such that for all $a \in [p_1, \ldots, p_k]$, the natural map
  \begin{equation*}
    \bigoplus_{i_1 = 1}^{a_1} \ldots \bigoplus_{i_k = 1}^{a_k} M_i \lra M
  \end{equation*}
  is a monomorphism with image $\phi(a)$.
  \end{enumerate}
\end{lemma}

\begin{proof}That (i) $\Rightarrow$ (ii) is obvious, as we are requiring condition \eqref{eq:6} for fewer cubes. 
	
	\vspace{.5em}
	
	For (iii) $\Rightarrow$ (i), suppose the functor $\phi$ satisfies the condition about the existence of $M_i$'s as stated. In particular the direct sum of all the $M_i$ maps to $M$ by an isomorphism.  Then it is easily verified that the natural maps
	\begin{equation}\label{eq:SumComparison}
	\colim_{[b_1 - 1 \leq b_1] \times \ldots \times [b_k -1 \leq b_k] \setminus \{b\}} \phi \lra \sum_{i < b} M_i \longleftarrow
	\bigoplus_{i < b} M_i
	\end{equation}
	are both isomorphisms, where $\sum$ denotes sum of subspaces and $\bigoplus$ the abstract direct sum.  In both cases the index ``$i < b$'' denotes the set of $i \in \underline{p_1} \times \ldots \times \underline{p_k}$ with $i \leq b$ in the product order but $i \neq b$.  Firstly, the right-hand map is an isomorphism because it is certainly surjective and is injective as we have assumed that $\bigoplus_i M_i \to M$ is an isomorphism.  Secondly, an inverse map from the abstract direct sum to the colimit may be defined by choosing for $i < b$ a $j \in [b_1 - 1 \leq b_1] \times \ldots \times [b_k -1 \leq b_k] \setminus \{b\}$ with $i \leq j$.  The inclusions $M_i \subset \phi(i) \subset \phi(j)$ give a map from $M_i$ to the colimit, which is independent of choice of $j$.  Then the lattice conditions translate to the natural map
	\begin{equation*}
	\bigoplus_{i < b} M_i \lra \bigoplus_{i \leq b} M_i
	\end{equation*}
	being a monomorphism admitting a complement, which is of course true.\vspace{.5em}
	
	For (ii) $\Rightarrow$ (iii), suppose $\phi \colon [p_1, \ldots, p_k] \to \cat{Sub}(M)$ satisfies the stated condition, and choose for each $i$ a submodule $M_i \subset \phi(i)$ which is a complement to the image of the monomorphism
	\begin{equation*}
	\colim_{[i_1 - 1 \leq i_1] \times \ldots \times [i_k -1 \leq i_k] \setminus \{b\}} \phi \lra \phi(b).        
	\end{equation*}
	Then for all $i \leq j$ we have $M_i \subset \phi(i) \subset \phi(j)$, so the inclusions induce a map
	\begin{equation}\label{eq:2}
	\bigoplus_{i \leq a} M_i \lra \phi(a),
	\end{equation}
	which we claim is an isomorphism.  Writing $|a| = a_1 + \ldots + a_k$, this is obvious for $|a| = 0$.  Assuming by induction that this has been proven for all $|a| < |b|$ and in particular for all $a < b$, we deduce as above that the natural maps \eqref{eq:SumComparison}
	are both isomorphisms.  Adding the complement $M_b$ then proves~(\ref{eq:2}) for $a = b$, providing the induction step.
\end{proof}

  The previous lemma implies that a full lattice $\phi \colon [p_1,\ldots, p_k] \to \cat{Sub}(M)$ is determined by the objects
  \begin{equation*}
    L_i^j = \phi(p_1,\ldots, p_{j-1},i,p_{j+1}, \ldots, p_k) \in \cat{Sub}(M), \quad j = 1, \ldots, k; \,\, i = 1, \ldots, p_j.
  \end{equation*}
  Indeed, if $\phi(a) = \sum_{i \leq a} M_i$, then $L_a^j = \bigoplus_{i \text{ with } i_j \leq a} M_i$. Since the $M_i$ intersect trivially, we have
  \[\phi(a) = L^1_{a_1} \cap \ldots \cap L^k_{a_k}.\]
  For each $j = 1, \ldots, k$, the flag $0 \subset L^j_1 \subset \ldots \subset L^j_{p_j}$ forms an element of $D^1(M)_p$. We conclude that:
  
  \begin{lemma}\label{lem:smash-of-E-one-buildings} There is an injection of pointed sets
  \begin{equation}\label{eq:5}
    D^k(M)_{p_1, \ldots, p_k} \lra D^1(M)_{p_1} \wedge \ldots \wedge D^1(M)_{p_k},
  \end{equation}
  assembling to a levelwise injective map of multisimplicial pointed sets.
	\end{lemma}

\subsection{Split buildings}\label{sec:split-building}

The $k$-dimensional \emph{split building} is a $k$-fold simplicial pointed set $\widetilde{D}^k(M)$ where a non-basepoint simplex is a non-basepoint simplex $\sigma$ of $D^k(M)$ together with choices of  submodules $M_i$ as in Lemma~\ref{lem:splittings-exist}.  For $k=1$, a twice desuspended version was first defined by Charney \cite[p.~3]{Charney}. It also appeared in \cite[Section 3.3]{e2cellsIII}. Let us spell out the definition.
  \newcommand{\splitting}{\mathscr{S}}
\begin{definition}\label{def:splitting-segal} 
  Let $X$ be a finite pointed set and $f\colon X \to \cat{Sub}(M)$ a function.  We say that $f$ is a \emph{splitting} if the natural map
  \begin{equation*}
    \bigoplus_{x \in X} f(x) \lra M
  \end{equation*}
  is an isomorphism.

  If $f$ is a splitting and $\theta \colon X \to Y$ is a map of finite pointed sets, then the map
  \begin{align*}
    \theta_* f \colon Y &\lra \cat{Sub}(M)\\
    y &\longmapsto \sum_{x \in \theta^{-1}(x)} f(x)
  \end{align*}
  is also a splitting.  If $f(\ast) \neq 0$, then also $(\theta_* f)(\ast) \neq 0$.  Hence the association
  \begin{equation}\label{eq:3}
    \splitting \colon X \longmapsto \frac{\text{splittings $f \colon X \to \cat{Sub}(M)$}}{\text{splittings with $f(\ast) \neq 0$}}
  \end{equation}
  defines a functor from finite pointed sets to pointed sets, i.e.\ a $\Gamma$-set in the sense of Segal.
\end{definition}

\begin{definition}\label{def:split-building}
  The $k$-dimensional \emph{split building}
   $\widetilde{D}^k(M)_{\bullet, \ldots, \bullet}$ is the $k$-fold simplicial set given as the composition of~(\ref{eq:3}) with the functor
  \begin{align*}
    (\Delta^\mathrm{op})^{\times k} & \lra \text{finite pointed sets}\\
    ([p_1],\ldots, [p_k]) & \longmapsto S^1_{p_1} \wedge \ldots \wedge S^1_{p_k},
  \end{align*}
  where $S^1_\bullet = \Delta[1] / \partial \Delta[1]$ denotes the usual model of the simplicial circle: $S^1_p$ is the quotient of the set of maps $[p] \to [1]$ in $\Delta$ by the subset of constant maps.  Write $\widetilde{D}^k(M)$ for the $k$-fold geometric realisation of this $k$-fold simplicial pointed set.

  A non-basepoint element $\theta \colon [p] \to [1]$ of $S^1_{[p]}$ may be identified with the unique number $i \in \{1, \ldots, p\}$ for which $\theta^{-1}(0) = \{0,\ldots, i-1\}$.  Hence non-basepoint elements of $\widetilde{D}^k(M)_{p_1, \ldots, p_k}$ may be identified with functions (recalling $\underline{p} = \{1,\ldots,p\}$)
  \begin{equation*}
    \mathcal{M} \colon \underline{p_1} \times \ldots \times \underline{p_k} \lra \cat{Sub}(M)
  \end{equation*}
  for which the natural map $\bigoplus_i \mathcal{M}(i) \to M$ is an isomorphism.
\end{definition}

There is a forgetful map of multisimplicial pointed sets
\begin{equation}\label{eq:Comp}
  \begin{aligned}
    \widetilde{D}^k(M)_{\bullet, \ldots, \bullet} &\lra {D}^k(M)_{\bullet, \ldots, \bullet}\\
    \mathcal{M} & \longmapsto \phi_\mathcal{M}
  \end{aligned}
\end{equation}
whose value on a non-basepoint $\mathcal{M}$ is defined by
\begin{equation*}
  \phi_\mathcal{M}(a) = \sum_{i \leq a} \mathcal{M}(i).
\end{equation*}
This is a full flag and, comparing with Lemma~\ref{lem:splittings-exist}, we see that the modules denoted $M_i$ there may be chosen as $\mathcal{M}(i)$.  In this way we have set up a bijection between non-basepoints of $\widetilde{D}^k(M)_{p_1,\ldots, p_k}$ and non-basepoints of $D^k(M)_{p_1,\ldots, p_k}$ together with a choice of splitting submodules as in Lemma~\ref{lem:splittings-exist}.  Under this bijection the map~(\ref{eq:Comp}) amounts to forgetting the splittings.

Just like in Lemma \ref{lem:smash-of-E-one-buildings} we may compare $\widetilde{D}^k$ to a $k$-fold smash product of $\widetilde{D}^1$:

\begin{lemma}
The analogous formula to that of Lemma \ref{lem:smash-of-E-one-buildings} gives maps of pointed sets
  \begin{equation}\label{eq:5tilde}
    \widetilde{D}^k(M)_{p_1, \ldots, p_k} \lra \widetilde{D}^1(M)_{p_1} \wedge \ldots \wedge \widetilde{D}^1(M)_{p_k},
  \end{equation}
  which assemble to a map of multisimplicial pointed sets. These are related to those of Lemma \ref{lem:smash-of-E-one-buildings} by the maps \eqref{eq:Comp}.
\end{lemma}

The group $\mr{GL}(M) = \Aut_\R(M)$ acts on the set $\cat{Sub}(M)$ by sending a submodule to its image under an automorphism of $M$, and hence acts on the multisimplicial pointed sets $D^k(M)_{\bullet, \ldots, \bullet}$ and $\widetilde{D}^k(M)_{\bullet, \ldots, \bullet}$ and their thin geometric realisations.

\begin{proposition}\label{prop:split-map} 
  Let $\R$ be a commutative ring and $M$ a finitely generated projective module.  Then the map of $\mr{GL}(M)$-orbit sets
  \begin{equation*}
    \widetilde{D}^k(M)_{p_1,\ldots,p_k}/\mr{GL}(M) \lra D^k(M)_{p_1,\ldots,p_k}/\mr{GL}(M)
  \end{equation*}
  induced by the $\mr{GL}(M)$-equivariant map~(\ref{eq:Comp}) is a bijection.
\end{proposition}

\begin{proof}
  Lemma~\ref{lem:splittings-exist} implies that the map~(\ref{eq:Comp}) is surjective and hence the induced map of orbit sets is also surjective.

  To prove injectivity, let us first observe that two non-basepoint $(p_1, \ldots, p_k)$-simplices $\mathcal{M}$ and $\mathcal{M}'$ of $\widetilde{D}^k(M)$ are in the same $\mr{GL}(M)$-orbit if and only if $\mathcal{M}(a)$ and $\mathcal{M}'(a)$ are isomorphic as abstract $\R$-modules for each $a \in \underline{p_1} \times \ldots \times \underline{p_k}$.  Indeed, a choice of abstract isomorphisms $f_a \colon \mathcal{M}(a) \to \mathcal{M}'(a)$ may be summed to an automorphism
  \begin{equation*}
    M \xleftarrow{\cong} \bigoplus_a \mathcal{M}(a) \xrightarrow{\oplus_a f_a} \bigoplus_a \mathcal{M}'(a) \xrightarrow{\cong} M
  \end{equation*}
  which will take $\mathcal{M}$ to $\mathcal{M}'$.  But if an automorphism $f \in \mr{GL}(M)$ takes $\phi_\mathcal{M}$ to $\phi_{\mathcal{M}'}$, then it indeed induces isomorphisms
  \begin{equation*}
    \mathcal{M}(a) \cong \frac{\phi_{\mathcal{M}}(i)}{\sum_{i < a} \phi_{\mathcal{M}}(i)} \xrightarrow{\sim}\frac{\phi_{\mathcal{M}'}(i)}{\sum_{i < a} \phi_{\mathcal{M}'}(i)} \cong \mathcal{M}'(a).\qedhere
  \end{equation*}
\end{proof}

\subsection{The Nesterenko--Suslin property}\label{sec:ns-prop}  In this section we explain how to compare $\R$-linear isomorphisms preserving a splitting $\bigoplus_{x \in X} M_x \cong M$ to the $\R$-linear isomorphisms preserving any filtration induced a partial order on the set $X$. We shall use this later to compare the two kinds of buldings we have just introduced.

Suppose we have a ring $\R$ and a filtration 
\[0 \subset P \subset P \oplus Q\]
of $\R$-modules, then we define a subgroup $\mr{GL}(P \oplus Q, \text{fix }P) \leq \mr{GL}(P \oplus Q)$ of automorphisms which fix $P$ pointwise. There are homorphisms
\[\mr{GL}(Q) \lra \mr{GL}(P \oplus Q, \text{ fix }P) \lra \mr{GL}(Q)\]
given respectively by extending an automorphism of $Q$ by the identity on $P$, and taking the induced automorphism of $(P \oplus Q)/P \cong Q$.  These homomorphisms identify $\mr{GL}(P \oplus Q, \text{ fix }P)$ with the semi-direct product $\mr{GL}(Q) \ltimes \Hom_\R(Q,P)$.
Nesterenko and Suslin gave an condition on $\R$ under which these maps are isomorphisms on homology with coefficients in $\bk$, a condition which we axiomatise:

\begin{definition}\label{def:nesterenko-suslin}
	We say that the pair $(\R, \bk)$ satisfies the \emph{Nesterenko--Suslin property} if the maps
\[\mr{GL}(Q) \overset{i}\lra \mr{GL}(P \oplus Q,\text{ fix } P) \overset{\rho}\lra \mr{GL}(Q)\]
	induce isomorphisms on $\bk$-homology for all finitely-generated projective $\R$-modules $P$ and $Q$.
\end{definition}

We shall discuss below some examples where this property holds. If $\mr{GL}(P \oplus Q,\text{ pres } P)$ denotes the subgroup of $\mr{GL}(P \oplus Q)$ of those automorphisms which preserve (setwise) the submodule $P$, then applying the Serre spectral sequence to the map of extensions of groups
\begin{equation*}
\begin{tikzcd}	\mr{GL}(P \oplus Q, \text{ fix } P) \rar & \mr{GL}(P \oplus Q, \text{ pres } P) \rar &  \mr{GL}(P) \\
	\mr{GL}(Q) \rar \uar & \mr{GL}(P) \times \mr{GL}(Q) \rar \uar & \mr{GL}(P) \uar[equals]
\end{tikzcd}
\end{equation*}
shows that the natural inclusion
\[\mr{GL}(P) \times \mr{GL}(Q) \lra \mr{GL}(P \oplus Q, \text{ pres } P)\]
is also a homology isomorphism. In this section we will provide a technical generalisation of this result, where we consider a module with a collection of submodules indexed by an arbitrary finite poset.

Let $M$ be an $\R$-module, $X$ a finite set, and $s \in \splitting(X_+)$ a non-basepoint as in~\eqref{eq:3}. That is, $s$ is the data of a direct sum decomposition $M \cong \bigoplus_{x \in X} M_x$ indexed by the set $X$.  Then any endomorphism $f \in \End_\R(M)$ is given by a matrix with entries $f_{x,y} \in \Hom_\R(M_x,M_y)$.  We say that $f$ is \emph{upper triangular} with respect to a partial order $\preceq$ on $X$ when $f_{x,y} = 0$ unless $y \preceq x$.  Equivalently, $f$ should preserve for all $x \in X$ the submodules 
\[M_{\preceq x} \coloneqq \bigoplus_{y \preceq x} M_y.\]

We let $\mr{GL}(M, \preceq) \subseteq \mr{GL}(M)$
denote the subgroup of those automorphisms which are upper triangular with respect to the partial order $\preceq$. If $\preceq$ and $\leq$ are partial orders on $X$ such that $\preceq$ is contained in $\leq$, when regarding the partial orders as subsets of $X \times X$
(so the containment means $x \preceq y \Rightarrow x \leq y$), then we have an inclusion of subgroups
\begin{equation}\label{eq:4}
  \mr{GL}(M,\preceq) \hookrightarrow \mr{GL}(M,\leq).
\end{equation}
For example, let $M = \bigoplus_{i=1}^n \R$. On the one hand, if $\preceq$ is the identity partial order on $\{1,\ldots,n\}$ (i.e.~$x \preceq y$ if and only if $x = y$) then $\mr{GL}(M,\leq)$ consists of the diagonal matrices. On the other hand, if $\leq$ is the usual linear order (i.e.~$x \leq y$ if and only if $y-x$ is non-negative) then $\mr{GL}(M,\leq)$ consist of the upper-triangular matrices. Thus $\preceq$ is contained in $\leq$, and we indeed have an inclusion as \eqref{eq:4}.

\begin{proposition}\label{propthm:SN}
	If $(\R, \bk)$ satisfies the Nesterenko--Suslin property then the inclusion~\eqref{eq:4} induces an isomorphism in group homology with constant coefficients in $\bk$. In particular this holds for
	\[i \colon \prod_{x\in X} \mr{GL}(M_x) \lra \mr{GL}(M, \leq).\]
\end{proposition}
\begin{proof}
  The map $i$ is the special case of~\eqref{eq:4} where $\preceq$ is the identity relation on $X$.
  Conversely that special case implies the general case, by comparing the identity relation to both $\leq$ and $\preceq$.
  
  If $\leq$ is an arbitrary partial order on $X$, we may obtain a sequence of partial orders, each contained in the previous, by successively declaring a maximal element $t \in X$ to be unrelated to elements of $X \setminus \{t\}$ until we reach the identity relation. We prove the proposition by induction over these partial orders; the initial case of the identity relation was given above.

  For the induction step, we must prove the case where $\preceq$ and $\leq$ agree on $X' = X \setminus \{t\}$ for some $t \in X$, such for any $x \in X'$, neither $t \preceq x$, $x \preceq t$, nor $x \geq t$ holds (but $x \leq t$ is allowed).  Writing $M' = \sum_{x \in X'} M_x \subset M$, we get an isomorphism $M_t \oplus M' \to M$, and any morphism $f \in \mr{GL}(M,\leq)$ will satisfy $f(M_t) \subset M_t$.  Using the isomorphism $M' \to M/M_t$ gives an epimorphism
  \begin{equation*}
    \mr{GL}(M,\leq) \lra \mr{GL}(M',\leq_{\vert M'}),
  \end{equation*}
  whose kernel is the subgroup $\mr{GL}(M,\text{ fix } M') < \mr{GL}(M)$.  Replacing $\leq$ by $\preceq$ instead gives an epimorphism $\mr{GL}(M',\leq_{\vert M'}) \times \mr{GL}(M_t) \cong \mr{GL}(M,\preceq) \to \mr{GL}(M',\leq_{\vert M'})$ with kernel $\mr{GL}(M_t)$.  We have a comparison of group extensions
  \begin{equation*}
    \begin{tikzcd}
      \mr{GL}(M,\text{ fix } M') \rar \dar & \mr{GL}(M, \leq) \rar \dar{\rho}& \mr{GL}(M', \leq _{\vert_{M'}})\dar[equals]\\
      \mr{GL}(M_t) \rar& \mr{GL}(M,\preceq) \rar  & \mr{GL}(M', \leq_{\vert_{M'}}),
    \end{tikzcd}
  \end{equation*}
  where the map
  \begin{equation*}
    \rho: \mr{GL}(M, \leq) \lra \mr{GL}(M,\preceq) \cong \mr{GL}(M',\leq_{\vert M'}) \times \mr{GL}(M_t)
  \end{equation*}
  sends $f \in \mr{GL}(M,\leq)$ to the associated graded with respect to the filtration $0 \subset M_t \subset M$.  The induced map of kernels is
  \begin{equation*}
    \mr{GL}(M,\text{ fix $M'$}) \lra \mr{GL}(M_t),
  \end{equation*}
  an instance of the homomorphisms assumed to induce isomorphisms in group homology with constant coefficients $\bk$ in the Nesterenko--Suslin property, Definition \ref{def:nesterenko-suslin}. Hence a spectral sequence argument implies that $\rho$ also induces an isomorphism in group homology with constant coefficients $\bk$.  Since $\rho$ is a one-sided inverse to the inclusion, the latter has the same property. This concludes the proof of the induction step.
\end{proof}

It remains to give conditions under which the Nesterenko--Suslin property holds. The main condition we will be interested in is that studied in \cite[Section 1]{SN}:

\begin{definition}A ring $\R$ has \emph{many units} if for each $n \in \bN$ there exist elements $a_1, a_2, \ldots, a_n \in \R$ such that
  \begin{equation*}
    \sum_{i \in J} a_i \in \R^\times
  \end{equation*}
  for any non-empty $J \subset \{1, \ldots, n\}$.
\end{definition}

\begin{proposition}[Nesterenko--Suslin]\label{propthm:SN-many-units}
	If $\R$ has many units then $(\R, \mathbb{Z})$ satisfies the Nesterenko--Suslin property.
\end{proposition}

Strictly speaking Nesterenko and Suslin only prove the case that $A$ and $B$ are finitely-generated free modules, but the proof of the key result \cite[Proposition 1.10]{SN} goes through with $A^s$ replaced by a general $A$-module.

\begin{example}\label{exam:many-units} 
Every infinite field $\bF$ has many units, and in fact has the stronger property that there exists an infinite list $a_1,a_2,\ldots$ of elements of $\bF$ such that each finite partial sum is a unit: suppose we have found a list $a_1,\ldots,a_{n-1}$ of elements of $\bF$ all of whose partial sums are units, then because $\bF$ is infinite we can find non-zero element $a_n$ of $\bF$ which is not a negative of any of these partial sums and $a_1,\ldots,a_n$ is a longer list all of whose partial sums are units. 
	
  More generally, if $\R$ is a local ring with infinite residue field $\bF$, then any infinite list of elements of $\bF$ as above may be lifted to elements $a_1, a_2, \ldots \in \R$ in which any non-trivial partial sum is a unit (since being a unit may be checked in $\bF$). 
  
  Recall that the Jacobson radical $J(\R)$ is the intersection of all maximal ideals. Then, more generally, if $\R$ is a semi-local ring then the codomain of the surjective homomorphism
  \begin{equation*}
    \R \lra \R/J(\R)
  \end{equation*}
  can be identified with $\prod_{\mathfrak{m}} \R/\mathfrak{m}$ using the Chinese remainder theorem, where the product is indexed by the finite set of maximal ideals. If each residue field $\R/\mathfrak{m}$ is infinite then we may choose an infinite list of elements where all non-trivial partial sums are invertible, independently in each residue field, and lift this to an infinite list of elements of $\R$ where each non-trivial partial sum is a unit.
\end{example}

\begin{remark}\label{rem:ns-other} There are other conditions under which the Nesterenko--Suslin property holds.  In particular, Quillen proved that if $\R$ contains $\frac{1}{p}$ for some $p \in \mathbb{Z}$ then $(\R, \mathbb{Q})$ satisfies the Nesterenko--Suslin property, \cite[p.~203]{QuillenChar}. (He also asked whether the assumption was really necessary, and in particular whether $(\bZ, \bQ)$ satisfies the Nesterenko--Suslin property. In Section \ref{sec:QuillenQuestion} we will answer this question by proving that it does not.)
\end{remark}

The Nesterenko--Suslin property can be used to compare the homology of the $k$-dimensional split building to that of $k$-dimensional building, after taking homotopy orbits by $\mr{GL}(M)$:

\begin{theorem}\label{thmcor:BlockvsFlag}
  Assume that $(\R, \bk)$ satisfies the {Nesterenko--Suslin property}.  Then the map
  \[\tilde{D}^{k}(M) \hcoker \mr{GL}(M) \lra {D}^{k}(M) \hcoker \mr{GL}(M)\]
  of pointed homotopy orbit spaces of~(\ref{eq:Comp}) induces an isomorphism on  $\bk$-homology.
\end{theorem}

\begin{proof}
	By Proposition \ref{prop:split-map} the map \eqref{eq:Comp} induces a bijection on $\mr{GL}(M)$-orbits of simplices, so it suffices to prove that the map
	\begin{equation}\label{eq:CompStab}
	\mr{Stab}_{\mr{GL}(M)}(\mathcal{M}) \lra \mr{Stab}_{\mr{GL}(M)}(\phi_\mathcal{M})
	\end{equation}
	induces an isomorphism on group homology with constant coefficients $\bk$ for all non-basepoint $\mathcal{M} \in \widetilde{D}^k(M)_{p_1,\ldots, p_k}$.
        
	This holds by an application of Proposition~\ref{propthm:SN}. Namely, for the given $\mathcal{M} \colon X \to \cat{Sub}(M)$ with $X = \underline{p_1} \times \ldots \times \underline{p_k}$, we take the order $\leq$ obtained by restricting the product order on $[p_1] \times \cdots \times [p_k]$ to $X$.  Writing $M_x = \mathcal{M}(x)$, the map \eqref{eq:CompStab} is then identified with an instance
	\[i \colon \prod_{x \in X} \mr{GL}(M_x) \lra \mr{GL}(M, \leq)\]
	of the homomorphism in Proposition~\ref{propthm:SN}, so by this proposition it induces an isomorphism in group homology with coefficients in $\bk$ as long as $(\R, \bk)$ satisfies the Nesterenko--Suslin property.
\end{proof}

\subsection{Relationship with $E_k$-homology}\label{sec:buidings-ek}
In this section we explain the relationship between the $k$-dimensional split buildings $\tilde{D}^k(M)$ and the $E_k$-homology groups $H^{E_k}_{*,\ast}(\gR)$ with $\gR$ the $E_\infty$-algebra constructed in Section \ref{sec:contructing-r}. 

To do so, we specialise the results of Section $E_k$.17.3 to the symmetric monoidal groupoid $\cat{P}_A$ with rank functor $r \colon \cat{P}_A \to \bN$, which satisfies the required hypotheses by Lemma \ref{lem:vect-cat-props}. This section explains that the derived indecomposables of $\gR \in \Alg_{E_k}(\cat{sSet}^\bN)$ may be computed rather explicitly in terms of $k$-fold pointed simplicial sets $Z^{E_k}_{\bullet,\ldots,\bullet}(M)$ with $\mr{GL}(M)$-action known as \emph{$E_k$-splitting complexes}.

Considering the finite set $\underline{p} = \{1,\ldots,p\}$ as a discrete category, we let $\cat{P}_A^{p_1 \cdots p_k}$ denote the category of functors $\cat{Fun}(\underline{p_1} \times \cdots \times \underline{p_k},\cat{P}_A)$; an object of this category consists of a collection of objects $M_{i_1,\ldots,i_k}$ for $1 \leq i_j \leq p_j$.

\begin{definition}For $M \in \cat{P}_A$, we define a $k$-fold simplicial pointed set  $Z^{E_k}_{\bullet,\ldots,\bullet}(M)$ with the pointed set of $(p_1,\ldots,p_k)$-simplices given by 
	\[Z^{E_k}_{p_1,\ldots,p_k}(M) \coloneqq \left(\underset{\cat{P}_A^{p_1\cdots p_k}}{\mr{colim}} \,\cat{P}_A(M_{1,\ldots,1} \oplus \cdots \oplus M_{p_1,\ldots,p_k},M) \right)_+.\]

For $1 \leq i \leq k$, the face maps $\smash{d^i_0},\smash{d^i_{p_i}}$ are the constant maps to the basepoint. The face maps $d^i_j$ for $0<j<p_i$ are induced by replacing each pair of objects $(M_{a_1,\ldots,a_{i-1},j,a_j,\ldots,a_p},M_{a_1,\ldots,a_{i-1},j+1,a_j,\ldots,a_p})$ by its direct sum. The degeneracy maps insert $0$'s. It has a remaining $\mr{GL}(M)$-action, and we write $Z^{E_k}(M) \coloneqq |Z_{\bullet,\ldots,\bullet}(M)| \in \cat{sSet}_\ast^{\mr{GL}(M)}$ for its $k$-fold thin geometric realisation.\end{definition}

Corollary $E_k$.17.15 then says that for each finitely-generated projective $\R$-module $M$ there is a weak equivalence
\[S^k \wedge Q^{E_k}_\bL(\gR)(n) \simeq \bigvee_{\substack{[M]\\r(M)=n}}Z^{E_k}(M) \hcoker \mr{GL}(M)\]
This may be related to the split buildings as follows.

\begin{theorem}\label{thm:split-building-ek} 
For any connected ring $\R$ and any finitely-generated projective $\R$-module $M$ there are $\mr{GL}(M)$-equivariant isomorphisms $Z^{E_k}(M) \cong \tilde{D}^k(M)$. Thus if $\R$ is connected and $\gR_\bk$ is the $E_\infty$-algebra constructed in Section \ref{sec:contructing-r}, there are isomorphisms
\[H^{E_k}_{n,d}(\gR_\bk) \cong \bigoplus_{\substack{[M]\\r(M)=n}} \widetilde{H}_{n,d-k}(\tilde{D}^k(M) \hcoker \mr{GL}(M);\bk).\]
\end{theorem}
\begin{proof}
For $\vec{M} \in \cat{P}_\R^{p_1\cdots p_k}$, an isomorphism $f \colon M_{1,\ldots,1} \oplus \cdots \oplus M_{p_1,\ldots,p_k} \overset{\sim}\lra M$ determines a function (recall $\underline{p}=\{1,\ldots,p\}$)
\[\mathcal{M}_f \colon \underline{p_1} \times \cdots \times \underline{p_k} \lra \cat{Sub}(M)\]
	by the formula $\mathcal{M}_f(i_1, \ldots, i_k) = f(M_{i_1, \ldots, i_k})$, giving a non-basepoint element of $\tilde{D}^k_{p_1, \ldots, p_k}(M)$. The element obtained is invariant under precomposing $f$ with an isomorphism obtained from $\cat{P}_\R^{p_1\cdots p_k}$, so this extends to a pointed map
	\[Z^{E_k}_{p_1,\ldots,p_k}(M) \coloneqq \left(\underset{\vec{M} \in \cat{P}_\R^{p_1\cdots p_k}}{\mr{colim}} \cat{P}_\R(M_{1,\ldots,1} \oplus \cdots \oplus M_{p_1,\ldots,p_k},M)\right)_+ \lra \tilde{D}^k_{p_1, \ldots, p_k}(M).\]
This is a bijection, and it is straightforward to check that these assemble into an isomorphism of $k$-fold pointed simplicial sets $Z^{E_k}_{\bullet,\ldots,\bullet}(M) \overset{\sim}\to \tilde{D}^k_{\bullet,\ldots,\bullet}(M)$.
\end{proof}

\section{General linear groups over a field}\label{sec:GLfield}

In this section we take $\R=\bF$ a field and return to the $E_\infty$-algebra $\gR$ in $\cat{sSet}^\bN$ which we constructed in Section \ref{sec:contructing-r}. As projective modules are free over $\bF$, this satisfies 
\[\gR(n) \simeq \begin{cases}\varnothing & \text{if $n=0$,} \\
B\mr{GL}_n(\bF) & \text{otherwise.}\end{cases}\]
The associated non-unital $E_\infty$-algebra $\gR_\bk$ in $\cat{sMod}_\bk^\bN$ then satisfies $H_{n,d}(\gR_\bk) = H_d(\mr{GL}_n(\bF);\bk)$ for all $n>0$. 

In this section we shall study its $E_k$-homology groups $\smash{H_{n,d}^{E_k}(}\gR_\bk)$ for $k \in \{1,2,\infty\}$.  Our main result is that these groups vanish for $d < n-1$ when $k=1$, and when $\bF$ is an \emph{infinite} field they vanish for $d < 2n-2$ when $k \in \{2,\infty\}$.  On the boundary line $d = 2n-2$ we will also calculate these groups explicitly for both $k=2$ and $k=\infty$.  

Most of the results in this section hold more generally for local rings, or even connected semi-local rings whose residue fields are infinite, but for the sake of clarity we prefer to first explain the simpler case of fields where most of the work has already been done; we postpone these generalisations to Section \ref{sec:local-semi-local}. Furthermore, as we have discussed finite fields in detail in \cite{e2cellsIII}, here we shall only consider infinite fields. 

An infinite field $\bF$ has many units by Example \ref{exam:many-units}, so $(\bF, \bk)$ satisfies the Nesterenko--Suslin property for any $\bk$, and hence by combining Theorems \ref{thm:split-building-ek} and \ref{thmcor:BlockvsFlag} we have isomorphisms
\begin{equation}\label{eqn:ek-homology-buildings}
  H_{n,d}^{E_k}(\gR_\bk) \cong \widetilde{H}_{d-k}(\widetilde{D}^k(\bF^n)\hcoker \mr{GL}_n(\bF);\bk) \cong \widetilde{H}_{d-k}(D^k(\bF^k) \hcoker \mr{GL}_n(\bF);\bk).
\end{equation}
The results will be established by studying the (equivariant) homotopy type of the pointed spaces $D^1(\bF^n)$ and $D^2(\bF^n)$.  As we shall see, both of these may be described in terms of the classical Tits building for $\mr{GL}_n(\bF)$.

\subsection{$E_1$-homology}
\label{sec:e_1-cells}

Recall from Definition \ref{defn:TitsBldg} that the Tits building $T_\bullet(\bF^n)$ is the nerve of the poset of proper nontrivial subspaces of $\bF^n$. It is easy to relate this to $D^1(\bF^n)_\bullet$ as in Definition \ref{def:k-dim-building} by comparing definitions: the non-degenerate elements of $T_p(\bF^n)$ are in bijection with the non-degenerate elements of $D^1(M)_{p+2}$.  In fact, up to homotopy equivalence (through a zig-zag of based $\mr{GL}(\bF^n)$-equivariant maps) it may be identified with the unreduced double suspension of $T_\bullet(\bF^n)$, as we now explain.  

\begin{lemma}\label{lem:splittingcomplex-vs-building} There is a weak equivalence of $\mr{GL}(\bF^n)$-spaces
	\[S^2 \cT(\bF^n) \simeq D^1(\bF^n),\]
	where $S^2(-)$ denotes the unreduced double suspension.
\end{lemma}

\begin{proof}
Let us define
\begin{align*}
\cT^{\pm}_p(\bF^n) &\coloneqq N_p(\{0\} \cup \cT(\bF^n) \cup \{\bF^n\},\subset)\\
\cT^{+}_p(\bF^n) &\coloneqq N_p(\cT(\bF^n) \cup \{\bF^n\},\subset)\\
\cT^{-}_p(\bF^n) &\coloneqq N_p(\{0\} \cup \cT(\bF^n),\subset).
\end{align*}
Then $\cT_p(\bF^n) = \cT^-_p(\bF^n) \cap \cT^+_p(\bF^n)$ and so we have pushout diagrams of sets
\begin{equation*}
  \begin{tikzcd}
    \cT_p(\bF^n) \rar[hook]\dar[hook] & \cT^+_p(\bF^n) \dar &[-25pt] & \cT^-_p(\bF^n) \cup \cT_p^+(\bF^n) \rar[hook] \dar & \cT^{\pm}_p(\bF^n)\dar\\
    \cT^-_p(\bF^n) \rar & \cT^-_p(\bF^n) \cup \cT_p^+(\bF^n) &[-25pt] & \{\ast\} \rar & D^1(\bF^n)_p,
  \end{tikzcd}
\end{equation*}
which---together with the contractibility of the thin geometric realisations $|\cT^\pm_\bullet(\bF^n)|$, $|\cT^+_\bullet(\bF^n)|$, and $|\cT^-_\bullet(\bF^n)|$, as these posets have top or bottom elements---give a zig-zag of weak equivalences between the unreduced double suspension of $T(\bF^n)$ and $D^1(\bF^n)$.\end{proof}

Recall that the Solomon--Tits theorem---here Theorem \ref{thm:SolomonTits}---says that $T(\bF^n)$ has the homotopy type of a wedge of $(n-2)$-spheres, and the Steinberg module, denoted $\mr{St}(\bF^n)$, is by definition its reduced integral homology in this degree; we write \[\mr{St}_\bk(\bF^n) \coloneqq \mr{St}(\bF^n) \otimes _\bZ \bk=\widetilde{H}_{n-2}(\cT(\bF^n);\bk).\] 
Combining the weak equivalence between the unreduced double suspension of $T(\bF^n)$ and $D^1(M)$ with the identification \eqref{eqn:ek-homology-buildings} for $k=1$, we have thus proved the following.

\begin{theorem}\label{thm:E1Homology}
Suppose that $(\bF, \bk)$ satisfies the Nesterenko--Suslin property. Then there is an isomorphism
  \begin{equation*}
    H^{E_1}_{n,d}(\gR_\bk) \cong H_{d-(n-1)}(\mr{GL}_n(\bF);\mr{St}_\bk(\bF^n)).
  \end{equation*}
In particular, $H^{E_1}_{n,d}(\gR_\bk)=0$ for $d < n-1$.
\end{theorem}

In addition, using the Lee--Szczarba resolution \cite[Section 3]{leeszczarba}, which we already used in Section \ref{sec:coinvariants-of}, one may easily compute the coinvariants of Steinberg modules: \cite[Theorem 4.1]{leeszczarba} implies that 
\[H_{0}(\mr{GL}_n(\bF);\mr{St}(\bF^n))=0 \quad \text{for} \quad n \geq 2.\]
Thus $H^{E_1}_{n,n-1}(\gR_\bk)=0$ for $n \geq 2$ too.

As we have explained, our motivation in computing $E_k$-homology is that it allows us to estimate $E_k$-cell structures. The general result in this direction is Theorem $E_k$.11.21. In the current situation it says that under the assumptions of Theorem \ref{thm:E1Homology} there exists a CW-$E_1$-algebra $\gC$ whose underlying $E_1$-algebra is weakly equivalent to $\gR$, and so that the associated graded of the skeletal filtration is a free $E_1$-algebra
\begin{equation*}
\mr{gr}(\gC) \cong \gE_1\left[S_\bk^{1,0,0} \oplus \bigoplus_\alpha S^{n_\alpha,d_\alpha, d_\alpha}_\bk\right]
\end{equation*}
on a generator of tridegree $(1,0,0)$ and further generators $\alpha$ of tridegrees $(n_\alpha,d_\alpha, d_\alpha)$ such that $d_\alpha \geq n_\alpha$.

\subsection{$E_2$-homology}
\label{sec:e_2-cells}

We wish to repeat this analysis for $E_2$-homology. Recall from Lemma \ref{lem:smash-of-E-one-buildings} that the $k$-fold buildings for any module $M$ over any ring $\R$ are simplicial pointed subsets of $k$-fold smash powers of $D^1(M)_\bullet$. In the case where $k=2$ and $\R = \bF$ is a field we have the following at first surprising coincidence.

\begin{proposition}\label{prop:smash-splitting-E2}
For any $\bF$-vector space $M$ the injection
  \begin{equation*}
      D^2(M) \lra D^1(M) \wedge D^1(M).
  \end{equation*}
from~(\ref{eq:5}) is in fact a bijection.
\end{proposition}

\begin{proof}
  In Lemma~\ref{lem:smash-of-E-one-buildings} we already saw that these maps are always injective. They identify non-degenerate elements of $D^2(M)_{p_1,p_2}$ with pairs of flags
  \begin{align*}
    & (0 = L^1_0 \subsetneq L^1_1 \subsetneq \ldots \subsetneq L^1_{p_1} = M)\\
    & (0 = L^2_0 \subsetneq L^2_1 \subsetneq \ldots \subsetneq L^2_{p_2} = M)
  \end{align*}
  with the property that the functor $\phi: [p_1,p_2] \to \cat{Sub}(M)$ defined by
  \begin{equation*}
    \phi(i_1,i_2) = L^1_{i_1} \cap L^2_{i_2}
  \end{equation*}
  satisfies the lattice condition.  

However, the lattice condition is automatic for $k=2$ when working over a field because any submodule is a summand: for a 1-cube there is nothing to check and for a 2-cube we use that the pushout in abstract vector spaces of a diagram $V' \leftarrow V \to V''$ of inclusions of sub vector spaces of $M$ injects into $M$ precisely when $V = V' \cap V''$.  In our case
  the colimit over $([i_1 < j_1] \times [i_2 < j_2]) \setminus \{(j_1,j_2)\}$ is the pushout of the diagram
  \begin{equation*}
    (L^1_{j_1} \cap L^2_{i_2}) \longleftarrow (L^1_{i_1} \cap L^2_{i_2}) \lra (L^1_{j_1} \cap L^2_{j_2}),
  \end{equation*}
  which indeed satisfies $(L^1_{i_1} \cap L^2_{i_2}) = (L^1_{j_1} \cap L^2_{i_2}) \cap (L^1_{i_1} \cap L^2_{j_2})$.
\end{proof}

\begin{remark}\label{rem:NoGo}
 	It is instructive to see what goes wrong in the above argument when one replaces 2 by a larger number. For example consider the 3-dimensional building, and define flags in $\bF^2$ by
 	\[L_1^1 = \mr{span}(e_1), \quad\quad L^2_1 = \mr{span}(e_2), \quad\quad L^3_1 = \mr{span}(e_1+e_2), \]
 	and $L^1_2=L^2_2=L^3_2 = \bF^2$. Then we see that
 	\[L_1^1 \cap L^2_2 \cap L^3_2 = L_1^1, \quad\quad L^1_2 \cap L^2_1 \cap L^3_2 = L_1^2, \quad \text{and} \quad L^1_2 \cap L^2_2 \cap L^3_1 = L_1^3\]
 	are all 1-dimensional. Furthermore, an intersection is zero whenever more than a single subscript equals 1. 
 	
 	If this pair of flags came from a 3-dimensional lattice $\phi$ then $\phi(2,1,1)$ must be a submodule both of $\phi(2,2,1) = L^3_1 = \mr{span}(e_1+e_2)$ and of $\phi(2,1,2) = L^2_1 = \mr{span}(e_2)$, so must be trivial; similarly $\phi(1,2,1)$ and $\phi(1,1,2)$ must be trivial, as must $\phi(1,1,1)$. But then
 	\[\colim_{[0 \leq 2]^3 \setminus (2,2,2)} \phi = L_1^1 \oplus L_1^2 \oplus L_1^3 \]
 	which does not inject into $\bF^2$; thus such a lattice $\phi$ cannot exist.
\end{remark}

We can now prove the analogue for $E_2$-homology of Theorem \ref{thm:E1Homology}. We shall write 
\[\mr{St}^2(\bF^n)\coloneqq \mr{St}(\bF^n) \otimes \mr{St}(\bF^n) \qquad \text{and} \qquad \mr{St}^2_\bk(\bF^n) = \mr{St}^2(\bF^n) \otimes_\bZ \bk.\]

\begin{theorem}\label{thm:E2Homology}
Suppose that $(\bF, \bk)$ satisfies the Nesterenko--Suslin property. Then there is an isomorphism
  \begin{equation*}
    H^{E_2}_{n,d}(\gR_\bk) \cong H_{d-2(n-1)}(\mr{GL}_n(\bF);\mr{St}_\bk^2(\bF^n)).
  \end{equation*}
In particular $H^{E_2}_{n,d}(\gR_\bk)=0$ if $d < 2(n-1)$, and furthermore $H^{E_2}_{n,2n-2}(\gR) \cong \bZ$ for $n \geq 1$ by Theorem \ref{thm:C}.
\end{theorem}

\begin{proof}
We have already explained that when $(\bF, \bk)$ satisfies the Nesterenko--Suslin property there is an identification \eqref{eqn:ek-homology-buildings} of $H^{E_2}_{n,d}(\gR_\bk)$ with $\widetilde{H}_{d+2}(D^2(\bF^n)\hcoker \mr{GL}_n(\bF);\bk)$. By Proposition \ref{prop:smash-splitting-E2} we may write the latter as
\[\widetilde{H}_{d+2}\left((D^1(\bF^n) \wedge D^1(\bF^n))\hcoker \mr{GL}_n(\bF);\bk\right).\]
By the discussion in Section \ref{sec:e_1-cells}, $D^1(\bF^n)$ has the homotopy type of a wedge of $n$-spheres, with top integral homology $\mr{St}(\bF^n)$, a free $\bZ$-module. Thus the smash product $D^1(\bF^n) \wedge D^1(\bF^n)$ has the homotopy type of a wedge of $2n$-spheres, and the K\"unneth theorem gives
  \begin{equation*}
    \widetilde{H}_{2n}\left(D^1(\bF^n) \wedge D^1(\bF^n)\right) \cong \mr{St}(\bF^n) \otimes_\bZ \mr{St}(\bF^n) = \mr{St}^2(\bF^n).
  \end{equation*}
Hence the spectral sequence calculating the $\bk$-homology of the homotopy orbits $(D^1(\bF^n) \wedge D^1(\bF^n))\hcoker \mr{GL}_n(\bF)$ degenerates to an isomorphism
  \begin{equation*}
    \widetilde{H}_{d+2}\left((D^1(\bF^n) \wedge D^1(\bF^n))\hcoker \mr{GL}_n(\bF);\bk\right) \cong {H}_{d-2(n-1)}(\mr{GL}_n(\bF);\mr{St}^2_\bk(\bF^n)).\qedhere
  \end{equation*}
\end{proof}

Just as in the case of $E_1$-algebras, this vanishing of $E_2$-homology groups yields, via Theorem $E_k$.11.21, the following conclusion about cellular $E_2$-algebra models for $\gR_\bk$: under the same assumption as Theorem \ref{thm:E1Homology} there exists a CW-$E_2$-algebra $\gC$ whose underlying $E_2$-algebra is $\gR$ and such that the associated graded of the skeletal filtration is a free $E_2$-algebra
\begin{equation*}
\mr{gr}(\gC) \cong \gE_2\left[\bigoplus_\alpha S^{n_\alpha,d_\alpha, d_\alpha}_\bk\right]
\end{equation*}
on generators $\alpha$ of tridegrees $(n_\alpha,d_\alpha, d_\alpha)$ such that $d_\alpha \geq 2(n_\alpha-1)$.

\subsection{Products}
\label{sec:products}

In Section $E_k$.13.6 we have explained that the $k$-fold reduced bar construction of an $E_{n+k}$-algebra is naturally equivalent to an $E_n$-algebra, and moreover the $n$-fold reduced bar construction of this algebra is equivalent to the $(n+k)$-fold reduced bar construction of the $E_{n+k}$-algebra that we started with. (In Theorem $E_k$.13.23 this is stated in terms of the suspended derived indecomposables instead of reduced bar constructions but these are equivalent by Theorem $E_k$.13.7, and in any case the proof of Theorem $E_k$.13.23 goes through the reduced bar construction statement.) In particular the $k$-fold reduced bar construction of an $E_\infty$-algebra is again an $E_\infty$-algebra, up to canonical equivalence.

As we explained in Section~\ref{sec:buidings-ek}, the split building $M \mapsto \widetilde{D}^k(M)$ may be identified with the $k$-fold bar construction of the $E_\infty$-algebra $\gT \in \cat{sSet}^{\cat{P}_\bF}$ introduced in Section~\ref{sec:contructing-r}. We deduce that the split buildings have the structure of $E_\infty$-algebras in $\cat{sSet}_*^{\cat{P}_\bF}$, the category of functors from $\cat{P}_\bF$ to pointed simplicial sets.  The underlying $E_1$-algebra structure corresponds to the product
\begin{align*}
  \widetilde{D}^k(M)_{p_1, \ldots, p_k} \wedge \widetilde{D}^k(N)_{p_1, \ldots, p_k} &\lra \widetilde{D}^k(M \oplus N)_{p_1, \ldots, p_k}\\
\mathcal{M} \wedge \mathcal{N} \quad\quad\quad\quad\,\,\,\,\, &\longmapsto \left[(i_1, \ldots, i_k) \mapsto \mathcal{M}(i_1, \ldots, i_k) \oplus \mathcal{N}(i_1, \ldots, i_k)\right].
\end{align*}
That is, we take a direct sum decomposition of $M$ and one of $N$ indexed on the same set, and send it to the direct sum decomposition of $M\oplus N$ obtained by adding the summands.

This product is strictly associative and strictly commutative, in the sense that it makes $\widetilde{D}^k$ into a non-unital commutative ring object in $\cat{sSet}_*^{\cat{P}_\bF}$, i.e.\ a functor from $\cat{P}_\bF$ to pointed simplicial sets with a non-unitary lax symmetric monoidality. Furthermore, the natural transformations $\widetilde{D}^k(-) \to \widetilde{D}^1(-)^{\wedge k}$ of~(\ref{eq:5tilde}) are maps of non-unital commutative ring objects.

\begin{lemma}
  There are unique natural transformations 
  \[D^k(M) \wedge D^k(N) \lra D^k(M \oplus N)\] 
  making the diagram
  \begin{equation*}
    \begin{tikzcd}
      \widetilde{D}^k(M) \wedge \widetilde{D}^k(N) \rar\dar & \widetilde{D}^k(M \oplus N) \dar\\
      D^k(M) \wedge D^k(N) \rar & D^k(M \oplus N)
    \end{tikzcd}
  \end{equation*}
  commute.  These products on $D^k$ are again associative and commutative in the sense of defining a non-unital commutative ring object. Furthermore, the natural transformations $D^k(-) \to D^1(-)^{\wedge k}$ of~(\ref{eq:5}) are maps of non-unital commutative ring objects.
\end{lemma}
\begin{proof}
  Uniqueness is evident from surjectivity of the forgetful map $\widetilde{D}^k(M)_{p_1, \ldots, p_k} \to D^k(M)_{p_1, \ldots, p_k}$, and existence may be seen by writing down a definition: indeed, on non-basepoints we may send full lattices 
	\[\phi \colon [p_1, \ldots, p_k] \lra \cat{Sub}(M) \quad \text{ and } \quad \phi' \colon [p_1, \ldots, p_k] \lra \cat{Sub}(N)\]
	to the lattice
  \begin{align*}
    \phi \oplus \phi' \colon [p_1, \ldots, p_k] & \lra \cat{Sub}(M \oplus N)\\
    i & \longmapsto \phi(i) \oplus \phi'(i).
  \end{align*}
  Commutativity and associativity follows from the corresponding properties for the split building, as does the multiplicativity of $D^k(-) \to D^1(-)^{\wedge k}$.
\end{proof}

These in turn induce pairings
\begin{equation*}
  - \cdot - \colon \widetilde{H}_*(D^k(M)) \otimes \widetilde{H}_*(D^k(N)) \lra \widetilde{H}_*(D^k(M \oplus N)).
\end{equation*}
For $k$ being 1 or 2 there is only one interesting degree to consider. Using Theorem \ref{thm:SolomonTits}, for $k=1$ we have a $\mr{GL}(M) \times \mr{GL}(N)$-equivariant product 
\[\mr{St}(M) \otimes \mr{St}(N) \lra \mr{St}(M \oplus N),\] making $M \mapsto \mr{St}(M)[\dim(M)]$ into a non-unitary lax symmetric monoidal functor from $\cat{F}_\bF$ to graded abelian groups (the symmetry of course involves a Koszul sign rule). Similarly, using Proposition \ref{prop:smash-splitting-E2} for $k=2$ we have a $\mr{GL}(M) \times \mr{GL}(N)$-equivariant product 
\[\mr{St}^2(M) \otimes \mr{St}^2(N) \lra \mr{St}^2(M \oplus N),\]
making $M \mapsto \mr{St}^2(M)[2\dim(M)]$ into a non-unitary lax symmetric monoidal functor from $\cat{F}_\bF$ to graded abelian groups. Furthermore, the $k=2$ products are simply given by the tensor square of the $k=1$ products.

\begin{remark}For $k=1$, this product has also been defined by Miller, Nagpal, and Patzt in \cite[Section 2.3]{millernagpalpatzt}, who furthermore proved that $M \mapsto \mr{St}(M)$ is Koszul \cite[Theorem 1.5]{millernagpalpatzt}. Since the present paper was written, Miller, Patzt, and Wilson have identified its Koszul dual as $M \mapsto \mr{St}(M) \otimes \mr{St}(M)$ \cite[Theorem 6.3]{MillerPatztWilson}.
  \end{remark}

The following lemma therefore lets us calculate with the $k = 1$ or $k=2$ pairings. It uses the apartment classes introduced in Section \ref{sec:steinberg}.

\begin{lemma}\label{lem:product-apts}
Let $\gL = (L_1, L_2, \ldots, L_m)$ be a splitting of the module $M$ into ordered 1-dimensional subspaces, and $\gL' = (L'_1, L'_2, \ldots, L'_n)$ be such a splitting of $N$, defining apartments $a_\gL \in \mr{St}(M)$ and $a_{\gL'} \in \mr{St}(N)$. Then $a_\gL \cdot a_{\gL'} = a_{\gL''}$ where $\gL''$ is the concatenation of $\gL$ with $\gL'$, i.e.\
$\gL'' = (L_1 \oplus 0, \ldots, L_n \oplus 0, 0 \oplus L'_1, \ldots, 0 \oplus L'_{n'})$.
\end{lemma}
\begin{proof}
By definition $a_\gL$ is represented by the cycle $\sum_{\sigma \in \fS_m} \mr{sign}(\sigma) \cdot F_\sigma \in \widetilde{C}_{m}(D^1_\bullet(M))$ where $F_\sigma$ is the full flag (= full 1-dimensional lattice)
\[\varnothing \subset L_{\sigma(1)} \subset L_{\sigma(1)} \oplus L_{\sigma(2)} \subset \cdots \subset L_{\sigma(1)} \oplus \cdots \oplus L_{\sigma(m)} = M,\]
and similarly $a_{\gL'}$ is given by the analogous cycle $\sum_{\tau \in \fS_n} \mr{sign}(\tau) \cdot F'_\tau \in \widetilde{C}_{n}(D^1_\bullet(N))$. Applying the Eilenberg--Zilber map
\[\nabla \colon \widetilde{C}_m(D^1_\bullet(M)) \otimes \widetilde{C}_n(D^1_\bullet(N)) \lra \widetilde{C}_{m+n}(D^1_\bullet(M) \wedge D^1_\bullet(N))\]
to $a_\gL \otimes a_{\gL'}$ gives
\[\sum_{(\mu, \nu) \in \mathfrak{Sh}_{m,n}} \sum_{\sigma \in \mathfrak{S}_m} \sum_{\tau \in \mathfrak{S}_n} \mr{sign}(\mu, \nu)\mr{sign}(\sigma)\mr{sign}(\tau) \cdot s_\nu (F_{\sigma}) \otimes s_\mu(F'_\tau),\]
where $\mathfrak{Sh}_{m,n}$ is the set of $(m,n)$-shuffles and $s_\mu$ and $s_\nu$ denote the associated degeneracy maps. But on applying the pairing this is $\sum_{\upsilon \in \mathfrak{S}_{m+n}} \mathrm{sign}(\upsilon) \cdot F''_\upsilon$, where $F''_\upsilon$ is the full flag given by permuting by $\upsilon$ the splitting $(L_1, L_2, \ldots, L_m, L'_1, L'_2, \ldots, L'_n)$ of $M \oplus N$, i.e.\ it is the apartment associated to this ordered splitting.
\end{proof}

These products are natural in $V$ and $V'$, which implies a factorisation over coinvariants
\begin{equation*}
\mr{St}^2_\bk(V)_{\mr{GL}(V)} \otimes \mr{St}^2_\bk(V')_{\mr{GL}(V')} \lra \mr{St}^2_\bk(V\oplus V')_{\mr{GL}(V\oplus V')},
\end{equation*}
making the sum $\bk 1 \oplus \bigoplus_{n = 1}^\infty \mr{St}^2_\bk(\bF^n)_{\mr{GL}_n(\bF)}$
into a commutative $\bk$-algebra. As an application of Lemma \ref{lem:product-apts}, we may prove the following theorem:

\begin{theorem}\label{thm:DivPowAlg}
	Define elements $\gamma_n(x) \in \mr{St}^2(\bF^n)_{\mr{GL}(\bF^n)}$ to correspond to the integer $(-1)^{\frac{n(n-1)}{2}}$ under the isomorphism of Theorem \ref{thm:C}. Then the induced map
	\begin{equation*}
	\Gamma_\bZ[x] = \bigoplus_{n\geq 0} \bZ\{\gamma_n(x)\} \lra \bZ 1 \oplus \bigoplus_{n \geq 1} \mr{St}^2(\bF^n)_{\mr{GL}(\bF^n)},
	\end{equation*}
	from the divided power algebra on a generator $x$, is a ring isomorphism.
	
	More generally these maps give an isomorphism  $\Gamma_\bk[x] \cong \bk 1 \oplus \bigoplus_{n \geq 1} \mr{St}^2_\bk(\bF^n)_{\mr{GL}(\bF^n)}$ of $\bk$-algebras for any commutative ring $\bk$.
\end{theorem}

\begin{proof}
By Theorem \ref{thm:C} the pairing on the Steinberg module gives an isomorphism \[\mr{St}^2(\bF^m)_{\mr{GL}(\bF^m)} \overset{\sim}\lra \bZ,\]
which---for $\gL = (L_1, L_2, \ldots, L_m)$ any splitting of $\bF^m$ into ordered 1-dimensional subspaces---sends $a_\gL \otimes a_\gL$ to $m!$. Recall that we denote by $\gamma_m(x) \in \mr{St}^2(\bF^m)_{\mr{GL}(\bF^m)}$ the class which maps to $(-1)^{\frac{m(m-1)}{2}}$ under this isomorphism, so that \[(-1)^{\frac{m(m-1)}{2}} m! \gamma_m(x) = [a_\gL \otimes a_\gL].\]

If $\gL' = (L'_1, L'_2, \ldots, L'_n)$ is a splitting of $\bF^n$ into ordered 1-dimensional subspaces, the product
\[\mr{St}^2(\bF^m)_{\mr{GL}(\bF^m)} \otimes \mr{St}^2(\bF^n)_{\mr{GL}(\bF^n)} \lra \mr{St}^2(\bF^{m+n})_{\mr{GL}(\bF^{m+n})}\]
sends $[a_{\gL} \otimes a_{\gL}] \otimes  [a_{\gL'} \otimes a_{\gL'}]$ to $(-1)^{mn} [a_{\gL''} \otimes a_{\gL''}]$ by the previous lemma (and graded-commutativity), where $\gL''$ is the concatenation of $\gL$ with $\gL'$. Thus it sends 
\[(-1)^{\frac{m(m-1)}{2}} m! \gamma_m(x) \otimes (-1)^{\frac{n(n-1)}{2}} n! \gamma_n(x)\]
to 
\[(-1)^{mn+ \frac{(m+n)(m+n-1)}{2}} (m+n)! \gamma_{m+n} (x).\] 
The signs cancel out, and as both sides are torsion-free we may divide by $m! n!$, to see that it sends $\gamma_m(x) \otimes \gamma_n(x)$ to $\binom{m+n}{n} \gamma_{m+n}(x)$. This is by definition the multiplication of the divided power algebra $\Gamma_\bZ[x]$. The claim over a general commutative ring $\bk$ follows by tensoring.
\end{proof}

\subsection{$E_\infty$-homology}\label{sec:e-infty-homology}

To understand the consequences for $E_\infty$-homology of the above results about $E_2$-homology, we will consider maps of $E_\infty$-algebras into and out of $\gR$. Firstly, upon picking an element $\sigma \in \gR_\bZ(1)$ we obtain by adjunction a map of non-unital $E_\infty$-algebras
\[\epsilon \colon \gE_{\infty}(S_\bZ^{1,0} \sigma) \lra \gR_\bZ.\]
Here $S_\bZ^{1,0} \sigma$ denotes $\bZ \cdot \sigma$
placed in bidegree $(n,d) = (1,0)$.

On the other hand, we let $\gN$ denote the non-unital commutative algebra given by $\gN(0) = 0$ and $\gN(n) = \bZ$ otherwise, and multiplication maps $\gN(n) \otimes \gN(m) \to \gN(n+m)$ given by the canonical identification $\bZ \otimes_\bZ \bZ \overset{\sim}\lra \bZ$ for $n,m \geq 1$. In other words $\gN = (\underline{\ast}_{>0})_{\bZ}$, where $\underline{\ast}_{>0}$ is the unique non-unital $E_\infty$-algebra having $\underline{\ast}_{>0}(n) = \ast$ for $n > 0$ and $\underline{\ast}_{>0}(0) = \varnothing$.
of non-unital $E_\infty$-algebras
\[\gR_\bZ \lra \gN.\]
In this section we will investigate the effect of these maps on $E_2$- and $E_\infty$-homology. It will be convenient to work in terms of reduced $k$-fold bar constructions $\tilde{B}^{E_k}(-)$ of non-unital $E_k$-algebras. Recall that, by Theorem $E_k$.13.7, for any $E_k$-algebra $\gA$ this computes the $k$-fold reduced suspension of the derived $E_k$-indecomposables
\[\Sigma^k Q^{E_k}_\bL(\gA) \simeq \tilde{B}^{E_k}(\gA).\]

We start with an investigation of the free non-unital $E_\infty$-algebra $\gE_{\infty}(S_\bZ^{1,0} \sigma)$. By Theorem $E_k$.13.8 there is an equivalence
\[\tilde{B}^{E_2}(\gE_{\infty}(S_\bZ^{1,0} \sigma)) \simeq \gE_\infty(S^{1,2}_\bZ\bar{\bar{\sigma}})\]
of $E_\infty$-algebras (as we mentioned at the beginning of the last section, the discussion in Section $E_k$.13.6 endows the left-hand side with an $E_\infty$-structure). Thus the homology groups of the right hand side form a non-unital bigraded ring. Adding a unit and restricting our attention to bidegrees which are multiples of $(n,d) = (1,2)$, this identifies the ring
\[\bZ \oplus \bigoplus_{n \geq 1} H_{n,2n}(\tilde{B}^{E_2}(\gE_{\infty}(S_\bZ^{1,0} \sigma)))\]
with the polynomial ring $\bZ[\bar{\bar{\sigma}}]$.

Now let us consider $\gN$. Since this is weakly equivalent to the free non-unital $E_1$-algebra on a point in rank $1$, $H^{E_1}_{n,d}(\gN) = 0$ unless $(n,d) = (1,0)$ in which case it is $\bZ$. As explained in Lemma's~$E_k$.14.2 and $E_k$.14.3, the homology groups $H_{*,*}(B^{E_1}(\gN^+,\epsilon))$ of the unreduced bar construction of the unitalisation $\gN^+$ form a bigraded augmented commutative algebra, which can only be the exterior algebra $\Lambda_\bZ[\bar{\sigma}]$ with $\bar{\sigma}$ of bidegree $(n,d) = (1,1)$. We conclude that there is a bar spectral sequence of $\bZ$-algebras, given by
	\[E^1_{n,p,q} = \mr{Tor}_{p}^{\Lambda_\bZ[\bar{\sigma}]}(\bZ,\bZ)_{n,q} \Longrightarrow  H_{n,p+q}(B^{E_2}(\gN^+,\epsilon)).\]
	From this we conclude that $H_{*,*}(B^{E_2}(\gN^+,\epsilon))$ vanishes for $d \neq 2n$ and that the $E_2$-homology groups form a divided power algebra:
	\begin{equation}\label{eqn:e2-n} \bigoplus_{n \geq 0} H_{n,2n}(B^{E_2}(\gN^+,\epsilon)) \cong \Gamma_\bZ[\bar{\bar{\sigma}}]\end{equation} for a class $\bar{\bar{\sigma}}$ of bidegree $(1,2)$. This calculation is well-known and essentially goes back to Cartan \cite{CartanII}.

As $\gN$ can be obtained from $\gE_{\infty}(S_\bZ^{1,0} \sigma)$ by attaching $E_\infty$-cells of strictly positive dimensions, it is clear that the map
\[H_{1,2}(\tilde{B}^{E_2}(\gN)) \lra H_{1,2}(\tilde{B}^{E_2}(\gE_{\infty}(S_\bZ^{1,0} \sigma)))\]
is an isomorphism, sending the canonical generator $\bar{\bar{\sigma}}$ to the class of the same name. Therefore the factorisation $\gE_{\infty}(S_\bZ^{1,0} \sigma) \to \gR_\bZ \to \gN$ gives ring maps
\[\bZ[\bar{\bar{\sigma}}] \lra \bZ \oplus \bigoplus_{n \geq 1} H_{n,2n}(\tilde{B}^{E_2}(\gR_\bZ)) \lra \Gamma_\bZ[\bar{\bar{\sigma}}]\]
and by Theorem~\ref{thm:DivPowAlg} the graded ring in the middle is abstractly isomorphic to a divided power algebra on a class of degree $(1,2)$. This is only possible if the right-hand map is an isomorphism, so we have proved:

\begin{lemma}\label{lem:E2RToN} 
The map $\gR_\bZ \to \gN$ induces an isomorphism
\[H^{E_2}_{n,2n-2}(\gR_\bZ) \overset{\sim}\lra H^{E_2}_{n,2n-2}(\gN)\]
for each $n \geq 1$.
\end{lemma}

Along with the fact that $H^{E_2}_{n,2n-1}(\gN)=0$ by \eqref{eqn:e2-n}, the long exact sequence on $E_2$-homology for a pair shows that $\smash{H^{E_2}_{n,d}}(\gN, \gR_\bZ)=0$ for $d < 2n$.

\begin{corollary}\label{cor:EInfHomNrelR}
We have $H^{E_\infty}_{n,d}(\gN, \gR_\bZ)=0$ for $d < 2n$.
\end{corollary}

\begin{proof} It is enough to show that the $E_k$-homology of this pair vanishes in this range if the $E_{k-1}$-homology does, for $k-1 \geq 2$. This follows from Proposition $E_k$.14.5 applied to $\gR_\bZ \to \bN$ with $\rho(n) = \sigma(n) = 2n$.\end{proof}

This has an immediate implication for $E_\infty$-homology, which establishes Theorem \ref{thm:A} in the case of infinite fields, and Theorem \ref{thm:B}.

\begin{corollary}\label{cor:einfty-hom-indec} 
For any commutative ring $\bk$ we have $H^{E_\infty}_{n,d}(\gR_\bk)=0$ for $d < 2n-2$, and
	\[H^{E_\infty}_{n,2n-2}(\gR_\bk) = \begin{cases} \bk & \text{if $n=1$,} \\
	\bk/p & \text{if $n=p^k$ with $p$ prime,} \\
	0 & \text{otherwise.}\end{cases}\]
\end{corollary}

\begin{proof}
Let $\gN_\bk \coloneqq \gN \otimes_\bZ \bk$. By Corollary \ref{cor:EInfHomNrelR} and the universal coefficient theorem we have $H^{E_\infty}_{n,d}(\gN_\bk, \gR_\bk)=0$ for $d < 2n$, so by the long exact sequence for relative $E_\infty$-homology it is sufficient to calculate $\smash{H^{E_\infty}_{n,d}}(\gN_\bk)$ for $d \leq 2n-2$.

We consider the bar spectral sequence of Theorem $E_k$.14.1, computing the $E_{k+1}$-homology in terms of the $E_k$-homology. It is immediate from this that $H^{E_k}_{n,d}(\gN_\bk)=0$ for $d < 2n-2$ for all $k \geq 2$, including $k=\infty$. For the line $d=2n-2$, the groups $H^{E_3}_{n,2n-2}(\gN_\bk)$ are given by the rank $n$ piece of the indecomposables in the augmented $\bk$-algebra
\[\bigoplus_{n \geq 0} H_{n,2n}(B^{E_2}(\gN_\bk^+,\epsilon)) = \Gamma_\bk[\bar{\bar{\sigma}}].\] This is $\bk\{\gamma_n(\bar{\bar{\sigma}})\}/(\gamma_a(\bar{\bar{\sigma}})\cdot \gamma_b(\bar{\bar{\sigma}}) \text{ for } a+b=n \text{ with } a,b>0)$. Now $\gamma_a(\bar{\bar{\sigma}})\cdot \gamma_b(\bar{\bar{\sigma}}) = \binom{n}{a}\gamma_n(\bar{\bar{\sigma}})$, and
\[\gcd\left\{\textstyle \binom{n}{a} \, \middle\vert \, 0 < a < n\right\} = \begin{cases} 
	1 & \text{if $n=1$,} \\
	p & \text{if $n=p^k$ with $p$ prime,} \\
	1 & \text{otherwise,}\end{cases}\]
which shows that $H^{E_3}_{n,2n-2}(\gN_\bk)$ is given by the formula in the statement of the corollary. Taking further bar constructions does not change the groups along this line.
\end{proof}

\begin{remark}\label{rem:EInfHomOfN}
Continuing to write $\gN_\bk \coloneqq \gN \otimes_\bZ \bk$, it will be helpful to have a classical interpretation of the groups $H^{E_\infty}_{p,d}(\gN_{\bF_p})$ when $p$ is prime. Using the equivalences for $n \geq 1$
\[\gE_\infty(S^{1,0}_{\bF_p} \sigma)(n) \simeq \bF_p[B\fS_n], \qquad \gN_{\bF_p}(n) \simeq \bF_p,\]
the fact that $B\fS_n$ has trivial $\bF_p$-homology for $n<p$ means that \[H_{n,d}(\gN_{\bF_p},\gE_\infty(S^{1,0}_{\bF_p} \sigma)) = 0 \quad \text{ for $n<p$}.\]
The objects $\gN_{\bF_p}$ and $\gE_\infty(S^{1,0}_{\bF_p} \sigma)$ are $(\infty,0,0,\ldots)$-connective, and the map between them is $(\infty, \infty, \ldots, \infty, 0, 0, \ldots)$-connective, with $(p-1)$ infinities. By Proposition $E_k$.11.9 the Hurewicz maps
\[H_{p,d}(\gN_{\bF_p},\gE_\infty(S_{\bF_p}^{1,0} \sigma)) \lra H_{p,d}^{E_\infty}(\gN_{\bF_p},\gE_\infty(S_{\bF_p}^{1,0} \sigma))\]
are therefore isomorphisms. The connecting map of the long exact sequence on homology identifies the domain of this map with $\tilde{H}_{d-1}(B\fS_p;\bF_p)$. Similarly, the connecting map of the long exact sequence on $E_\infty$-homology identifies its target with $H^{E_\infty}_{p,d}(\gN_{\bF_p})$.

In particular the first non-trivial such reduced homology group of $\fS_p$ is is $H_{2p-3}(B\fS_p;\bF_p) \cong \bF_p$ generated by $\beta Q^1_p(\sigma)$ if $p$ is odd (or $Q^1_2(\sigma)$ if $p=2$). Thus $H_{p,2p-2}^{E_\infty}(\gN_{\bF_p}) = \bF_p$, encoding a reason that the operation $\beta Q^1_p(\sigma)$ if $p$ is odd (or $Q^1_2(\sigma)$ if $p=2$) vanishes in the homology of $\gN_{\bF_p}$.
\end{remark}

\section{Local and semi-local rings}
\label{sec:local-semi-local}

In this section we will extend the results of Section \ref{sec:GLfield} from the case of infinite fields to connected semi-local rings with infinite residue fields.  Recall that a \emph{semi-local ring} is a commutative ring with
finitely many maximal ideals. The most important examples are the \emph{local rings}, which are the commutative rings with a unique maximal ideal.
As we have mentioned earlier, if $\R$ is a connected semi-local ring then every finitely-generated projective $\R$-module is free \cite[\href{https://stacks.math.columbia.edu/tag/02M9}{Tag 02M9}]{stacks-project}. 

Let us summarise what we know and what remains to be proved.  Firstly, the isomorphism
\begin{equation*}
  H^{E_k}_{n,d}(\gR_\bk) \cong \bigoplus_{\substack{[M]\\r(M)=n}} \widetilde{H}_{n,d-k}(\tilde{D}^k(M) \hcoker \mr{GL}(M);\bk)
\end{equation*}
from Theorem \ref{thm:split-building-ek} holds for any $\R$ and any coefficient ring $\bk$, and when $\R$ is semi-local with infinite residue fields, the Nesterenko--Suslin comparison map
\begin{equation*}
  \widetilde{H}_*(\tilde{D}^{k}(M) \hcoker \mr{GL}(M);\bk)  \lra \widetilde{H}_*({D}^{k}(M) \hcoker \mr{GL}(M))
\end{equation*}
is an isomorphism for any $\bk$, by Theorem~\ref{thmcor:BlockvsFlag}.

In the case where $\R = \bF$ is an infinite field, we then identified $D^1(M)$ with the double suspension of the Tits building and used the Lee--Szczarba resolution to give a presentation of its homology, and for $k=2$ we used that a certain map $D^2(M) \to D^1(M) \wedge D^1(M)$ is an isomorphism.  The goal of this section is to provide adequate substitutes for these ingredients in the case $\R$ is a connected semi-local ring with infinite residue fields.

In Section \ref{sec:defin-stat} we state analogues of all these results when the infinite field $\bF$ is replaced by a connected semi-local ring $\R$ with infinite residue fields.  In Sections~\ref{sec:gener-posit-lemma}--\ref{sec:proofs} we then prove these statements, with the exception of the analogue of calculating of coinvariants of the tensor square of the Steinberg module (we formulate a conjecture about this in Section~\ref{sec:coinv-e_2-steinb}, which we prove in ranks $\leq 3$).

\subsection{Definitions and statements}
\label{sec:defin-stat}

We first establish the proper generalisations of the main definitions and the theorems about them from the infinite field case to the case of connected semi-local rings with infinite residue fields, before getting to the proofs of these generalisations in later subsections.

\subsubsection{Connectivity theorems}
\label{sec:defin-conn-theor}

The first order of business is to establish connectivity.  Let us immediately remark that if a $M$ is a projective $\R$-module of rank $d$ then $D^1(M)$ is the realisation of a simplicial set whose non-degenerate simplices all have dimension $\leq d$, so it is a $d$-dimensional CW-complex, and by a similar argument $D^2(M)$ is a $2d$-dimensional CW-complex.  The content of the following theorem is therefore the connectivity of these spaces.
\begin{theorem}\label{thm:Solomon--Tits-local}
  Let $\R$ be a connected semi-local ring with infinite residue fields, and let $M$ be a projective $\R$-module of rank $d$.
  \begin{enumerate}[(i)]
  \item\label{item:7} $D^1(M)$ is homotopy equivalent to a wedge of spheres of dimension $d$
  \item\label{item:8} $D^2(M)$ is homotopy equivalent to a wedge of spheres of dimension $2d$.
  \end{enumerate}
\end{theorem}
Just as in the field case there is a natural map $D^2(M) \to D^1(M) \wedge D^1(M)$, but it is not an isomorphism when $\R$ is not a field.  Therefore connectivity of $D^2(M)$ cannot be immediately deduced from connectivity of $D^1(M)$, but nevertheless we prove that it is  as highly connected as $D^1 \wedge D^1$.

\begin{definition}
  For $\R$ and $M$ as in Theorem~\ref{thm:Solomon--Tits-local}, and $\bk$ a commutative ring, define $\bk[\mr{GL}(M)]$-modules
  \[\mr{St}_\bk(M) \coloneqq \widetilde{H}_d(D^1(M);\bk) \quad \text{ and } \quad
    \mr{St}^2_\bk(M) \coloneqq \widetilde{H}_{2d}(D^2(M);\bk).\]
  For $\bk = \bZ$ we denote these simply by $\mr{St}(M)$ and $\mr{St}^2(M)$.
\end{definition}
Note that $\mr{St}_\bk(M) = \bk \otimes \mr{St}(M)$ and $\mr{St}^2_\bk(M) = \bk \otimes \mr{St}^2(M)$. The natural map $D^2(M) \to D^1(M) \wedge D^1(M)$ is injective, and also induces an injection $\mr{St}^2(M) \to \mr{St}(M) \otimes \mr{St}(M)$.  When $\R$ is not a field, it seems $\mr{St}^2(M)$ is the better object, and we shall work with that when generalising arguments from the field case.

Using these definitions and results, we now have
\begin{align*}
  H^{E_1}_{n,d}(\gR_{\bk}) &\cong \widetilde{H}_{d-(n-1)}\left(\mr{GL}_n(\R);\mr{St}_\bk(\R^n)\right)\\
  H^{E_2}_{n,d}(\gR_{\bk}) &\cong \widetilde{H}_{d-2(n-1)}\left(\mr{GL}_n(\R);\mr{St}^2_\bk(\R^n)\right),
\end{align*}
and in particular we obtain the vanishing of $E_1$-homology for $d < n-1$ and the vanishing of $E_2$-homology for $d < 2(n-1)$, just as we did in the case where $\R$ is an infinite field (cf.~Theorems \ref{thm:E1Homology} and \ref{thm:E2Homology}). Proceeding as in Section \ref{sec:e-infty-homology}, we deduce Theorem \ref{thm:A}:

\begin{corollary}\label{cor:thmA}
 For $\R$ as in Theorem~\ref{thm:Solomon--Tits-local}, and $\bk$ a commutative ring, we have $H^{E_\infty}_{n,d}(\gR_{\bk}) = 0$ for $d < 2n-2$.
\end{corollary}

However, we do not yet have a formula for $H^{E_\infty}_{n,2n-2}(\gR_\bk)$ analogous to the second part of Corollary~\ref{cor:einfty-hom-indec}.  To derive that formula by the same method we would need to prove $H^{E_2}_{n,2n-2}(\gR_{\bk}) \cong \bZ$ and that the product between different $n$ formed a divided power algebra.  We do not currently know how to prove this, but see Section~\ref{sec:coinvariants-mrst2} for partial results and a conjecture.

\subsubsection{Resolutions of Steinberg modules}
\label{sec:resol-steinb-modul}

Our strategy for constructing the resolutions can be paraphrased as follows.  Suppose we are given a pair of CW complexes $X \subset Y$ such that $X$ contains the $(m-1)$-skeleton of $Y$, $Y$ is contractible, and $X$ is homotopy equivalent to a wedge of $(m-1)$-spheres.  Then the relative cellular chain complex $C_*(Y,X)$ is concentrated in degrees $\ast \geq m$ and the relative homology $H_*(Y,X)$ is concentrated in degree $m$, so that there is an induced exact sequence
\begin{equation*}
  \dots \lra C_{m+1}(Y,X) \overset{\partial}\lra C_m(Y,X) \lra H_m(Y,X) \lra 0.
\end{equation*}
Lee and Szczarba's resolution (\cite[Section 3]{leeszczarba}), as used in Section \ref{sec:coinvariants-of}, was constructed from a $\mr{GL}_n(\bF)$-equivariant such pair  (numbers align as $m = n-2$) together with an equivariant homotopy equivalence between ``$X$'' and the Tits building $T(\bF^{n})$. We mimick this strategy to construct resolutions of $\mr{St}(M)$ and $\mr{St}^2(M)$ for modules over connected semi-local rings with infinite residue fields. It is also used in \cite{KahnSun}.

\begin{definition}
  Let $M$ be a finitely generated projective $\R$-module and $\sigma \subset M^\vee$ a finite set of elements of $M^\vee = \mr{Hom}_\R(M,\R)$.  Let us say that $\sigma$ \emph{intersects cleanly} if the cokernel of the evaluation map $M \to \R^\sigma$ is projective.
\end{definition}

\newcommand{\pphi}{f} 
\newcommand{\ppsi}{f'}

\begin{definition}\
  \begin{enumerate}[(i)]
  \item Let $E(M)$ be the simplicial complex whose vertices are surjections $\pphi \colon M \to \R$ and where $\sigma = \{\pphi_0, \ldots \pphi_p\}$ is a  simplex if every subset $\tau \subset \sigma$ intersects cleanly. Let $E(M)_\bullet$ be the corresponding simplicial set, i.e.\ $E(M)_p$ is the set of ordered $(p+1)$-tuples $(\pphi_0, \ldots, \pphi_p) \in (M^\vee)^{p+1}$ for which $\{\pphi_0, \ldots, \pphi_p\} < E(M)$.
  \item Let $E_0(M) < E(M)$ be the subcomplex whose $p$-simplices are those $\sigma < E(M)$ for which the evaluation map $M \to \R^\sigma$ is not injective, and let $E_0(M)_\bullet$ be the corresponding simplicial set.
  \item   Let $E^2(M)_{p,q} \subset E(M)_p \times E(M)_q$ be the subset consisting of those pairs $((\pphi_0, \ldots, \pphi_p),(\ppsi_0,\ldots, \ppsi_q))$ for which $\{\pphi_0, \ldots, \pphi_p,\ppsi_0,\ldots, \ppsi_q\} < E(M)$.
  \item Let $E^2_0(M)_{p,q} \subset E^2(M)_{p,q}$ be the subset for which either $\{\pphi_0, \ldots, \pphi_p\} < E_0(M)$ or $\{\ppsi_0, \ldots, \ppsi_q\} < E_0(M)$.
  \end{enumerate}
\end{definition}
The resolutions now come from the following space-level result.
\begin{theorem}\label{thm:Lee--Szczarba-local-1-and-2}
    There are weak equivalences
    \begin{align}
      \label{eq:16a} D^1(M) &\simeq \Sigma \frac{|E(M)_\bullet|}{|E_0(M)_\bullet|}\\
      \label{eq:16b} D^2(M) &\simeq \Sigma^2 \frac{|E^2(M)_{\bullet,\bullet}|}{|E^2_0(M)_{\bullet,\bullet}|}
    \end{align}
    through zig-zags of $\mr{GL}(M)$-equivariant based maps.
\end{theorem}

If $M$ has rank $d$, the complex of simplicial chains of $E(M)_\bullet$ relative to $E_0(M)_\bullet$ is concentrated in degrees $\geq d -1$, and gives rise to a resolution of $\mr{St}(M)$ by free $\bZ[\mr{GL}(M)]$-modules, generalising (up to transposing some matrices) the Lee--Szczarba resolution from the field case.  Similarly, the bisimplicial set $E^2(M)_{\bullet,\bullet}$ relative to $E^2_0(M)_{\bullet,\bullet}$ gives a bicomplex concentrated in bidegrees $(p,q)$ with $p,q \geq d -1$.  The associated total complex then gives a resolution of $\mr{St}^2(M)$ by free $\bZ[\mr{GL}(M)]$-modules, generalising the tensor product of the Lee--Szczarba resolution with itself.  In Section~\ref{sec:coinv-e_2-steinb} we spell this out in more detail.

\subsubsection{Coinvariants of $\mr{St}^2(M)$}
\label{sec:coinvariants-mrst2}

Let $\R$ be a semi-local ring and let $\bF = \R/\mathfrak{m}$ be one of its residue fields. For a finitely generated projective $\R$-module $M$, tensoring with $\bF$ defines a map
\begin{equation*}
  D^2(M) \lra  D^2(M \otimes_\R \bF)
\end{equation*}
and in turn a homomorphism
\begin{equation*}
  \mr{St}^2(M) \lra \mr{St}^2(M \otimes_\R \bF).
\end{equation*}
Both are equivariant with respect to the group homomorphism $\mr{GL}_\R(M) \to \mr{GL}_\bF(M \otimes_\R \bF)$, and we get an induced homomorphism of coinvariants
\begin{equation}\label{eq:25}
  (\mr{St}^2(M))_{\mr{GL}(M)} \lra (\mr{St}^2(M \otimes_\R \bF))_{\mr{GL}(M \otimes_\R \bF)} \overset{\cong}\lra \bZ,
\end{equation}
where the last isomorphism is the one from Theorem~\ref{thm:C}.  The composition is easily seen to be surjective: indeed, a pair of apartments in $M \otimes_\R \bF$ with precisely one flag in common may be chosen to lift to $\R$. 

\begin{theorem}\label{thm:proof-of-conjecture-rank-leq-3}
  The surjection~(\ref{eq:25}) is an isomorphism for any connected semi-local ring $\R$ and any projective $\R$-module $M$ of rank $\leq 3$.
\end{theorem}
The case where $M$ has rank 1 is obvious, but rank 2 and especially 3 require more work.  Motivated by these low-rank results and the case where $\R$ is a field, we formulate the following expectation.

\begin{conjecture}\label{conj:double-steinberg-local} 
  The surjection~(\ref{eq:25}) is an isomorphism for any connected semi-local ring $\R$ and any finitely generated projective $\R$-module $M$.
\end{conjecture}

If this conjecture is true, we can make all the same conclusions about $E_\infty$-homology when $\R$ is a connected semi-local ring with infinite residue fields as we did when it was an infinite field.  

\subsection{A general position lemma and the contractibility of $E(M)$}
\label{sec:gener-posit-lemma}

In the case of an infinite field $\bF$, some arguments used the infinitude of $\bF$ to make ``general position'' type argument, typically using that the complement of a finite union of proper subspaces of a vector space is non-empty. In the case of connected semi-local rings with infinite residue fields, we shall make similar arguments by base changing to each of the finitely many residue fields, but more care is needed.  In Lemma~\ref{lemprop:choose-enough-elements} below we present a generic such general position statement which will play an important role in establishing a homotopy colimit decomposition of $D^k(M)$ in the following subsection.

\begin{definition}
  Let $M$ be a finitely generated projective $\R$-module and let $P, Q \in \cat{Sub}(M)$.  We say that $P$ and $Q$ \emph{intersect cleanly} if the cokernel of the canonical map $M \to (M/P) \oplus (M/Q)$ is projective.
\end{definition}

A diagram chase in
\begin{equation*}
  \begin{tikzcd}
    0 \rar & P \cap Q \rar \dar[hook] & P \oplus Q \rar \dar[hook] & P + Q \rar \dar[hook] & 0\\
    0 \rar & M \rar{(\mathrm{id},-\mathrm{id})} & M \oplus M \rar{+} & M \rar & 0
  \end{tikzcd}
\end{equation*}
identifies the cokernel of the map $M \to (M/P) \oplus (M/Q)$ with $M/(P + Q)$, so the condition is equivalent to asking that the submodule $P+Q \subset M$ be a summand, i.e.\ that $P + Q \in \cat{Sub}(M)$.  The induced short exact sequence of cokernels shows that if $P$ and $Q$ intersect cleanly then $M/(P \cap Q)$ is a projective $\R$-module and hence $P \cap Q \in \cat{Sub}(M)$.

\begin{lemma}\label{P-and-Q-intersect-cleanly}
  Let $\tau < E(M)$ and let $\tau',\tau'' < \tau$ be two faces.  Let $P = \Ker(M \to \R^{\tau'})$ and $Q = \Ker(M \to \R^{\tau''})$ be the kernels of the evaluation maps.  Then $P$ and $Q$ intersect cleanly.
\end{lemma}
\begin{proof}
  Let us first remark that the defining assumption that $\Cok(M \to \R^\tau)$ be projective implies that the image of $M \to \R^\tau$ is a summand and hence projective.  This in turn implies that the kernel of $M \to \R^\tau$ is a summand of $M$.  Similarly when $\tau$ is replaced by $\tau'$, $\tau''$, $\tau' \cup \tau''$, and $\tau' \cap \tau''$.
  
  Now consider the commutative diagram
  \begin{equation*}
    \begin{tikzcd}
      0 \rar & M \dar{\pphi_{\tau' \cup \tau''}} \rar{(\mathrm{id},-\mathrm{id})} & M \oplus M \dar{\pphi_{\tau'} \oplus \pphi_{\tau''}} \rar{+} & M \dar{\pphi_{\tau' \cap \tau''}} \rar & 0\\
      0 \rar & \R^{\tau' \cup \tau''} \rar & \R^{\tau'} \oplus \R^{\tau''} \rar & \R^{\tau' \cap \tau''} \rar & 0
    \end{tikzcd}
  \end{equation*}
  where the two rows are exact and all three vertical maps are the evaluation maps. Observe that $\Ker(\pphi_{\tau'}) = P$, $\Ker(\pphi_{\tau''}) = Q$, and $\Ker(\pphi_{\tau' \cup \tau''}) = P \cap Q$.  The snake lemma gives an exact sequence of the form
  \begin{equation*}
  \begin{tikzcd} 0 \rar & P \cap Q \rar & P \oplus Q \rar
  \ar[draw=none]{d}[name=X, anchor=center]{}
  & \Ker(\pphi_{\tau' \cap \tau''}) \ar[rounded corners,
  to path={ -- ([xshift=2ex]\tikztostart.east)
  	|- (X.center) \tikztonodes
  	-| ([xshift=-2ex]\tikztotarget.west)
  	-- (\tikztotarget)}]{dll}[at end]{} & \\      
  & \Cok(\pphi_{\tau' \cup \tau''}) \rar & \Cok(\pphi_{\tau'}) \oplus \Cok(\pphi_{\tau''}) \rar & \Cok(\pphi_{\tau' \cap \tau''}) \rar & 0.\end{tikzcd}
  \end{equation*}
  where the cokernel terms are by assumption projective.  We have also explained why the kernel terms $P \cap Q$, $P \oplus Q$, $\Ker(\pphi_{\tau' \cap \tau''})$ are projective, so the short exact sequences from which the snake is spliced all split.  It follows that the image of $+\colon P \oplus Q \to M$ is a summand of $\Ker(\pphi_{\tau' \cap \tau''})$, but since the latter is a summand of $M$ we deduce that $P+Q \subset M$ is also a summand.
\end{proof}

\newcommand{\XA}{Z}

\begin{lemma}\label{lemprop:choose-enough-elements}
  Let $\R$ be a connected semi-local ring, $M$ a finitely generated projective $\R$-module and $V \subset M^\vee$ a submodule.  Let $\tau < E(M)$ be a simplex and assume there exists a simplex $\tau' < E(M)$ with $\tau < \tau'$ and $V = \mr{span}(V \cap \tau')$. Then for each maximal ideal $\mathfrak{m} \subset \R$ there exists a proper linear subspace $\XA_\mathfrak{m} \subset V/\mathfrak{m} V$ such that whenever $f \in V$ is an element satisfying $f + \mathfrak{m} V \not \in \XA_\mathfrak{m}$ for all maximal $\mathfrak{m}$, then the map
  \begin{equation*}
    M \xrightarrow{(f,\mathrm{ev})} \R \times \R^{\tau}
  \end{equation*}
  has projective cokernel.
\end{lemma}

\begin{proof}
  Recall that if an $\R$-module $X$ is presented by an exact sequence
  \begin{equation*}
    \R^{\oplus J} \xrightarrow{\phi} \R^t \to X \to 0,
  \end{equation*}
  with $n \in \bN$ and $J$ a (possibly infinite) set, then for each $s \in \bZ$ the \emph{Fitting ideal} $\mr{Fit}_s(X)\subset \R$ is defined as the ideal generated by the minors of size $(t-s)$ in the matrix for $\phi$, and that this ideal is independent of choice of presentation.  Then $X$ is projective of rank $d$ if and only if $\mr{Fit}_d(X) = \R$ and $\mr{Fit}_{d-1}(X) = 0$, see for instance \cite[\href{https://stacks.math.columbia.edu/tag/07Z6}{Tag 07Z6}]{stacks-project}.

  Let us define $t,d \in \N$ by
  \[t \coloneqq \mr{rk}(\R^\tau), \qquad d \coloneqq \mr{rk}(\Cok(\mr{ev} \colon M \to \R^\tau)).\]
  The number $d$ is well-defined because of the assumptions that the homomorphism $\mr{ev} \colon M \to \R^\tau$ has projective cokernel and that $\R$ is connected \cite[\href{https://stacks.math.columbia.edu/tag/00NV}{Tag 00NV}]{stacks-project}.  Let us also choose a basis (or just a generating set) for $M$ so that $\mr{ev}$ is represented by a matrix with $t$ rows.  Since $\Cok(\mr{ev} \colon M \to \R^\tau)$ is projective of rank $d$, we have $\mr{Fit}_d(\mr{Cok}(\mr{ev})) = \R$ and $\mr{Fit}_{d-1}(\mr{Cok}(\mr{ev})) = 0$.  In other words, the minors of size $(t-d)\times (t-d)$ generate $\R$ while those of size $(t-d+1) \times (t-d+1)$ all vanish.

  The matrix for $(f,\mr{ev}) \colon M \to \R \times \R^\tau$ is obtained from that of $\mr{ev}$ by adding one row, so in its matrix the minors of size $(t-d+2) \times (t-d+2)$ must all vanish, since determinants may be computed by expanding along a column, while those of size $(t-d) \times (t-d)$ must still generate $\R$, since they include the old generating set.  In other words,
  \begin{align*}
    \mr{Fit}_{d-1}(\Cok((f,\mr{ev}) \colon M \to \R \times \R^\tau)) & = 0\\
    \mr{Fit}_{d+1}(\Cok((f,\mr{ev}) \colon M \to \R \times \R^\tau)) & = \R,
  \end{align*}
  regardless of $f \in V$.  The assumptions on $\mr{ev}$ do not imply anything about the ideal generated by the minors of size $t - d + 1$, which by definition is the Fitting ideal
  \begin{equation*}
    \mr{Fit}_{d}(\Cok((f,\mr{ev}) \colon M \to \R \times \R^\tau)) \subset \R.
  \end{equation*}
  If this ideal is trivial then the cokernel of $(f,\mr{ev})$ is projective of rank $d+1$, and if it is all of $\R$ then that cokernel is projective of rank $d$.  If it is a proper non-zero ideal then the cokernel is not projective: this is what we must avoid.

  There are now two cases to consider.  The easy case is when $V \subset \mr{span}(\tau)$, since any $f \in V$ is then a linear combination of the remaining coordinates, from which one easily deduces $\Cok((f,\mr{ev})) \cong \R \oplus \Cok(\mr{ev} \colon M \to \R^\tau)$ which is then projective of rank $d + 1$ regardless of $f$.  We may then take $\XA_\mathfrak{m} = 0$ for all $\mathfrak{m}$.  In the rest of the argument we shall therefore assume $V \setminus \mr{span}(\tau) \neq \varnothing$.

  Since $V$ is spanned by $V \cap \tau'$, this assumption is equivalent to the existence of an $f_0 \in V \cap \tau'$ which is not in the span of $\tau$, and for such an $f_0$ the cokernel of $(f_0,\mr{ev}) \colon M \to \R \times \R^\tau$ is assumed projective.  Since $\mr{ev}$ and $(f_0,\mr{ev})$ both have projective cokernels, they also have projective images and these are summands.  The fact that $f_0$ is not in the span of $\tau$ implies that the obvious surjection from the image of $(f_0,\mr{ev})$ to the image of $\mr{ev}$ is not an isomorphism, so the rank of the image of $(f_0,\mr{ev})$ must be one larger than the rank of the image of $\mr{ev}$ and hence their cokernels have the same rank $d$.  Therefore $\mr{Fit}_d(\Cok((f_0,\mr{ev}))) = \R$.

  For a maximal ideal $\mathfrak{m} \subset \R$ the condition that $\mr{Fit}_d(\Cok((f,\mr{ev}))) \subset \mathfrak{m}$ is the condition that all minors of size $t-d+1$ vanish when reduced modulo $\mathfrak{m}$.  By definition of $d$, this is automatically true for the minors not involving $f$, while for the minors that do involve $f$ it is an $(\R/\mathfrak{m})$-linear condition on the matrix entries of $f$'s reduction to a linear map $M/\mathfrak{m} M \to \R/\mathfrak{m}$.  We now let $\XA_\mathfrak{m} \subset V /\mathfrak{m} V \subset (M/\mathfrak{m} M)^\vee$ be the subspace defined by these linear conditions.  It is a proper subspace because $f_0 + \mathfrak{m} V$ is not in it.

  Finally, if $f \in V$ satisfies $f \not \in \XA_\mathfrak{m}$ for all $\mathfrak{m}$, then the ideal $\mr{Fit}_d(\Cok((f,\mr{ev})))$ is not contained in any maximal ideal, and therefore must equal $\R$.  Together with $\mr{Fit}_{d-1}(\Cok((f,\mr{ev}))) = 0$, which holds for any $f \in V$, these conditions therefore imply that the cokernel of $(f,\mr{ev}) \colon M \to \R \times \R^\tau$ is projective of rank $d$, by op.cit.
\end{proof}

It is convenient to combine the conclusion of Lemma~\ref{lemprop:choose-enough-elements} with surjectivity of the reduction map
\begin{equation}\label{eq:22}
  V \lra \prod_{\mathfrak{m}} V/\mathfrak{m} V,
\end{equation}
where the product is over the finitely many maximal ideals (surjectivity follows surjectivity of $\R \to \prod_{\mathfrak{m}} \R/\mathfrak{m}$, which is an easy consequence of the Chinese remainder theorem).  Since $\XA_\mathfrak{m} \subset V/\mathfrak{m}V$ is a proper subspace, its complement is non-empty and therefore there exist $f$ such that $(f,\mr{ev}) \colon M \to \R \times \R^\tau$ has projective cokernel.  Using that the residue fields of $\R$ are all infinite, we can achieve this simultaneously for a finite collection of $\tau$'s.  This idea is used in the proofs of Proposition~\ref{prop:E-of-M-is-contractible} and Lemma~\ref{lemprop:contractible-space-of-sigmas} below.

\begin{lemma}\label{lem:coning-vertex}
  For any $\tau < E(M)$ the conclusion of Proposition \ref{lemprop:choose-enough-elements} holds for $V = M^\vee$.
\end{lemma}

\begin{proof}
  This is similar to Lemma~\ref{lemprop:choose-enough-elements}.  In the easy case where $\mr{ev} \colon M \to \R^\tau$ is injective, we can again set $\XA_\mathfrak{m} = 0$ for all $\mathfrak{m}$.  If not, we define $d \in \bN$ as before, and again have $\mr{Fit}_{d-1}(\Cok((f,\mr{ev}))) = 0$ and $\mr{Fit}_{d+1}(\Cok((f,\mr{ev}))) = \R$, regardless of $f \in V = M^\vee$.  The condition that $\mr{Fit}_d(\Cok((f,\mr{ev}))) \subset \mathfrak{m}$ defines a linear subspace $\XA_\mathfrak{m} \subset M^\vee/\mathfrak{m} M^\vee$, which is a proper subspace because it is contained in the image of $\mr{ev}/\mathfrak{m} \colon M^\vee/\mathfrak{m} M^\vee \to (\R/\mathfrak{m})^\tau$ whose cokernel has dimension $d >0$.  If the reduction of $f \in M^\vee$ avoids $\XA_\mathfrak{m}$ for all $\mathfrak{m}$, then $\mr{Fit}_d(\Cok((f,\mr{ev}))) = \R$, which implies that the cokernel of $(f,\mr{ev}) \colon M \to \R \times \R^\tau$ is projective of rank $d$.
\end{proof}

\begin{proposition}\label{prop:E-of-M-is-contractible}
  The simplicial complex $E(M)$ has contractible realisation when $M \neq 0$.
\end{proposition}
\begin{proof}
  Let $K < E(M)$ be a finite subcomplex, and apply Lemma~\ref{lemprop:choose-enough-elements}
    for each $\tau < K$ with $V = M^\vee$ (see Lemma~\ref{lem:coning-vertex}).  Over each maximal ideal, each such $\tau$ excludes a linear subspace of positive codimension, so it is possible to choose an element in the complement of their unions, and lift these choices to one element $f \in M^\vee$.  This element $f$ is a vertex of $E(M)$ whose star contains $K$.
\end{proof}

\subsection{A homotopy colimit decomposition of $D^k(M)$}

As a final preparation to proving the results stated in Section~\ref{sec:defin-stat} we shall establish a homotopy colimit decomposition
\begin{equation}\label{eq:10}
  \hocolim_{\sigma \in \mr{Simp}(E(M))} D^k(M)^\sigma \stackrel\simeq\lra D^k(M),
\end{equation}
where $\mr{Simp}(E(M))$ is the poset of simplices in the simplicial complex $E(M)$, and $D^k(M)^\sigma \subset D^k(M)$ is a certain subspace, see Definition~\ref{def:compatibility-consolidated} (\ref{item:6}) below.  By Proposition~\ref{prop:E-of-M-is-contractible}
it does not matter whether we interpret the domain as pointed or unpointed homotopy colimit.

The map~\eqref{eq:10} is the canonical map induced by the inclusions.  For each $\sigma$ the space $D^k(M)^\sigma$ is a finite CW complex, and is easier to analyse than the full $D^k(M)$; in particular, the delicacy of when sums and intersections of summands are again summands will have gone away.

The decomposition~\eqref{eq:10} holds for all $k \geq 1$, but it seems most useful when $k \leq 2$ and we shall only make use of it in those two cases.

\begin{definition}
  For a simplex $\sigma < E(M)$, let $\cat{Sub}(M)_\sigma \subset \cat{Sub}(M)$ be the subposet consisting of those $V \in \cat{Sub}(M)$ which may be written as the intersection of kernels of some of the vertices in $\sigma$.  That is, $\cat{Sub}(M)_\sigma$ consists of those $V$ satisfying
  \begin{equation*}
    V = \bigcap_{\substack{\phi \in \sigma\\ \phi \vert_{V} = 0}} \mr{Ker}(\phi).
  \end{equation*}
\end{definition}

Letting $(-)^\perp$ denote the annihilator, each subset $\sigma' \subset \sigma$ gives an element $(\mr{span}(\sigma'))^\perp \in \cat{Sub}(M)_\sigma$, and all elements are of this form.

\begin{definition}\
  \label{def:compatibility-consolidated}
  \begin{enumerate}[(i)]
  \item \label{item:4}
    For a full lattice $\phi \colon [p_1, \ldots, p_k] \to \cat{Sub}(M)$, we say that a simplex $\sigma < E(M)$ is \emph{compatible with $\phi$} if
    \begin{equation}
      \label{eq:23}
      \sum_{j = 1}^k \sum_{\substack{x: \\ x_j \leq a_j}}   \phi(x) \in \cat{Sub}(M)_\sigma
    \end{equation}
    for all $a \in [p_1,\ldots, p_k]$.
  \item \label{item:5}
    For a lattice $\phi\colon [p_1, \ldots, p_k] \to \cat{Sub}(M)$, let 
    \[\mr{Simp}(E(M))_\phi \subset \mr{Simp}(E(M))\]
    be the full subposet on the simplices $\sigma \in \mr{Simp}(E(M))$ which are compatible with $\phi$.
  \item \label{item:6}
    For a simplex $\sigma < E(M)$, let
    \begin{equation*}
      D^k(M)^\sigma_{p_1, \ldots, p_k} \subset D^k(M)_{p_1, \ldots, p_k}
    \end{equation*}
    be the based subset whose non-basepoint elements are the full lattices with which $\sigma$ is compatible.  Let $D^k(M)^\sigma$ be the $k$-fold geometric realisation of $D^k(M)^\sigma_{\bullet,\ldots, \bullet}$.
  \end{enumerate}
\end{definition}

The relation between $\sigma$ and $\phi$ expressed in~(\ref{item:4}) turns out to be equivalent to an a priori stronger looking compatibility condition, as we shall now show in Lemma~\ref{lem:compatible-strong}.  Taking $X = \{x \mid x \leq a\}$ in that lemma, we see in particular that $\phi(a) \in \cat{Sub}(M)_\sigma$ for all $a \in [p_1, \ldots, p_k]$.  Since $\cat{Sub}(M)_\sigma$ is likely not closed under taking sum of subspaces, condition~(\ref{eq:23}) is strictly stronger (for $k > 1$) than merely requiring $\phi(a) \in \cat{Sub}(M)_\sigma$ for all $a$.
\begin{lemma}\label{lem:compatible-strong} 
  A simplex $\sigma < E(M)$ is compatible with $\phi\colon [p_1, \ldots, p_k] \to \cat{Sub}(M)$ if and only if
  \begin{equation}\label{eq:24}
    \sum_{x \in X} \phi(x) \in \cat{Sub}(M)_\sigma
  \end{equation}
  for any subset $X \subset [p_1, \ldots, p_k]$.
\end{lemma}
\begin{proof}
  For the non-trivial direction let us first remark that we may replace $X$ by the larger set $\{x \mid \exists a \in X \text{ such that } x \leq a\}$ without changing the left hand side of~(\ref{eq:24}).  It therefore suffices to prove the claim when $X$ is downwards closed.  Choose a decomposition $M = \bigoplus_x M_x$ as in Lemma~\ref{lem:splittings-exist}, i.e.,\ so that $\phi(a) = \sum_{x \leq a} M_x \in\cat{Sub}(M)$ for all $a$.  Then $\sum_{x \in X} \phi(x) = \sum_{x \in X} M_x$ when $X$ is downwards closed, so we must show that $\sum_{x \in X} M_x \in \cat{Sub}(M)_\sigma$ for such $X$.  By assumption this holds when $X = X_a$, defined for $a \in [p_1, \ldots, p_k]$ as
  \begin{equation*}
    X_a = \{x \mid \exists i \text{ such that } x_i \leq a_i\}.
  \end{equation*}
  Since $X \mapsto \sum_{x \in X} M_x$, regarded as a function from subsets of $\{1, \ldots, p_1\} \times \ldots \times \{1, \ldots, p_k\}$ to submodules of $M$, preserves intersection, the claim holds for any $X$ which may be written as an intersection of $X_a$'s.  But any downwards closed $X$ is the intersection of $X_a$ over those $a$ with $(a_1+1,\ldots, a_k+1) \not\in X$.
\end{proof}

The condition of a full lattice defining a multi-simplex in $D^k(M)^\sigma_{p_1, \ldots, p_k}$ may be illuminated by ``dualising'' the lattice $\phi$ in the following way.
\begin{definition}
  Suppose $\phi\colon [p_1, \ldots, p_k] \to \cat{Sub}(M)$ be a full lattice, and choose a decomposition $M = \oplus_x M_x$ with $\phi(a) = \sum_{x \leq a} M_x$ as in Lemma~\ref{lem:splittings-exist}. Then $M^\vee \cong \oplus_x M^\vee_x$ is the induced decomposition of the linear dual, and we set $\phi^\vee(a) \coloneqq \sum_{x \geq a} M_x^\vee \subset M^\vee$.
\end{definition}
To see that the resulting functor $[p_1,\ldots, p_k] \to \cat{Sub}(M^\vee)^\mr{op}$
is independent of the chosen decomposition, note that $\phi^\vee(a_1 + 1,\ldots, a_k + 1)$ is the annihilator of the submodule~(\ref{eq:23}).  The following characterisation will be used in the proof of Lemma~\ref{lemprop:contractible-space-of-sigmas}, which is itself the key step in the proof of Proposition~\ref{propcor:resolve}.  
\begin{lemma}
  The condition that $\sigma < E(M)$ be compatible with a lattice $\phi$ is equivalent to requiring that for all $a$ the submodule $\phi^\vee(a) \subset M^\vee$ be spanned by $\sigma \cap \phi^\vee(a)$.
\end{lemma}
\begin{proof}
  As noted above, the summands
  \begin{align*}
    \sum_{j = 1}^k \sum_{x \text{ with } x_j \leq a_i}   \phi(x) & \in \cat{Sub}(M)\\
    \phi^\vee(a_1 + 1,\ldots, a_k + 1) & \in \cat{Sub}(M^\vee)
  \end{align*}
  are annihilators of each other.  Taking annihilator of summands interchanges intersection and sum.
\end{proof}

\begin{lemma}\label{lemprop:contractible-space-of-sigmas}
  For any full lattice $\phi\colon [p_1, \ldots, p_k] \to \cat{Sub}(M)$ the poset $\mr{Simp}(E(M))_\phi$ has contractible realisation.
\end{lemma}

\begin{proof}
  It suffices to prove that for any finite non-empty sub-poset $P \subset \mr{Simp}(E(M))_\phi$ there exists a $\sigma \in \mr{Simp}(E(M))_\phi$ such that $\tau \cup \sigma \in \mr{Simp}(E(M))_\phi$ for all $\tau \in P$.  Indeed, the inclusion $P \hookrightarrow \mr{Simp}(E(M))_\phi$ then induces a null homotopic map of geometric realisations since the face relations $\tau < \tau \cup \sigma$ and $\sigma < \tau \cup \sigma$ may be regarded as natural transformations.
    
  To produce such a $\sigma$ we choose its vertices inside $\phi^\vee(a) \subset M^\vee$, downwards inductively on $|a| \coloneqq a_1 + \ldots + a_k$.  Suppose $\sigma' < E(M)$ has been chosen such $\tau \cup \sigma'$ is a simplex in $E(M)$ for all $\tau \in P$ and that 
  \begin{equation*}
    \mr{span}(\sigma' \cap \phi^\vee(a)) = \phi^\vee(a)
  \end{equation*}
  whenever $|a| > n$. To make this hold also when $|a| = n$, we add to $\sigma'$ enough elements of $\phi^\vee(a) \in \cat{Sub}(M)$ to span that submodule.  We choose these new elements in a way that $(\text{new $\sigma'$}) \cup \tau$ is a simplex of $E(M)$ for all $\tau \in P$, which is possible by Lemma~\ref{lemprop:choose-enough-elements} since $\phi^\vee(a)$ is spanned by some elements of $\tau$ (and each face of each $\tau$ rules out only a proper linear subspace of $\phi^\vee(a) \otimes_\R A/\mathfrak{m}$ for each maximal ideal $\mathfrak{m}$.

  Doing this for each $a \in [p_1, \ldots, p_k]$ with $|a| = n$ results in a set $\sigma''$ of vertices of $E(M)$ containing $\sigma'$ as a subset, so that $\tau \cup \sigma''$ is a simplex of $E(M)$ for all $\tau \in P$, and that
  \begin{equation*}
    \mr{span}(\sigma'' \cap \phi^\vee(a)) = \phi^\vee(a)
  \end{equation*}
  whenever $|a| \geq n$.

  By induction we end up with a simplex $\sigma < E(M)$ such that $\phi^\vee(a) \subset M^\vee$ is spanned by the vertices of $\sigma$ that it contains, for each $a \in [p_1, \ldots, p_k]$. By the above characterisation, this means $\sigma$ is compatible with $\phi$.
\end{proof}

We are now finally ready for the main result of this subsection.
\begin{proposition}\label{propcor:resolve}
  The natural map~\eqref{eq:10} is a weak equivalence.
\end{proposition}
\begin{proof}
  The homotopy colimit may be calculated level-wise in the $k$-fold simplicial structure of $D^k(M)_{\bullet,\ldots, \bullet}$, so it suffices to prove that $\mr{hocolim}_\sigma\, D^k(M)^\sigma_{p_1,\ldots,p_k} \to D^k(M)_{p_1,\ldots,p_k}$ is a weak equivalence. The target is discrete with non-basepoint elements given by full lattices $\phi \colon [p_1,\ldots,p_k] \to \cat{Sub}(M)$. The preimage of $\phi$ in $D^k(M)^\sigma_{p_1,\ldots,p_k}$ is either empty if $\sigma$ is not compatible with $\phi$, or a single element otherwise, and hence the preimage of $\phi$ in $\mr{hocolim}_\sigma\, D^k(M)^\sigma_{p_1,\ldots,p_k}$ is the realisation of $\Simp(E(M))_\phi$. The claim then follows from Proposition~\ref{lemprop:contractible-space-of-sigmas}.
  \end{proof}

\subsection{Proofs of connectivity and resolution}
\label{sec:proofs}

We are finally ready to prove the results stated in Section~\ref{sec:defin-stat}.  The strategy is easy to explain: for the connectivity statements we prove that each of the spaces in the homotopy colimit decomposition is highly connected (and that the indexing category has contractible realisation), while for the weak equivalences we find a matching homotopy colimit decomposition of the right hand sides.

\subsubsection{The case $k=1$}

In this section we use the homotopy colimit decomposition~\eqref{eq:10} of $D^1(M)$ to prove the first half of Theorem~\ref{thm:Solomon--Tits-local}, as well as the first half of Theorem~\ref{thm:Lee--Szczarba-local-1-and-2}. These theorems concern the homotopy type of the spaces $D^1(M)$.  Their proofs start by studying the homotopy type of each $D^1(M)^\sigma$.
\begin{definition}\label{def:Q-of-sigma}
  Let $\sigma = \{f_0, \ldots, f_p\} < E(M)$ be a $p$-simplex.
  \begin{enumerate}[(i)]
  \item  Let $Q(\sigma)$ denote the simplex $\sigma$ regarded as a simplicial complex in its own right.
  \item Let $Q_0(\sigma) < Q(\sigma)$ be the subcomplex whose simplices are the non-empty $\tau < \sigma$ for which $\mr{span}(\tau) \neq M^\vee$, or equivalently that the canonical map $M \to \R^\tau$ is not injective.
  \end{enumerate}
\end{definition}

\begin{lemma}\label{lemma:hocolim-formula-D-one}
  There is a weak equivalence
  \begin{equation}\label{eq:11}
    D^1(M)^\sigma \simeq \Sigma \bigg( \frac{|Q(\sigma)|}{|Q_0(\sigma)|} \bigg)
  \end{equation}
  for each $\sigma < E(M)$.  This weak equivalence is induced by a zig-zag of weak equivalences which are natural in $\sigma$ and induces a weak equivalence
  \begin{equation}\label{eq:26}
    D^1(M) \simeq \Sigma \bigg( \hocolim_{\sigma \in \mr{Simp}(E(M))} \frac{|Q(\sigma)|}{|Q_0(\sigma)|} \bigg),
  \end{equation}
  which in turn is induced by a zig-zag of $\mr{GL}(M)$-equivariant maps.
\end{lemma}

The following auxiliary definitions are used in the proof of Lemma~\ref{lemma:hocolim-formula-D-one}.
\begin{definition}\label{def:P-of-sigma}
  Let $\sigma = \{f_0, \ldots, f_p\} < E(M)$ be a $p$-simplex.
  \begin{enumerate}[(i)]
  \item   Let $\mr{Simp}(Q(\sigma))^{\mr{op}}$ denote the poset of simplices of $Q(\sigma)$, i.e.\ the set of non-empty subsets of $\{f_0, \ldots, f_p\}$, ordered by reverse inclusion.  
  \item Let $\mr{Simp}(Q_0(\sigma))^\mr{op}$ be the subposet consisting of the non-spanning subsets.
  \item Let $P(\sigma) \subset \mr{Simp}(Q(\sigma))^\mr{op}$ denote the subposet consisting of those $\tau < \sigma$ for which $\sigma \cap \mr{span}(\tau) = \tau$.
  \item Let $P_0(\sigma) = P(\sigma) \cap \mr{Simp}(Q_0(\sigma))^\mr{op}$.
  \item Let $P_\varnothing(\sigma) = \{\varnothing\} \cup P(\sigma)$ and $P_{0\varnothing} = \{\varnothing\} \cup P_0(\sigma)$, regarded as subposets of the set of all subsets of $\{f_0, \ldots, f_p\}$, ordered by reverse inclusion.
  \end{enumerate}
\end{definition}

In the rest of this section we shall for readability use the same notation for geometric realisation of a simplicial complex, the thin geometric realisation of a simplicial set, and the geometric realisation of the nerve of a poset (e.g.\ we write $|P_0(\sigma)|$ instead of $|N(P_0(\sigma))|$).  The inclusions between the posets give a diagram
\begin{equation*}
  \begin{tikzcd}
    {|P_0(\sigma)|} \arrow[hookrightarrow, r] \arrow[hookrightarrow]{d}  & {|P_{0\varnothing}(\sigma)|} \arrow[hookrightarrow, d]\\
    {|P(\sigma)|} \arrow[hookrightarrow, r] & {|P_\varnothing(\sigma)|}.
  \end{tikzcd}
\end{equation*}

\begin{proof}[Proof of Lemma~\ref{lemma:hocolim-formula-D-one}]
  Let us consider the map
  \begin{equation}\label{eq:13}
    \begin{aligned}
      P_\varnothing(\sigma) & \lra \cat{Sub}(M)\\
      \tau & \longmapsto \bigcap_{f \in \tau} \mr{Ker}(f),
    \end{aligned}
  \end{equation}
  which is a poset map when $\cat{Sub}(M)$ is ordered by inclusion,
  inducing a map of simplicial sets, given in degree $p$ by
  \begin{equation}\label{eq:14}
    \begin{aligned}
      N_p(P_{\varnothing}(\sigma)) & \lra D^1(M)^\sigma_p\\
      (\tau_0 \supset \ldots \supset \tau_p) & \longmapsto \bigg( \bigcap_{f \in \tau_0} \mr{Ker}(f) \subset \ldots \subset \bigcap_{f \in \tau_p} \mr{Ker}(f) \bigg),
    \end{aligned}
  \end{equation}
  which is implicitly identified with the basepoint when that lattice isn't full, i.e.\ when either $\tau_p$ is non-empty or $\tau_0\subset M^\vee$ does not span.
  By definition of $D^1(M)^\sigma_p \subset D^1(M)_p$ this map~\eqref{eq:14} is surjective. 
  
  An element $\tau \in P_\varnothing(\sigma)$ may be recovered from the linear subspace $\mr{span}(\tau) = (\bigcap_{f \in \tau} \Ker(f))^\perp \subset M^\vee$ by intersecting with the finite subset $\sigma\subset M^\vee$, and hence the map~(\ref{eq:13}) is injective, from which it follows that~(\ref{eq:14}) is injective away from the inverse image of the basepoint.

 By definition of what it means for a lattice $[p] \to \cat{Sub}(M)$ to be non-full, we see that the inverse image of the basepoint is the subset
  \begin{equation*}
    N_p(P(\sigma)) \cup_{N_p(P_0(\sigma))} N_p(P_{0\varnothing}(\sigma)) \subset N_p(P_\varnothing(\sigma)), 
  \end{equation*}
  so we obtain homeomorphisms of based spaces
  \begin{equation*}
    \frac{|P_\varnothing(\sigma)|/|P_{0\varnothing}(\sigma)|}{|P(\sigma)|/|P_0(\sigma)|} \overset{\cong}\longleftarrow \frac{|P_\varnothing(\sigma)|}{|P_{0\varnothing}(\sigma)| \cup_{|P_0(\sigma)|} |P(\sigma)|} \overset{\cong}\lra |D^1(M)^\sigma_\bullet| = D^1(M)^\sigma.
  \end{equation*}
  The posets $P_0(\sigma)$ and $P_{0\varnothing}(\sigma)$ have contractible realisations because they have terminal object $\varnothing$, so we get a zig-zag
  \begin{equation*}
  \frac{|P_\varnothing(\sigma)|/|P_{0\varnothing}(\sigma)|}{|P(\sigma)|/|P_0(\sigma)|}
  \overset{\simeq}\longleftarrow
  C\left(\frac{|P(\sigma)|}{|P_0(\sigma)|} \hookrightarrow \frac{|P_\varnothing(\sigma)|}{|P_{0\varnothing}(\sigma)|}\right)
  \overset{\simeq}\lra
  C\left(\frac{|P(\sigma)|}{|P_0(\sigma)|} \to \ast\right) = \Sigma \frac{|P(\sigma)|}{|P_0(\sigma)|},
  \end{equation*}
  where $C$ denotes mapping cone and $\Sigma$ denotes suspension.

  To identify the right hand side, we observe that the poset map
  \begin{align*}
    \mr{Simp}(Q(\sigma))^\mr{op} & \lra P(\sigma)\\
    \tau & \longmapsto \sigma \cap \mr{span}(\tau)
  \end{align*}
  is left inverse to the inclusion of posets $P(\sigma) \hookrightarrow \mr{Simp}(Q(\sigma))^\mr{op}$.  It might not be right inverse, but the inclusion $\tau \subset \sigma \cap \mr{span}(\tau)$ may be regarded as a natural transformation from the identity, so the induced map $|Q(\sigma)| \to |P(\sigma)|$ is in fact a deformation retraction.  The same argument applies to $|Q_0(\sigma)| \to |P_0(\sigma)|$, so we get
  \begin{equation*}
    \frac{|Q(\sigma)|}{|Q_0(\sigma)|} \overset{\simeq}\lra \frac{|P(\sigma)|}{|P_0(\sigma)|},
  \end{equation*}
  naturally in $\sigma \in \mr{Simp}(E(M))$.  Combining with the other homotopy equivalences and homeomorphisms, we deduce~\eqref{eq:11}.

  Using that homotopy colimits commute with suspension (since $|E(M)|$ is contractible), we get
  \begin{align*}
    \Sigma \bigg( \hocolim_{\sigma \in \mr{Simp}(E(M))} \frac{|Q(\sigma)|}{|Q_0(\sigma)|} \bigg) & \simeq
                                                                                                        \hocolim_{\sigma \in \mr{Simp}(E(M))} \Sigma \frac{|Q(\sigma)|}{|Q_0(\sigma)|} \\
                                                                                                      & \simeq \hocolim_{\sigma \in \mr{Simp}(E(M))} \Sigma \frac{|P(\sigma)|}{|P_0(\sigma)|}\\
    & \simeq \hocolim_{\sigma \in \mr{Simp}(E(M))} \frac{|P_\varnothing(\sigma)|/|P_{0\varnothing}(\sigma)|}{|P(\sigma)|/|P_0(\sigma)|}\\
    & \cong \hocolim_{\sigma \in \mr{Simp}(E(M))} |D^1(M)^\sigma_\bullet|
      = \hocolim_{\sigma \in \mr{Simp}(E(M))} D^1(M)^\sigma \\
    & \simeq D^1(M),
  \end{align*}
  where in the last step we used Proposition~\ref{propcor:resolve}.  It is clear that spelling out all the maps involved in this string of equivalences will lead to a zig-zag of equivalences which are $\mr{GL}(M)$-equivariant.
\end{proof}

\begin{proof}[Proof of Theorem~\ref{thm:Solomon--Tits-local}.(\ref{item:7})]
  Since $\R$ is connected, any flag $0 = V_0 \subset \ldots \subset V_p = M$ with $p \geq d+1$ must involve a non-proper inclusion, so any element of $D^1(M)_p$ with $p > d$ must be degenerate.  Therefore $D^1(M)$ is a $d$-dimensional CW complex and it suffices to see that it is $(d-1)$-connected.

  For any $\sigma < E(M)$ the simplicial subcomplex $Q_0(\sigma) < Q(\sigma)$ contains the entire $(d-2)$-skeleton, since a spanning set must contain at least $d$ vectors.  Therefore the quotient $|Q(\sigma)|/|Q_0(\sigma)|$ is $(d-2)$-connected and hence the pointed homotopy colimit over all $\sigma < E(M)$ is also $(d-2)$-connected.  The suspension is therefore $(d-1)$-connected.
\end{proof}

\begin{proof}[Proof of Theorem~\ref{thm:Lee--Szczarba-local-1-and-2}, \eqref{eq:16a}]
  Recall that to a simplicial complex $X$ there is an associated simplicial set $X_\bullet$ where $X_p$ is the set of maps of simplicial complexes $\Delta^p \to X$, i.e.\ ordered tuples $(v_0, \ldots, v_p)$ of (not necessarily distinct)
    vertices of $X$ such that $\{v_0, \ldots, v_p\}$ is a simplex of $X$.  The natural map $|X_\bullet| \to |X|$ is a weak equivalence for any simplicial complex $X$; this can be proven using \cite[(1.1) (1.2)]{SuslinEquivalence}.
  
  For a simplex $\sigma < E(M)$, let $Q(\sigma)_\bullet$ and $Q_0(\sigma)_\bullet$ be the simplicial sets corresponding in this way to the simplicial complexes $Q(\sigma)$ and $Q_0(\sigma)$.  Then $D^1(M)$ is homotopy equivalent to the suspension of the based space
  \begin{equation*}
    \hocolim_{\sigma \in \mr{Simp}(E(M))} \frac{|Q(\sigma)_\bullet|}{|Q_0(\sigma)_\bullet|}.
  \end{equation*}
  Now the evident injections $Q(\sigma)_p \hookrightarrow E(M)_p$ and $Q_0(\sigma)_p \hookrightarrow E_0(M)_p$ induce maps
  \begin{equation}\label{eq:27}
    \begin{aligned}
    \hocolim_{\sigma \in \mr{Simp}(E(M))} Q(\sigma)_p & \lra E(M)_p\\
    \hocolim_{\sigma \in \mr{Simp}(E(M))} Q_0(\sigma)_p & \lra E_0(M)_p,
  \end{aligned}
  \end{equation}
  which are weak equivalences: the fiber over a $p$-simplex $(f_0, \ldots, f_p) \in E(M)_p$ may be identified with the star of $\{f_0, \ldots, f_p\} < E(M)$.  Using that geometric realisation commutes with homotopy colimits, we deduce~(\ref{eq:16a}).
\end{proof}

\subsubsection{The case $k=2$}
\label{sec:k-equals-two}

Recall from Lemma~\ref{lem:smash-of-E-one-buildings} that we have an injection
\begin{equation}\label{eq:18}
  D^2(M)_{p_1,p_2} \hookrightarrow D^1(M)_{p_1} \wedge D^1(M)_{p_2},
\end{equation}
induced by sending a full 2-dimensional lattice $\phi\colon [p_1,p_2] \to \cat{Sub}(M)$ to the pair of full 1-dimensional lattices
\begin{align*}
  & (0 = \phi(0,p_2) \subset \ldots \subset \phi(p_1,p_2) = M),\\
  & (0 = \phi(p_1,0) \subset \ldots \subset \phi(p_1,p_2) = M).
\end{align*}
This injection is likely not surjective: given a pair of 1-dimensional full lattices $\phi_1\colon [p_1] \to \cat{Sub}(M)$ and $\phi_2\colon [p_2] \to \cat{Sub}(M)$, the only possibility for a full lattice $\phi\colon [p_1, p_2] \to \cat{Sub}(M)$ corresponding to $(\phi_1,\phi_2)$ under~(\ref{eq:18}) would be $\phi(a_1,a_2) = \phi_1(a_1) \cap \phi_2(a_2)$, but that formula need not define a lattice $\phi\colon [p_1,p_2] \to \cat{Sub}(M)$ when $\R$ is not a field.  On closer inspection there are two phenomena that may prevent this: either that the submodule $\phi_1(a_1) \cap \phi(a_2) \subset M$ is not a summand for some $(a_1,a_2) \in [p_1,p_2]$, or that the submodule
\begin{equation*}
  (\phi_1(a_1) \cap \phi_2(b_2)) +   (\phi_1(b_1) \cap \phi_2(a_2)) \subset (\phi(b_1,b_2))
\end{equation*}
is not a summand for some $(a_1,a_2) \leq (b_1,b_2) \in [p_1,p_2]$.  As we will see next, both issues will be resolved if we restrict attention to lattices consisting of submodules made out of kernels of evaluation maps $M \to \R^\tau$ as $\tau$ runs through faces of a fixed simplex of $E(M)$.

\begin{proposition}\label{prop:D-two-sigma-vs-D-one-smash-D-one}
  For any $\sigma < E(M)$ the injection~(\ref{eq:18}) restricts to a bijection
  \begin{equation}
    D^2(M)^\sigma_{p_1,p_2} \hookrightarrow D^1(M)^\sigma_{p_1} \wedge D^1(M)^\sigma_{p_2},
  \end{equation}
\end{proposition}
\begin{proof}
  We must prove surjectivity.  Given two 1-dimensional lattices $\phi_1\colon [p_1] \to \cat{Sub}(M)$ and $\phi_2\colon [p_1] \to \cat{Sub}(M)$ we set $\phi(a_1,a_2) = \phi_1(a_1) \cap \phi_2(a_2)$.  Since $\phi(a_1,a_2) = \Ker(M \to \R^\tau)$ for some $\tau < \sigma$, this indeed defines a summand $\phi(a_1,a_2)$.  For $(a_1,a_2) \leq (b_1,b_2)$ we may write $P = \phi(a_1,b_2) = \Ker(M \to \R^{\tau'})$ and $Q = \phi(b_1,a_2) = \Ker(M \to \R^{\tau''})$ for some $\tau',\tau'' < \sigma$.  The fact that $P + Q \subset M$ is a summand then follows from Lemma~\ref{P-and-Q-intersect-cleanly}.  Therefore it is also a summand of the summand $\phi(b_1,b_2)$, so that $P+ Q \hookrightarrow \phi(b_1,b_2)$ is the inclusion of a summand.  This proves that $\phi\colon [p_1,p_2] \to \cat{Sub}(M)$ satisfies the lattice condition.
\end{proof}

\begin{proof}[Proof of Theorem~\ref{thm:Solomon--Tits-local}.(\ref{item:8})]
  We have 
  \begin{equation*}
    D^2(M)^\sigma \overset{\cong}\lra D^1(M)^\sigma \wedge D^1(M)^\sigma \simeq \left(\Sigma \frac{|Q(\sigma)|}{|Q_0(\sigma)|}\right) \wedge \left(\Sigma \frac{|Q(\sigma)|}{|Q_0(\sigma)|}\right),
  \end{equation*}
  and we already observed that $\Sigma\frac{|Q(\sigma)|}{|Q_0(\sigma)|}$ is $(d-1)$-connected.  The smash product is therefore $(2d-1)$-connected and hence by taking homotopy colimit over all $\sigma$ we deduce from Proposition~\ref{propcor:resolve} that $D^2(M)$ is also $(2d-1)$-connected.  
\end{proof}

\begin{proof}[Proof of Theorem~\ref{thm:Lee--Szczarba-local-1-and-2}, \eqref{eq:16b}]
  Recall that $E^2(M)_{p,q} \subset E(M)_p \times E(M)_q$ is the subset consisting of those pairs of ordered tuples whose union (of underlying unordered sets of vertices) is a simplex of $E(M)$.  In other words, it is the union over all $\sigma < E(M)$ of the image of the inclusion map
  \begin{equation*}
    Q(\sigma)_p \times Q(\sigma)_q \hookrightarrow E^2(M)_{p,q} \subset E(M)_p \times E(M)_q.
  \end{equation*}
  These maps induce a weak equivalence
  \begin{equation*}
    \hocolim_{\sigma \in \mr{Simp}(E(M))} Q(\sigma)_p \times Q(\sigma)_q \overset{\simeq}\lra E^2(M)_{p,q},
  \end{equation*}
  for each $p$ and $q$ because the fibers are contractible, as in the proof of~\eqref{eq:27}, and this weak equivalence restricts to a weak equivalence
  \begin{equation*}
    \hocolim_{\sigma \in \mr{Simp}(E(M))} \left(Q_0(\sigma)_p \times Q(\sigma)_q \right) \cup \left(Q(\sigma)_p \times Q_0(\sigma)_q \right) \overset{\simeq}\lra E^2_0(M)_{p,q}.
  \end{equation*}
  Using once more that $|E(M)|$ is contractible, taking homotopy colimits of the quotients leads to a weak equivalence of pointed sets
  \begin{equation*}
    \hocolim_{\sigma \in \mr{Simp}(E(M))} \frac{Q(\sigma)_p \times Q(\sigma)_q}{(Q_0(\sigma)_p \times Q(\sigma)_q) \cup (Q(\sigma)_p \times Q_0(\sigma)_q)} \overset{\simeq}\lra \frac{E^2(M)_{p,q}}{E^2_0(M)_{p,q}}.
  \end{equation*}
  
  For each $\sigma$ and each $p$, we have an isomorphism of based sets
  \begin{equation*}
    \frac{Q(\sigma)_p}{Q_0(\sigma)_p} \wedge \frac{Q(\sigma)_q}{Q_0(\sigma)_q} = \frac{Q(\sigma)_p \times Q(\sigma)_q}{(Q(\sigma)_p \times Q_0(\sigma)_q) \cup (Q_0(\sigma)_p \times Q(\sigma)_q)}.
  \end{equation*}
  Using that geometric realisation commutes with homotopy colimit, we then get
  \begin{equation*}
    \hocolim_{\sigma \in \mr{Simp}(E(M))} \frac{|Q(\sigma)_\bullet|}{|Q_0(\sigma)_\bullet|} \wedge \frac{|Q(\sigma)_\bullet|}{|Q_0(\sigma)_\bullet|} \simeq \frac{|E^2(M)_{\bullet,\bullet}|}{|E^2_0(M)_{\bullet,\bullet}|}.
  \end{equation*}
  On the other hand we may argue similarly to the case $k=1$ that
  \begin{align*}
    D^2(M) & \simeq \hocolim_{\sigma \in \mr{Simp}(E(M))} D^2(M)^\sigma \cong \hocolim_{\sigma \in \mr{Simp}(E(M))} D^1(M)^\sigma \wedge D^1(M)^\sigma \\
           & \simeq
             \hocolim_{\sigma \in \mr{Simp}(E(M))} \left(\Sigma \frac{|Q(\sigma)_\bullet|}{|Q_0(\sigma)_\bullet|} \right) \wedge \left(\Sigma \frac{|Q(\sigma)_\bullet|}{|Q_0(\sigma)_\bullet|}\right)
            \simeq
             \Sigma^2 \frac{|E^2(M)_{\bullet,\bullet}|}{|E^2_0(M)_{\bullet,\bullet}|},
  \end{align*}
  where the last step used that homotopy colimit commutes with suspension.
\end{proof}

\subsection{Coinvariants of the $E_2$-Steinberg module}
\label{sec:coinv-e_2-steinb}

This section is the analogue for connected semi-local rings with infinite residue fields of Section \ref{sec:coinvariants-of}. Let $M$ be a projective $\R$-module of rank $d$.  Theorems~\ref{thm:Solomon--Tits-local} and~\ref{thm:Lee--Szczarba-local-1-and-2} provide a presentation of the $\mr{GL}(M)$-module $\mr{St}^2(M)$:
\begin{equation*}
  \Z\left[\frac{E^2(M)_{d,d-1}}{E_0^2(M)_{d,d-1}}\right] \oplus
  \Z\left[\frac{E^2(M)_{d-1,d}}{E_0^2(M)_{d-1,d}}\right] \overset\partial\lra
  \Z\left[\frac{E^2(M)_{d-1,d-1}}{E_0^2(M)_{d-1,d-1}}\right] \lra \mr{St}^2(M) \to 0.
\end{equation*}
We shall use this to compute the $\mr{GL}(M)$-coinvariants of $\mr{St}^2(M)$ when $d \leq 3$.

Given a homomorphism $f = (f_1, \dots, f_d) \colon M \to \R^d$ and another homomorphism $f' = (f'_1, \dots, f'_d) \colon M \to \R^d$ we obtain an element \[(f',f) \in \Z\left[\frac{E^2(M)_{d-1,d-1}}{E_0^2(M)_{d-1,d-1}}\right],\] provided that $\{f'_1, \dots, f'_d, f_1, \dots, f_d\} < E(M)$, i.e.\ that each subset
\begin{equation*}
  \sigma \subset \{f'_1, \dots, f'_d, f_1, \dots, f_d\} \subset M^\vee
\end{equation*}
intersects cleanly.  Let us recall that this means by definition that $\mr{ev}: M \to \R^\sigma$ has projective cokernel, and is equivalent to $\mr{span}(\sigma) \subset M^\vee$ being a summand.  For each such $(f',f)$ we obtain an element of $\mr{St}^2(M)$ which we shall denote $[f'] \otimes [f]$.  By a slight abuse of language we shall in this section say that ``$[f'] \otimes [f]$ is defined'' for the condition that $\{f'_1, \dots, f'_d, f_1, \dots, f_d\} < E(M)$. 

These elements $[f'] \otimes [f] \in \mr{St}^2(M)$ generate as an abelian group, but are not linearly independent.  Firstly, they satisfy $[f'] \otimes [f] = 0$ unless both $f'$ and $f$ are isomorphisms.  Secondly, they satisfy
\begin{equation}\label{eqn:rel-semi-local-first}
  \begin{aligned}
    & \sum_{i = 0}^d (-1)^i [(f'_0, \dots, \widehat{f'_i}, \dots, f'_d)] \otimes [f] = 0,\\
    & \text{provided $\{f'_0, \dots, f'_d, f_1, \dots, f_d\} < E(M)$,}
  \end{aligned}
\end{equation}
and thirdly they satisfy
\begin{equation}\label{eqn:rel-semi-local-second}
  \begin{aligned}
    & \sum_{i = 0}^d (-1)^{i} [f'] \otimes [(f_0, \dots, \widehat{f_i}, \dots, f_d)] = 0,\\
    & \text{provided $\{f'_1, \dots, f'_d, f_0, \dots, f_d\} < E(M)$.}
\end{aligned}
\end{equation}
Much of the difficulty when $\R$ is not a field arises from controlling whether the conditions (``provided \dots'') apply.  Neglecting this may lead to expressions with undefined terms, but fortunately will not lead a meaningful but untrue equation:
\begin{lemma}\label{lemma:fortunately}
  Let $f' \colon  M \to \R^d$ and $f = (f_1, \dots, f_d) \colon M \to \R^d$, and let $f_0 = a_1 f_1 + \dots +a_k f_k$ for some $k \in \{1, \dots, d\}$ and $a_1, \dots, a_k \in \R$.  Then the relation
  \begin{equation*}
    \sum_{i = 0}^k (-1)^i [f'] \otimes [(f_0, \dots, \widehat{f_i}, \dots, f_d)] = 0
  \end{equation*}
  holds, provided all its terms are defined.  Similarly when $f'_0 = a'_1 f'_1 + \dots + a_k'f'_k$.
\end{lemma}
\begin{proof}
  We must verify that any subset $\sigma \subset \{f'_1, \dots, f'_d, f_0, \dots, f_d\}$ spans a summand of $M^\vee$.  If $\{f_0, \dots, f_k\} \subset \sigma$ then $\mr{span}(\sigma) = \mr{span}(\sigma \setminus \{f_0\})$, so we are reduced to considering the case where $f_i \not\in \sigma$ for some $i \in \{0, \dots, k\}$.  That case follows from what it means for $[f'] \otimes [(f_0, \dots, \widehat{f_i}, \dots, f_d)]$ to be defined.
\end{proof}

\newcommand{\iiota}{\iota}
\newcommand{\pppsi}{\psi}

In this section we are interested in the coinvariants, for which we obtain the presentation
\[
  \parbox{3.5cm}{\centering $\Z\left[\tfrac{E^2(M)_{d,d-1}}{E_0^2(M)_{d,d-1}}\right]_{\mr{GL}(M)}$ \\ $\oplus$ \\
  $\Z\left[\tfrac{E^2(M)_{d-1,d}}{E_0^2(M)_{d-1,d}}\right]_{\mr{GL}(M)}$} \overset\partial\lra
  \Z\left[\tfrac{E^2(M)_{d-1,d-1}}{E_0^2(M)_{d-1,d-1}}\right]_{\mr{GL}(M)} \lra
  \mr{St}^2(M)_{\mr{GL}(M)} \to 0.
\]
If we fix a particular isomorphism $\iiota = (\iiota_1, \dots, \iiota_d) \colon M \xrightarrow{\sim} \R^d$, this amounts to imposing the extra relations
\begin{equation*}
  [f'] \otimes [f] = [\iiota] \otimes [f \circ (f')^{-1} \circ \iiota].
\end{equation*}
Notice that one side of this relation is defined if and only if the other is.  Writing $\pppsi = f \circ (f')^{-1} \in \mr{GL}_d(\R)$, we see that the coinvariants are generated as an abelian group by the elements
\begin{equation*}
  [\iiota] \otimes [\pppsi \iiota] \in \mr{St}^2(M)_{\mr{GL}(M)},
\end{equation*}
as $\pppsi \in \mr{GL}_d(\R)$ ranges over those matrices for which $[\iiota] \otimes [\pppsi \iiota]$ is defined, i.e.\ those for which the union of any subset of the standard basis and any subset of the set of rows of $\pppsi$ spans a summand.  Following conventions in this section, we shall write $\pppsi = (\pppsi_1, \dots, \pppsi_d) \colon \R^d \to \R^d$ with $\pppsi_i \colon \R^d \to \R$ the $i$th coordinate of $\pppsi$.  With respect to the standard basis, $\pppsi_i$ is the $i$th row of $\pppsi$.  (Thus $\pppsi$ is similar to the matrix $\phi$ appearing in Section~\ref{sec:coinvariants-of}, but the direct comparison is between  $\pppsi$ and the transpose of the $\phi$.)

We have analogues of the relations from Section~\ref{sec:coinvariants-of}.
\begin{lemma}
  Let $\iiota: M \to \R^d$ be a fixed isomorphism, as above.  Then the following relations hold among the elements $[\iiota] \otimes [\pppsi \iiota] \in \mr{St}^2(M)$.
  \begin{enumerate}[(I)]
  \item Permuting rows: if $\pppsi \in \mr{GL}_d(\R)$ and $[\iiota] \otimes [\pppsi  \iiota]$ is defined, and $\pppsi'  = \sigma \pppsi \in \mr{GL}_d(\R)$ is defined by left multiplication with a permutation matrix $\sigma \in \fS_d \subset \mr{GL}_d(\R)$, then $[\iiota] \otimes [\pppsi' \iiota]$ is also defined, and 
    \begin{equation*}
      [\iiota] \otimes [\pppsi' \iiota] = \mr{sign}(\sigma) [\iiota] \otimes [\pppsi \iiota].
    \end{equation*}
  \item Scaling rows: if $\pppsi \in \mr{GL}_d(\R)$ and $[\iiota] \otimes [\pppsi  \iiota]$ is defined, and $\pppsi' \in \mr{GL}_d(\R)$ is defined by $\pppsi' = (\alpha_1 \pppsi_1, \dots, \alpha_d \pppsi_d)$ for some $\alpha_i \in \R^\times$, then $[\iiota] \otimes [\pppsi' \iiota]$ is also defined, and 
    \begin{equation*}
      [\iiota] \otimes [\pppsi' \iiota] =  [\iiota] \otimes [\pppsi \iiota].
    \end{equation*}
  \item Adding rows: if $\pppsi \in \mr{GL}_d(\R)$, $1 \leq s < t \leq d$ and $\pppsi', \pppsi'' \in \mr{GL}_d(\R)$ are defined by
    \begin{align*}
      \pppsi'_s &= \pppsi_s + \pppsi_t = \pppsi''_t,\\
      \pppsi'_i & = \pppsi_i \quad\text{for $i \neq s$},\\
      \pppsi''_i & = \pppsi_i \quad\text{for $i \neq t$},
    \end{align*}
    then 
    \begin{equation*}
      [\iiota] \otimes [\pppsi \iiota] = [\iiota] \otimes [\pppsi' \iiota] + [\iiota] \otimes [\pppsi'' \iiota],
    \end{equation*}
    \emph{provided} the terms $[\iiota] \otimes [\pppsi \iiota]$, $[\iiota] \otimes [\pppsi' \iiota]$, and $[\iiota] \otimes [\pppsi'' \iiota]$ are all defined.
  \end{enumerate}
  In the coinvariants $\mr{St}^2(M)_{\mr{GL}(M)}$ the analogous relations hold about permuting, scaling, and adding columns.
\end{lemma}
We emphasise that in (I) and (II), one side of the equation is defined if and only if the other side is, while to apply (III) one must know a priori that all three terms are defined.  Let us say that the row or column addition is \emph{admissible} when this is the case.

\begin{proof}[Proof sketch]
  The row operations are all special cases of~\eqref{eqn:rel-semi-local-second}, and are proved by the same method as in Section~\ref{sec:coinvariants-of}.  For row addition  one considers~\eqref{eqn:rel-semi-local-second} with $f' = \iiota \colon M \to \R^d$, $f_i = \pppsi_i \iiota \colon M \to \R$, and $f_0 = f_s + f_t$.  The fact that the relation holds when all three terms are defined follows from Lemma~\ref{lemma:fortunately}.

  For the column operations one applies instead~\eqref{eqn:rel-semi-local-first}.  For column addition with $s = 1$ and $t = 2$ one sets $f = \pppsi \iiota$, $f'_0 = \iiota_1$, $f'_1 = \iiota_1 + \iiota_2$, and $f'_i = \iiota_i$ for $i \in \{2, \dots, d\}$.
  In the coinvariants we have
  \begin{align*}
    [(\iiota_1 + \iiota_2, \iiota_2, \iiota_3, \dots, \iiota_d)] \otimes [\pppsi \iiota]
    & = [\iiota] \otimes [\tilde\pppsi' \iiota]\\
    [(\iiota_1, \iiota_1 + \iiota_2, \iiota_3, \dots, f_d)] \otimes [\pppsi \iiota]
    & = [\iiota] \otimes [\tilde\pppsi''\iiota],
  \end{align*}
  where $\tilde\pppsi'$ is obtained from $\pppsi$ by subtracting the first column from the second column and $\tilde\pppsi''$ is obtained from $\pppsi$ by subtracting the second column from the first.  Combining with column scalings by $-1$ we obtain addition instead of subtraction, and combining with column permutations we obtain column addition for any $s, t$.  As in Lemma~\ref{lemma:fortunately} the row addition is a valid relation provided all three terms are defined, which is unaffected by the remaining operations.
\end{proof}

\begin{remark}
  The presentation of $\mr{St}^2(\R^d)_{\mr{GL}_d(\R)}$ described here bears a curious resemblance to the double scissors congruence groups of Be\u{\i}linson--Goncharov--Schechtmann--Varchenko \cite[Section 2.1]{BGSV}. However, their admissibility condition is stronger than ours.
\end{remark}

\begin{lemma}\label{lem:entries-zero-or-units}
  If $[\iiota] \otimes [\pppsi \iiota]$ is defined, then all entries of $\pppsi \in \mr{GL}_d(\R)$ are in $\{0 \} \cup \R^\times$.
\end{lemma}
\begin{proof}
  Consider the subset $\sigma = \{\iiota_1, \dots, \widehat{\iiota_j}, \dots, \iiota_d, \pppsi_i \iiota\}$.  Then the cokernel of $M \to \R^\sigma \cong \R^d$ is isomorphic to $\R/(\pppsi_i(e_j))$, which for connected $\R$ is projective if and only if either $\pppsi_i(e_j)$ is either zero or a unit.
\end{proof}

\begin{example}  Assume $\R$ is not a field, and choose $\epsilon \in \mathfrak{m} \setminus \{0\}$ for some maximal ideal $\mathfrak{m} \subset \R$, and consider $[\iiota] \otimes [\pppsi\iiota]$ where $\pppsi$ has the matrix
  	\[\begin{bmatrix} 1 & \epsilon-1 \\
  	0 & 1 \end{bmatrix}.\]
    Then $[\iiota] \otimes [\pppsi]$ is defined, but adding the first to second row gives rise to an entry which is neither a unit nor zero, so would involve a term which is not defined.
  \end{example}

\subsubsection{Jordan block elements of $\mr{St}^2(M)$}
\label{sec:jord-block-elem}

In the case where $\R$ is a field, we have seen that matrices of Jordan block form play a special role.  We start by showing that the corresponding elements here are at least always defined. In fact, we will do so for a larger class of elements, which will be used in Lemma \ref{lem:claim-two}.

\begin{lemma}\label{lem:Jordan-blocks-in-St-two}
  For an isomorphism $\iiota = (\iiota_1, \dots, \iiota_d) \colon M \to \R^d$ as above, let us  consider another isomorphism $f = (f_1, \dots, f_d) \colon M \to \R^d$, of the form
  \begin{equation}\label{eq:12}
    f_i = \iiota_i + \epsilon_i \iiota_{i-1}
  \end{equation}
  for some $\epsilon_i \in \{0,1\}$, and  using the notational convention that $\iiota_0 = 0 \in M^\vee$, i.e.\ $f_1 = \iiota_1$.  Then $[\iiota] \otimes [f] \in \mr{St}^2(M)$ is defined.
  
  More generally, $[\iiota] \otimes [f]$ is defined when $f_i = \iiota_i + \epsilon_i \iiota_{j_i}$ for some $j_i \in \{0, \dots, i-1\}$.
\end{lemma}

\begin{proof}
  Clearly it suffices to show that for any subset
  \begin{equation*}
    \sigma \subset \{\iiota_1, \dots, \iiota_d, \iiota_2 + \iiota_{j_2}, \dots,  \iiota_d + \iiota_{j_d}\}
  \end{equation*}
  the evaluation homomorphism $M \to R^\sigma$ has projective cokernel.  Equivalently, the span of $\sigma \subset M^\vee$ should be a summand.  If both $\iiota_i$ and $\iiota_i + \iiota_{j_i}$ are in $\sigma$ for some $i > j_i \geq 1$, then the set $\sigma' = (\sigma \setminus \{\iiota_i + \iiota_{j_i}\}) \cup \{\iiota_{j_i}\}$ satisfies $\mr{span}(\sigma) = \mr{span}(\sigma')$, so it suffices to show that the span of $\sigma'$ is a summand.  By induction we reduce to considering those $\sigma$ which don't contain $\{\iiota_i, \iiota_i + \iiota_{j_i}\}$ as a subset for any $i$ with $i > j_i \geq 1$.  But any such $\sigma$ is a subset of a basis for $M^\vee$, namely
  \begin{equation*}
    \sigma \cup \{\iiota_i \mid \text{neither $\iiota_i \in \sigma$ nor $\iiota_i + \iiota_{j_i} \in \sigma$}\} \subset M^\vee.
  \end{equation*}
  The span of a subset of a basis is a summand.
\end{proof}

Let us adopt the notation 
\begin{equation*}
  J_d = (\iiota_1, \iiota_2 + \iiota_1, \dots, \iiota_d + \iiota_{d-1}) \colon M \lra \R^d
\end{equation*}
for the special case $\epsilon_1 = \dots = \epsilon_{d-1} = 1$ of the Lemma above.  We then have an element $[\iiota] \otimes [J_d] \in \mr{St}^2(M)$, and our conjecture is that its image in $\mr{St}^2(M)_{\mr{GL}(M)}$ is a generator for that group.

Looking back at the proof of this statement in the case $\R$ is a field and attempting the same proof here (appropriately transposing matrices), we see there are two steps: in ``Claim 1'' we should prove that the elements $[\iiota] \otimes [f]$ generate when $f$ ranges over the $2^{d-1}$ elements of the form~\eqref{eq:12}, and in ``Claim 2'' we should prove that these $2^{d-1}$ elements are all multiples of $[\iiota] \otimes [J_d]$.  It turns out that the second step may be carried out in a very similar way over any commutative ring.
\begin{lemma}\label{lem:claim-two}
  Let $\iiota \colon M \to \R^d$ be an isomorphism, and let $f \colon M \to \R^d$ be given by~\eqref{eq:12} for some $\epsilon_i \in \{0,1\}$.  Then
  \begin{equation*}
    [\iiota] \otimes [f] \in \Z \cdot ([\iiota] \otimes [J_d]) \subset \mr{St}^2(M)_{\mr{GL}(M)}.
  \end{equation*}
\end{lemma}
\begin{proof}
  As in Section~\ref{sec:coinvariants-of} it suffices to explain how to merge two Jordan blocks into one, that is, to consider the special case where $f = J_{a,b}$, where $a + b = d$ and $J_{a,b}$ is defined as
  \begin{align*}
    J_{a,b} = (& \iiota_1, \iiota_2 + \iiota_1, \dots, \iiota_a + \iiota_{a-1},\\
               & \iiota_{a+1}, \iiota_{a+2} + \iiota_{a+1}, \dots, \iiota_{a+b} + \iiota_{a+b-1}).
  \end{align*}
  Mimicking the proof there, we also define
  \begin{align*}
    J_{a,b}(i)  = (&\iiota_1, \iiota_2 + \iiota_1, \dots, \iiota_a + \iiota_{a-1},\\
    &\iiota_{a+1} + \iiota_i, \iiota_{a+2} + \iiota_{a+1}, \dots, \iiota_{a+b} + \iiota_{a+b-1})
  \end{align*}
  for $i = 0, \dots, a$.  In particular $J_{a,b}(0) = J_{a,b}$ and $J_{a,b}(a) = J_d$.  By Lemma~\ref{lem:Jordan-blocks-in-St-two} we have defined generators $[\iiota] \otimes [J_{a,b}(i)] \in \mr{St}^2(M)_{\mr{GL}(M)}$, and as in Section~\ref{sec:coinvariants-of} we claim that
  \begin{equation}\label{eq:9}
    [\iiota] \otimes [J_{a,b}(i)] = [\iiota] \otimes [J_{a,b}(i+1)] + [\iiota] \otimes [J_{b+i,a-i}(i+1)].
  \end{equation}
  Indeed, transposing the proof there, this is proved by first scaling some rows and columns by $-1$, then applying one row addition operation which replaces $J_{a,b}(i)$ by a sum of two terms, and then finally some row and column permutations to arrive at~\eqref{eq:9}.  The row operation relation holds because all three terms are defined, again by Lemma~\ref{lemma:fortunately}.

  Applying~\eqref{eq:9} repeatedly, we arrive at a formula for $[\iiota] \otimes [J_{a,b}]$ as a multiple of $[\iiota] \otimes [J_{a+b}]$.
\end{proof}

\subsubsection{Jordan blocks generate for $d \leq 3$}
\label{sec:jord-blocks-gener}

To prove our conjecture it remains to establish an analogue of ``Claim 1'' from Section~\ref{sec:coinvariants-of}, i.e.\ to rewrite an arbitrary $[\iiota] \otimes [\pppsi \iiota] \in \mr{St}^2(M)_{\mr{GL}(M)}$ as a linear combination of Jordan block elements.  For $d = 1$ there is nothing to prove, and we now consider $d = 2$ and $d=3$.

This part of the argument involves more substantial manipulations of matrices, and for notational convenience we shall assume $M = \R^d$ and that $\iiota = 1$ is the identity.  Then $\iiota_i \colon \R^d \to \R$ denotes projection to the $i$th factor, and the $i$th row of a matrix $\pppsi \in \mr{GL}_d(\R)$ is denoted $\pppsi_i \colon \R^d \to \R$.  The matrix entries $\pppsi_{i,j} \in \R$ are characterised by
\begin{equation*}
  \pppsi_i = \sum_{j = 1}^d \pppsi_{i,j} \iiota_j.
\end{equation*}

\begin{lemma}\label{lem:iota-one}
  For any $d \geq 1$ the abelian group $\mr{St}^2(\R^d)$ is generated by elements of the form $[\iiota] \otimes [\pppsi]$, where $\pppsi \colon \R^d \to \R^d$ satisfies $\pppsi_1 = \iiota_1$.
\end{lemma}
\begin{proof}
  Given an arbitrary element of the form $[\iiota] \otimes [\pppsi] \in \mr{St}^2(\R^d)$, we know that $\{\iiota_1, \dots, \iiota_d, \pppsi_1, \dots, \pppsi_d\} < E(\R^d)$, and hence we may use the relation~\eqref{eqn:rel-semi-local-second} with $f_0 = \iiota_1$.  This gives the relation
  \begin{equation*}
    [\iiota] \otimes [\pppsi] = \sum_{i = 1}^d (-1)^{i-1}[\iiota] \otimes [(\iiota_1, \pppsi_1, \dots, \widehat{\pppsi_i}, \dots, \pppsi_d)].\qedhere
  \end{equation*}
\end{proof}

\begin{corollary}
  The abelian group $\mr{St}^2(\R^2)_{\mr{GL}_2(\R)}$ is generated by $[\iiota] \otimes [J_2]$.
\end{corollary}
\begin{proof}
  By Lemma~\ref{lem:iota-one}, elements of the form $[\iiota] \otimes [\pppsi]$ with $\pppsi_1 = \iiota_1$ generate.  We then have $\pppsi_{2,1}, \pppsi_{2,2} \in \R^\times\cup\{0\}$, by Lemma~\ref{lem:entries-zero-or-units}.  If $\pppsi_{2,2} = 0$ then $[\iiota] \otimes [\pppsi] = 0$, and otherwise we may scale the second row to achieve $\pppsi_{2,2} = 1$.

  If $\pppsi_{2,1} \in \R^\times$ we may scale the second column by $\pppsi_{2,1}^{-1}$ and the first row by $\pppsi_{2,1}$, after which $[\iiota] \otimes [\pppsi] = [\iiota] \otimes [J_2]$.  If $\pppsi_{2,1} = 0$ then $\pppsi = \iiota$, and $[\iiota] \otimes [\iiota]$ is a multiple of $[\iiota] \otimes [J_2] \in \mr{St}^2(\R^2)_{\mr{GL}(\R^2)}$, by Lemma~\ref{lem:claim-two}.
\end{proof}

\begin{proposition}
  The abelian group $\mr{St}^2(\R^3)_{\mr{GL}(\R^3)}$ is generated by $[\iiota] \otimes [J_3]$.
\end{proposition}

\begin{proof}
  By Lemma~\ref{lem:iota-one}, the elements of the form
  \begin{equation}
    \label{eq:28}
    [\iiota] \otimes
    \begin{bmatrix}
      1 & 0 & 0\\
      \pppsi_{2,1} & \pppsi_{2,2} & \pppsi_{2,3}\\
      \pppsi_{3,1} & \pppsi_{3,2} & \pppsi_{3,3}
    \end{bmatrix}
  \end{equation}
  generate.  We now reduce this set of generators to a set covered by Lemma~\ref{lem:claim-two}, in the following steps.

\vspace{.5em}

\noindent \emph{Step 1.} We first show that the elements~\eqref{eq:28} with $\pppsi_{3,1} = 0$ generate.

Starting from a generator of the form~\eqref{eq:28}, we achieve this immediately if $\pppsi_{2,1} = 0$ by permuting the first two rows, so we can from now on assume that $\pppsi_{2,1}$ and $\pppsi_{3,1}$ are both units, by Lemma~\ref{lem:entries-zero-or-units}. Hence by scaling rows we may assume $\pppsi_{2,1} = \pppsi_{3,1} = 1$.  We are reduced to considering $[\iiota] \otimes [\pppsi']$, where $\pppsi'$ has matrix
\begin{equation}\label{eq:29}
  \pppsi' = \begin{bmatrix}
    1 & 0  & 0\\
    1 & \pppsi_{2,2} & \pppsi_{2,3}\\
    1 & \pppsi_{3,2} & \pppsi_{3,3}
  \end{bmatrix}.
\end{equation}
We know the cokernel of $(\iiota_3,\pppsi_2, \pppsi_3): \R^3 \to \R^3$ is projective,
and the matrix for this homomorphism is
\begin{equation*}
  \begin{bmatrix}
    0 & 0 & 1 \\
    1 & \pppsi_{2,2} & \pppsi_{2,3}\\
    1 & \pppsi_{3,2} & \pppsi_{3,3}
  \end{bmatrix}.
\end{equation*}
For connected $\R$ this can only happen if the determinant (which generates the zeroth Fitting ideal of the cokernel) is either zero or a unit.  A similar argument applies starting from projectivity of the cokernel of $(\iiota_2, \pppsi_2, \pppsi_3)$, and we deduce
\begin{align*}
  \pppsi_{3,2} - \pppsi_{2,2} & \in \{0\} \cup \R^\times,\\
  \pppsi_{3,3} - \pppsi_{2,3} & \in \{0\} \cup \R^\times.
\end{align*}

We would now like to subtract the second and the third rows of~\eqref{eq:29}, that is, to apply~\eqref{eqn:rel-semi-local-second} with $f_0 = \pppsi_3 - \pppsi_2$.  To prove that operation is admissible we must prove that the span of any subset
\begin{equation*}
  \sigma \subset \{ \iiota_1, \iiota_2, \iiota_3, \pppsi_3 - \pppsi_2, \iiota_1, \pppsi_2, \pppsi_3\}
\end{equation*}
is a summand of $(\R^3)^\vee$.  We know this is the case for subsets not containing $\pppsi_3 - \pppsi_2$, and for subsets containing both $\pppsi_3 - \pppsi_2$ and either $\pppsi_2$ or $\pppsi_3$ we have
\begin{equation*}
  \mr{span}(\sigma) = \mr{span}((\sigma \setminus \{\pppsi_3 - \pppsi_2\}) \cup \{\pppsi_2, \pppsi_3\}),
\end{equation*}
which is a summand.  This leaves the case where
\begin{equation*}
  \sigma \subset \{ \iiota_1, \iiota_2, \iiota_3, \pppsi_3 - \pppsi_2\}
\end{equation*}
which is shown by an elementary case-by-case argument similar to Lemma~\ref{lem:Jordan-blocks-in-St-two}, using that $\pppsi_3 - \pppsi_2$ is a linear combination of $\iiota_2$ and $\iiota_3$ with coefficients in $\{0\} \cup \R^\times$.

For $\pppsi'$ of the form~\eqref{eq:29}, the row subtraction gives
\begin{align*}
  [\iiota] \otimes [\pppsi'] & = [\iiota] \otimes [(\iiota_1, \pppsi_2 - \pppsi_3, \pppsi_3)] + [\iiota] \otimes [(\iiota_1, \pppsi_2, \pppsi_3 - \pppsi_2)]\\
  & = [\iiota] \otimes [(\iiota_1, \pppsi_2, \pppsi_2 - \pppsi_3)] - [\iiota] \otimes [(\iiota_1, \pppsi_3, \pppsi_2 - \pppsi_3)],
\end{align*}
both terms of which are of the form~\eqref{eq:28} but with $\pppsi_{3,1} = 0$.

 \vspace{.5em}
 
 \noindent \emph{Step 2.}  We have shown that $\mr{St}^2(\R^3)_{\mr{GL}_3(\R)}$ is generated by elements of the form
 \begin{equation}\label{eq:7}
   [\iiota] \otimes
   \begin{bmatrix}
      1 & 0 & 0\\
      \pppsi_{2,1} & \pppsi_{2,2} & \pppsi_{2,3}\\
      0 & \pppsi_{3,2} & \pppsi_{3,3}
    \end{bmatrix}.
  \end{equation}
  We now claim that $\mr{St}^2(\R^3)_{\mr{GL}_3(\R)}$  is generated by elements of this form which also satisfy $\pppsi_{3,2} = 0$.

 If $\pppsi_{3,3} = 0$ we can obtain this by permuting the last two columns, so from now on we will assume both $\pppsi_{3,2}$ and $\pppsi_{3,3}$ are both units.  By scaling columns we can then assume $\pppsi_{3,2} = \pppsi_{3,3} = 1$.

 We would now like to subtract the last two columns, i.e.\ to apply~\eqref{eqn:rel-semi-local-first} with $f'_0 = \iiota_2 + \iiota_3$.  We must first prove that this column operation is admissible, i.e.\ that
 \begin{equation*}
   \{\iiota_2 + \iiota_3, \iiota_1, \iiota_2, \iiota_3, \iiota_1, \pppsi_2, \pppsi_3\} < E(\R^3).
 \end{equation*}
 But this is obvious, since $\pppsi_3 = \iiota_2 + \iiota_3$.   After applying this column subtraction and in one of the two resulting terms permuting the last two columns,  we see that~\eqref{eq:7} is a linear combination of elements of the form
 \begin{equation}\label{eq:16}
   [\iiota] \otimes
   \begin{bmatrix}
     1 & 0 & 0 \\
     \pppsi_{2,1} & \pppsi_{2,2} & \pppsi_{2,3}\\
     0 & 0 & 1
   \end{bmatrix}
 \end{equation}

 \vspace{.5em}
 
 \noindent \emph{Step 3.} Finally, starting from a generator of the form~\eqref{eq:16} we may, after row and column scalings, assume $\pppsi_{2,1}$, $\pppsi_{2,2}$, and $\pppsi_{2,3}$ are all in $\{0,1\}$.  If $\pppsi_{2,2} = 0$ then the rows do not form a basis, so the element~\eqref{eq:16} represents $0 \in \mr{St}^2(\R^3)$.  The case where $\pppsi_{2,3} = 0$ is covered by Lemma~\ref{lem:claim-two}, and the case $\pppsi_{2,1} = 0$ is covered by Lemma~\ref{lem:claim-two} after permuting the last two rows and the last two columns.

 This leaves the case where $\pppsi_{2,1} = \pppsi_{2,2} = \pppsi_{2,3} = 1$, in which we will subtract the last two columns once more.  To see that this column operation is admissible, we must as above show that the set
 \begin{equation*}
   \{\iiota_2 + \iiota_3, \iiota_1, \iiota_2, \iiota_3, \iiota_1, \iiota_1 + \iiota_2 + \iiota_3, \iiota_3\}  = \{\iiota_1, \iiota_2, \iiota_3, \iiota_1 + \iiota_2 + \iiota_3, \iiota_2 + \iiota_3\}
 \end{equation*}
 is a simplex of $E(\R^3)$, which is done by an elementary case-by-case argument similar to Lemma~\ref{lem:Jordan-blocks-in-St-two}.
  We may therefore subtract the last two columns, showing
 \begin{align*}
   [\iiota] \otimes
   \begin{bmatrix}
     1 & 0 & 0 \\
     1 & 1 & 1 \\
     0 & 0 & 1
   \end{bmatrix}
   & =
   [\iiota] \otimes
   \begin{bmatrix}
     1 & 0 & 0 \\
     1 & 0 & 1 \\
     0 & -1 & 1
   \end{bmatrix}
   + [\iiota] \otimes
   \begin{bmatrix}
     1 & 0 & 0 \\
     1 & 1 & 0 \\
     0 & 0 & 1
   \end{bmatrix}
   \\ & =
   [\iiota] \otimes
   \begin{bmatrix}
     1 & 0 & 0 \\
     1 & 1 & 0 \\
     0 & 1 & 1
   \end{bmatrix}
   + [\iiota] \otimes
   \begin{bmatrix}
     1 & 0 & 0 \\
     1 & 1 & 0 \\
     0 & 0 & 1
   \end{bmatrix},
 \end{align*}
both terms of which are covered by Lemma~\ref{lem:claim-two}.
\end{proof}

This finishes the proof of Theorem~\ref{thm:proof-of-conjecture-rank-leq-3}, asserting that Conjecture \ref{conj:double-steinberg-local} is true for connected semi-local rings $\R$ with infinite residue fields, and projective $\R$-modules $M$ of rank $d\leq 3$.

\section{Rognes' conjecture} \label{sec:rognes-conj}

Some of the results developed so far can be interpreted as proving a version of a conjecture of Rognes. The free algebraic $K$-theory spectrum $\mathbf{K}^f(\R)$ of $\R$ is obtained by restricting to those components of the algebraic $K$-theory spectrum $\mathbf{K}(\R)$ which correspond to free finitely-generated modules. In \cite{rognesrank}, when $\R$ has the invariant dimension property Rognes provides a filtration $\{F_n \mathbf{K}^f(\R)\}$ of $\mathbf{K}^f(\R)$ and equivalences
\[F_n \mathbf{K}^f(\R)/F_{n-1} \mathbf{K}^f(\R) \simeq \mathbf{D}^f(\R^n) \hcoker \mr{GL}_n(\R),\]
where $\mathbf{D}^f(\R^n)$ is the subspectrum of the stable building of Definition \ref{def:stable-building} where all summands are free. The connectivity of these homotopy orbit spectra therefore control the speed of convergence of this filtration. 

If $\R$ is a connected semi-local ring then projective modules are free, so we have $\gD^f(\R^n) = \gD(\R^n)$. We then have the following. 

\begin{theorem}\label{thm:RognesMain}
	If $\R$ is a connected semi-local ring with all residue fields infinite, then $H_d(\mathbf{D}(\R^n) \hcoker \mr{GL}_n(\R))=0$ for $d < 2n-2$. Furthermore, if $\R$ is an infinite field then $H_{2n-2}(\mathbf{D}(\R^n) \hcoker \mr{GL}_n(\R))$ is torsion for $n \geq 2$.
\end{theorem}

Rognes conjectured the stronger statement that for $\R$ a Euclidean domain or local ring the homology of $\gD(\R^n)$ is concentrated in degree $2n-2$ \cite[Conjecture 12.3]{rognesrank}, and he proved that it vanishes in degrees $>2n-2$ \cite[Theorem 12.1]{rognesrank}.
Theorem \ref{thm:RognesMain} for the case of local rings would be a consequence of that conjecture.

\begin{proof}[Proof of Theorem \ref{thm:RognesMain}]
By Example \ref{exam:many-units} the pair $(\R, \bZ)$ satisfies the Nesterenko--Suslin property, and so by Theorem \ref{thmcor:BlockvsFlag} the maps on pointed homotopy orbits
	\[\tilde{D}^{k}(\R^n)\hcoker \mr{GL}_n(\R) \lra {D}^{k}(\R^n)\hcoker \mr{GL}_n(\R)\]
are integral homology equivalences. By Theorem \ref{thm:split-building-ek} we have
	\[H^{E_k}_{n,d}(\gR_\bZ) \cong \widetilde{H}_{d-k}(\tilde{D}^{k}(\R^n)\hcoker \mr{GL}_n(\R);\bZ),\]
which combined with the above gives $H^{E_\infty}_{n,d}(\gR_\bZ) \cong H_{d}({\mathbf{D}}(\R^n) \hcoker \mr{GL}_n(\R);\bZ)$ in the limit $k \to \infty$. The first part of the Theorem then follows from Corollary \ref{cor:thmA}, and the second part follows from Corollary \ref{cor:einfty-hom-indec}.
\end{proof}

\begin{remark}
By the results of Section \ref{sec:coinv-e_2-steinb}, the second part is also true for $n= 2,3$ when $A$ is a connected semi-local $A$ with infinite residue fields. If Conjecture \ref{conj:double-steinberg-local} is true, it would follow for all $n \geq 2$.
\end{remark}

\begin{remark}
The argument of Theorem \ref{thm:RognesMain} applies with $\R$ any field $\bF$, as long as we use $\bk$-coefficients where $(\bF, \bk)$ satisfies the Nesterenko--Suslin property. If $\bF$ is a finite field this means that $\bk$ must have a different characteristic. But in this case Quillen \cite{quillenfinite} has completely calculated the $\bk$-homology of all $\mr{GL}_n(\bF)$'s, and in \cite[Section 7]{e2cellsIII} we used this to completely compute the corresponding $E_1$-homology. By iterating the bar spectral sequence of Theorem $E_k$.14.1 one could obtain far more information about $H_*(\gD(\bF^n) \hcoker \mr{GL}_n(\bF);\bk)$ than a mere vanishing range. The same applies when $\bF$ is a number field and $\bk$ has characteristic zero, using \cite{borelyang}.
\end{remark}

\begin{remark}\label{rem:RkFiltComparison}
Let us briefly explain how Rognes' filtration of $\gK(\R)$, called the ``stable rank filtration'', relates to our setup (see in particular Section $E_k$.13.8).  By the same method as in Section~\ref{sec:contructing-r} we may construct an object $\gR' \in \cat{Alg}_{E_\infty}(\cat{sSet}^{\bZ_{\leq}})$ with
\begin{equation*}
	\gR'(n) \simeq \coprod_{\substack{[M]\\1 \leq r(M) \leq n}} B\mr{GL}(M)
\end{equation*}
for $n \geq 1$ and $\gR'(n) = \varnothing$ for $n \leq 0$,  a filtered variant of~\eqref{eq:ValuesOfR}.  The reduced bar construction $\tilde{B}^{E_k}(-)$ may be carried out in the filtered category $\cat{sSet}^{\bZ_{\leq}}$, using the augmentation $\epsilon \colon (\gR')^+ \to 0_*(*)$ to the filtered object given by a point in filtration $\geq 0$ and the empty set otherwise (called the ``group-completion augmentation'' in Section~$E_k$.13.8.1).
This gives an object $\tilde{B}^{E_k}(\gR') \in \cat{sSet}_*^{\Z_{\leq}}$ which has
\begin{align*}
	\mr{gr}\big(\tilde{B}^{E_k}(\gR')\big) \cong \tilde{B}^{E_k}(\gR) 
\end{align*}
As explained in Section~$E_k$.13.8, the object $\colim \tilde{B}^{E_k}(\gR') \in \cat{sSet}_*$ is a $k$-fold delooping of the group-completion of $\colim \gR$, that is,
\begin{equation*}
	\colim \tilde{B}^{E_k}(\gR') \simeq \Omega^\infty \Sigma^k \gK(\R).
\end{equation*}

As $k$ varies, these objects $\tilde{B}^{E_k}(\gR')$ fit together into a filtered spectrum, which we shall denote $r \mapsto F_r \gK(A)$.  It is a mild variant of Rognes' stable rank filtration on $\gK^f(\R)$.  A minor difference is that $\gK(\R)$ may differ from $\gK^f(\R)$ when $\R$ admits non-free projective modules.  A more significant difference is that we have a filtration on \emph{direct sum} $K$-theory and that correspondingly we have
\begin{equation*}
	(\mr{gr}(\gK(\R)))(n) \simeq \bigvee_{\substack{[M]\\ r(M) = n}} \tilde{\gD}(M)\hcoker \mr{GL}(M),
\end{equation*}
with stable buildings replaced by their split versions, specifically $\tilde{\gD}(M)$ is the spectrum associated to the $\Gamma$-set~\eqref{eq:3}.  Forgetting splittings leads to a filtered map from this filtered spectrum to the non-split version, which on associated gradeds comes from the maps~\eqref{eq:Comp}.  When $(\R,\bZ)$ has the Nesterenko--Suslin property this map of associated gradeds is then a homology equivalence, by Theorem~\ref{thmcor:BlockvsFlag} and since the spectra in the associated gradeds are connective it must be a weak equivalence.  This identifies our ``split stable rank filtration'' on $\gK(\R)$ with Rognes', when all projective $\R$-modules are free and $(\R,\bZ)$ has the Nesterenko--Suslin property.

Defining the (split) stable rank filtration on $K_*(\R)$ as $F_r K_*(\R) = \mr{im} \left[ \pi_*(F_r \gK(\R)) \to \pi_*(\gK(\R))\right]$ we obtain for a connected semi-local ring with infinite residue fields that the inclusion $F_r K_d(\R) \hookrightarrow K_d(\R)$ is an equality for $d \leq 2r-1$.  When $d = 2r$ the quotient $K_{2r}(\R)/F_r K_{2r}(\R)$ is $\pi_{2r}(\tilde{\mathbf{D}}(\R^{r+1})\hcoker \mr{GL}_{r+1}(\R))$, which by Theorem~\ref{thm:split-building-ek} and the Hurewicz theorem may be identified with $H^{E_\infty}_{r+1,2r}(\gR_\bZ)$.  In particular if $\R$ is an infinite field (or if Conjecture~\ref{conj:double-steinberg-local} is true) then $K_{2r}(\R)/F_r K_{2r}(\R)$ is finite for $r > 0$, given by Corollary~\ref{cor:einfty-hom-indec}.
\end{remark}

\begin{remark}
Arone and Lesh \cite{AroneLesh2} have proved Rognes' conjecture for $\R$ given by $\bC$ or $\bR$ \emph{considered as topological rings}. Let us explain how their method relates to ours. For a general $\R$ they construct a ``modified stable rank filtration'' of $\gK(\R)$ in the framework of $\Gamma$-spaces: accounting for the differences of framework, it is conceptually the same as the ``split stable rank filtration'' described in Remark \ref{rem:RkFiltComparison}. They then show that for $\bC$ or $\bR$ this agrees with Rognes' filtration: from our point of view this is because when taken with their Euclidean topology the maps $\mr{GL}(Q) \overset{i}\to \mr{GL}(P \oplus Q,\text{ fix } P) \overset{\rho}\to \mr{GL}(Q)$ are homotopy equivalences, which is much stronger than the Nesterenko--Suslin property and ensures that the map from the split to the ordinary building is an equivalence. Finally, they identify the filtration quotients of their filtration in terms of certain spectra they had studied earlier \cite{AroneLesh} and had shown to be highly-connected: from our point of view the extremely high-connectivity of the topologised form of the Tits building \cite[p.~134]{Mitchell} gives a slope 2 vanishing line for $E_1$-homology, and taking iterated bar spectral sequences (using Proposition $E_k$.14.5) this persists to $E_\infty$-homology.
\end{remark}

\section{Applications to homological stability}\label{sec:AppHomStab}
In this section we shall apply the results obtained so far, along with the general machinery of \cite{e2cellsIv3}, to analyse homological stability properties of the groups $\mr{GL}_n(\R)$ when $\R$ is a connected semi-local ring with all reside fields infinite. We therefore let $\gR \in \cat{Alg}_{E_\infty}(\cat{sSet}^\bN)$ be the $E_\infty$-algebra constructed in Section \ref{sec:contructing-r}, having $\gR(n) \simeq B\mr{GL}_n(\R)$, and denote by $\gR_\bk \in \cat{Alg}_{E_\infty}(\cat{sMod}_\bk^\bN)$ its $\bk$-linearisation.

\subsection{$E_\infty$-homology in low degrees}\label{sec:EinfHomologyCalc}
In this section we shall describe the $E_\infty$-homology of $\gR_{\bZ}$ in low degrees, relying on calculations of the group homology of $\mr{GL}_2(\R)$. The result is summarised in the chart given in Figure \ref{fig:eInfHomology} of the $E_\infty$-homology of $\gR_\bZ$, though we shall also explain some further features related to the classes called $\rho_2$ and $\rho_3$.

\begin{figure}[ht]
	\begin{tikzpicture}
	\begin{scope}
	\clip (-2,-1) rectangle ({2.5*4+2},6.5);
	\draw (0,0)--(10.5,0);
	\draw (0,0) -- (0,6.5);
	\foreach \s in {0,...,6}
	{
		\draw [dotted] (-1.5,\s)--(10.5,\s);
		\node [fill=white] at (-1.5,\s) [left] {\tiny $\s$};
	}
	\foreach \s in {0,...,4}
	{
		\draw [dotted] ({2.5*\s},-0.5)--({2.5*\s},6.5);
		\node [fill=white] at ({2.5*\s},-.5) {\tiny $\s$};
	}
	\draw [very thick,Mahogany,densely dotted] (2.5,0) -- (10,6);

	\node [fill=white] at (2.5,0) {$\bZ$};
	\node [fill=white] at (2.5,1) {$H_1(\R^\times;\bZ)$};
	\node [fill=white] at (2.5,2) {$H_2(\R^\times;\bZ)$};
	\node [fill=white] at (2.5,3) {$H_3(\R^\times;\bZ)$};
	\node [fill=white] at (2.5,4) {$H_4(\R^\times;\bZ)$};
	\node [fill=white] at (2.5,5) {$H_5(\R^\times;\bZ)$};
	\node [fill=white] at (2.5,6) {$H_6(\R^\times;\bZ)$};

	\node [fill=white] at (5,2) {$\bZ/2\{\rho_2\}$};
	\node [fill=white] at (5,3) {$\mathfrak{p}(\R)$};
	\node [fill=white] at (5,4) {?};
	\node [fill=white] at (5,5) {?};
	\node [fill=white] at (5,6) {?};

	\node [fill=white] at (7.5,4) {$\bZ/3\{\rho_3\}$};
	\node [fill=white] at (7.5,5) {?};
	\node [fill=white] at (7.5,6) {?};

	\node [fill=white] at (10,6) {?};
	
	\node at (-.5,-.5) {$\nicefrac{d}{n}$};
	\end{scope}
	\end{tikzpicture}
	\caption{The $E_\infty$-homology of $\gR_\bZ$, which vanishes for $d < 2n-2$.}
	\label{fig:eInfHomology}
\end{figure}

\subsubsection{The pre-Bloch group}\label{sec:PreBloch}

The entry $\mathfrak{p}(\R)$ in position $(2,3)$ of Figure \ref{fig:eInfHomology} denotes the \emph{pre-Bloch group}, which is the abelian group generated by symbols $[x]$ for $x \in \R^\times \setminus \{1\}$ subject to the relations
	\[[x]-[y]+\left[\frac{y}{x}\right]-\left[\frac{1-x^{-1}}{1-y^{-1}}\right]+\left[\frac{1-x}{1-y}\right]=0\]
whenever $x$, $y$, $1-x$, $1-y$, and $x-y \in \R^\times$. 

When $\R=\bF$ is an infinite field, this group arises in Suslin's analysis \cite{SuslinK3} of the relation between $K_3(\bF)$ and the Bloch group of $\bF$. A similar analysis has been done for $\R$ a ring with many units---which includes semi-local rings with all residue fields infinite---by Mirzaii \cite{MirzaiiBW, MirzaiiBWErratum}. Let $\mr{GM}_n(\R)$ denote the subgroup of $\mr{GL}_n(\R)$ consisting of \emph{monomial matrices}, i.e.\ invertible matrices having a single non-zero entry in each row and column. 

\begin{theorem}[Suslin, Mirzaii]\label{thm:SuslinCalc}
Let $\R$ be a ring with many units. Then we have 
\begin{enumerate}[(i)]
	\item $H_1(\mr{GL}_2(\R), \mr{GM}_2(\R)) = 0$.
	\item $H_2(\mr{GL}_2(\R), \mr{GM}_2(\R)) \cong \bZ/2$.
	\item $H_3(\mr{GL}_2(\R), \mr{GM}_2(\R)) \cong \mathfrak{p}(\R)$.
\end{enumerate}
\end{theorem}

As we will need to use some details of this calculation, we provide an outline of the proof, following \cite{SuslinK3, MirzaiiBW, MirzaiiBWErratum}.

\begin{proof}[Proof sketch] We consider the augmented chain complex $C_\bullet(\R^2) \to \bZ$ with $C_p(\R^2)$ being the free abelian group with basis the ordered $(p+1)$-tuples $(L_0, L_1, \ldots, L_p)$ of free $\R$-submodules $L_i \leq \R^2$ of rank 1 which are direct summands, such that each pair gives a direct-sum decomposition of $\R^2$. The boundary is given by the alternating sum over omitting the $L_i$. This augmented chain complex is acyclic, by the argument of Proposition~\ref{prop:E-of-M-is-contractible}. Let us write $Z_i \coloneqq \mr{Ker}[\partial \colon C_i(\R^2) \to C_{i-1}(\R^2)]$.
  
There is a map of exact chain complexes
\begin{equation}\label{eq:17}
  \begin{tikzcd} 
    \bZ \arrow[equals]{d}& \bZ[\fS_2] \lar[swap]{\epsilon} \dar& \bZ[\fS_2] \lar[swap]{t-1} \dar & \bZ \dar \lar[swap]{t+1}\\
    \bZ & C_0(\R^2) \lar[swap]{\epsilon} & C_1(\R^2) \lar[swap]{\partial} & \lar Z_1 
  \end{tikzcd}
\end{equation}
equivariant for $\mr{GM}_2(\R) \to \mr{GL}_2(\R)$, where $t$ is the non-trivial element in the symmetric group $\fS_2$.  Here the second vertical map is $a+bt \mapsto a \langle e_1 \rangle + b \langle e_2 \rangle$, and the third vertical map is $a+bt \mapsto a(\langle e_1 \rangle,\langle e_2 \rangle) + b(\langle e_2 \rangle,\langle e_1 \rangle)$; the fourth vertical map is induced by the third.

Now the maps
	\[H_*(\mr{GM}_2(\R) ; \bZ[\fS_2])  \lra H_*(\mr{GL}_2(\R); C_i(\R^2))\]
for $i=0,1$ are isomorphisms: by Shapiro's lemma they are identified with
\[
	H_*(\begin{bsmallmatrix} * & 0 \\0 & * \end{bsmallmatrix} ; \bZ)  \lra H_*(\begin{bsmallmatrix} * & * \\0 & * \end{bsmallmatrix}; \bZ) \quad\text{ and }\quad
	H_*(\begin{bsmallmatrix} * & 0 \\0 & * \end{bsmallmatrix} ; \bZ)  \lra H_*(\begin{bsmallmatrix} * & 0 \\0 & * \end{bsmallmatrix} ; \bZ)
\]
respectively, where the second is the identity map and the first is an isomorphism by the Nesterenko--Suslin property for $(\R, \bZ)$. The $d^3$-differential in the associated hyperhomology spectral sequence must then be an isomorphism
\begin{equation}\label{eq:30}  
    H_{d-2}((\mr{GL}_2(\R);Z_1), (\mr{GM}_2(\R);\bZ)) \xrightarrow{\cong} H_d(\mr{GL}_2(\R), \mr{GM}_2(\R);\bZ)
\end{equation}
from relative group homology with coefficients.

As in \cite[p.~332 eq.~(1)]{MirzaiiBW}, which is not affected by the error described in \cite{MirzaiiBWErratum}, an explicit isomorphism
\begin{equation*}
	\mathfrak{p}(\R) \xrightarrow{\cong} H_0(\mr{GL}_2(\R);Z_2)
\end{equation*}
is constructed.  The short exact sequence $Z_2 \to  C_2(\R^2) \xrightarrow{k} Z_1$ induces a long exact sequence of $\mr{GL}_2(\R)$-homology, part of which is
\begin{equation*}
  \begin{tikzcd}[column sep=1.4em]  \cdots \rar &[-3pt] H_1(\mr{GL}_2(\R);C_2(\R^2)) \rar{H_1(k)}
  \ar[draw=none]{d}[name=X, anchor=center]{}
  &[3pt] H_1(\mr{GL}_2(\R);Z_1) \ar[rounded corners,
  to path={ -- ([xshift=2ex]\tikztostart.east)
  	|- (X.center) \tikztonodes
  	-| ([xshift=-2ex]\tikztotarget.west)
  	-- (\tikztotarget)}]{dll}[at end]{} & \\      
   H_0(\mr{GL}_2(\R);Z_2) \rar & H_0(\mr{GL}_2(\R);C_2(\R^2)) \rar{H_0(k)} &  H_0(\mr{GL}_2(\R);Z_1) \rar & 0.\end{tikzcd}
\end{equation*}
If we write $j\colon Z_1 \to C_1(\R^2)$ for the inclusion, then Shapiro's lemma identifies the composition $H_i(j \circ k)\colon H_i(\mr{GL}_2(\R);C_2(\R^2)) \to H_i(\mr{GL}_2(\R);C_1(\R^2))$ with $\mr{diag}\colon H_i(\R^\times) \to H_i(\R^\times \times \R^\times)$, which in turn may be identified with $\mr{id} \colon \bZ \to \bZ$ for $i = 0$ and with $\mr{diag}\colon A^\times \to A^\times \times A^\times$ for $i = 1$.  These are both split injective and so provide splittings of $H_0(k)$ and $H_1(k)$, so that the portion of the long exact sequence above gives isomorphisms
\begin{align*}
	H_0(\mr{GL}_2(\R);Z_1) & \cong H_0(\mr{GL}_2(\R);C_2(\R^2)) \cong H_0(\mr{GL}_2(\R);C_1(A^2)) \cong \bZ,\\
	H_1(\mr{GL}_2(\R);Z_1) & \cong  H_1(\mr{GL}_2(\R);C_2(\R^2)) \oplus H_0(\mr{GL}_2(\R);Z_2) \cong A^\times \oplus \mathfrak{p}(\R).
\end{align*}

By tracing through isomorphisms, it is now easy to identify $H_0(\mr{GM}_2(\R);\bZ) \to H_0(\mr{GL}_2(\R);Z_1)$ with $2 \colon \bZ \to \bZ$, and $H_1(\mr{GM}_2(\R);\bZ) \to H_1(\mr{GL}_2(\R);Z_1)$ with $\mr{id} \oplus 0 \colon \R^\times \oplus \bZ/2 \to \R^\times \oplus \mathfrak{p}(\R)$.  This shows that the groups~\eqref{eq:30} are as stated for $d = 2$ and $d = 3$, and $d=1$ is trivial.
\end{proof}

We must also recall a detail from this calculation regarding the connecting homomorphisms $\partial \colon H_{d+1}(\mr{GL}_2(\R), \mr{GM}_2(\R)) \to H_d(\mr{GM}_2(\R))$. The group $\mr{GM}_2(\R)$ is the semidirect product $(\R^\times)^2 \rtimes \fS_2$, and we first describe its homology.

Let us write $\tilde{\wedge}^n M \coloneqq [M^{\otimes n} \otimes \bZ^{-}]_{\fS_n}$. The exterior product $\wedge^2 M$ is the quotient of $M \otimes M$ by the subgroup generated by all elements of the form $m \otimes m$, so there is a surjective map $\tilde{\wedge}^2 M \to \wedge^2 M$ whose kernel is 2-torsion. As for any abelian group, we have
	\[H_1(\R^\times) \cong \R^\times, \quad\quad H_2(\R^\times) \cong \wedge^2 \R^\times\]
and the Serre spectral sequence for the semidirect product $(\R^\times)^2 \rtimes \fS_2$ provides a split short exact sequence
\begin{align*}
	0 \lra \R^\times = H_1(\R^\times \times \R^\times)_{\fS_2} \lra H_1(\mr{GM}_2(\R)) \lra H_1(\fS_2) = \bZ/2 \lra 0
\end{align*}
and, via $H_2(\R^\times \times \R^\times) \cong (\wedge^2 \R^\times \oplus \wedge^2 \R^\times) \oplus (\R^\times \otimes \R^\times)$ from the K{\"u}nneth theorem, an isomorphism $H_2(\mr{GM}_2(\R)) \cong  \wedge^2 \R^\times \oplus \tilde{\wedge}^2 \R^\times$. These give
	\[H_1(\mr{GM}_2(\R), \mr{GM}_1(\R)) \cong \bZ/2, \quad\quad H_2(\mr{GM}_2(\R), \mr{GM}_1(\R)) \cong \tilde{\wedge}^2 \R^\times.\]

The map
	\[\mathfrak{p}(\R) \cong H_3(\mr{GL}_2(\R), \mr{GM}_2(\R)) \overset{\partial}\lra H_2(\mr{GM}_2(\R)) \cong \wedge^2 \R^\times \oplus \tilde{\wedge}^2 \R^\times\]
is then identified with 
\begin{equation}\label{eq:AttMap}
	\mathfrak{p}(\R) \ni [x] \longmapsto  (x \wedge (1-x), -x \tilde{\wedge} (1-x))
\end{equation}
Indeed, the connecting homomorphism may be identified with the $d^3$-differential in the spectral sequence arising by replacing $\bZ$ by $0$ in the lower left corner of~\eqref{eq:17}.  This differential may in turn be identified with the $d^2$-differential calculated in \cite[Lemma 4.1]{MirzaiiBW} or the $d^3$-differential in \cite[Lemma 2.4]{SuslinK3}. 

What we shall need is the composition
\[
	\mathfrak{p}(\R) \cong H_3(\mr{GL}_2(\R), \mr{GM}_2(\R)) \overset{\partial}\lra H_2(\mr{GM}_2(\R)) \to H_2(\mr{GM}_2(\R), \mr{GM}_1(\R)) \cong \tilde{\wedge}^2 \R^\times,
\]
which is then $[x] \mapsto -x \tilde{\wedge} (1-x)$.

\subsubsection{Determining Figure~\ref{fig:eInfHomology}}
\label{sec:determ-figure-reff}

We will calculate the indicated entries in Figure \ref{fig:eInfHomology} by using Theorem \ref{thm:SuslinCalc} and a Hurewicz theorem. First use the identification $B\R^\times = B\mr{GL}_1(\R) = \gR(1)$, and the fact that $\gR$ is an $E_\infty$-algebra, to obtain by adjunction a map of $E_\infty$-algebras
\begin{equation}\label{eqn:first-approx} 
	i\colon \gA \coloneqq \gE_{\infty}(S^{1,0} \otimes \bZ[B\R^\times]) \lra \gR_{\bZ}\end{equation}
in $\cat{sMod}_\bZ^\bN$. Explicitly we have
\[\gA(n) \simeq \bZ[B((\R^\times)^n \rtimes \fS_n)]\]
and we think of $(\R^\times)^n \rtimes \fS_n$ as the subgroup $\mr{GM}_n(\R) \leq \mr{GL}_n(\R)$ of {monomial matrices}. Thus we may interpret Theorem \ref{thm:SuslinCalc} as a calculation of $H_{2,d}(\gR_\bZ, \gA)$ for $d \leq 3$. 

We now apply Proposition $E_k$.11.9, a Hurewicz theorem. The objects $\gA$ and $\gR_\bZ$ are $(\infty,0,0,\ldots)$-connective, and the map $i$ is $(\infty, \infty, 0, 0, \ldots)$-connective, so that proposition implies that the Hurewicz map
\begin{equation}\label{eq:HurewiczInGrading2}
H_{2,*}(\gR_\bZ, \gA) \lra H_{2,*}^{E_\infty}(\gR_\bZ, \gA)
\end{equation}
is an isomorphism. On the other hand as $\gA$ only has $E_\infty$-cells of rank 1, the map $H_{2,*}^{E_\infty}(\gR_\bZ) \to H_{2,*}^{E_\infty}(\gR_\bZ, \gA)$ is an isomorphism too, together giving an identification
	\[H_{2,d}^{E_\infty}(\gR_\bZ) \cong H_d(\mr{GL}_2(\R), \mr{GM}_2(\R);\bZ).\]
With Theorem \ref{thm:SuslinCalc} this determines the $n=2$ column of Figure~\ref{fig:eInfHomology}.

The $n=3$ column of Figure~\ref{fig:eInfHomology} follows by combining Theorem \ref{thm:proof-of-conjecture-rank-leq-3} with the argument of Corollaries \ref{cor:EInfHomNrelR} and \ref{cor:einfty-hom-indec} for $n \leq 3$, which gives that $H_{3,4}^{E_\infty}(\gR_\bZ) \cong \bZ/3$. The vanishing line follows from Theorem \ref{thm:Solomon--Tits-local}, as described in Section \ref{sec:defin-conn-theor}. That finishes our determination of Figure~\ref{fig:eInfHomology}.

\subsubsection{Some homology operations}

As we have discussed in \cite[p.\ 199]{e2cellsIv3}, if $\gX$ is an $E_\infty$-algebra in $\cat{sMod}_\bZ^\bN$ and $\gX_{\bF_p} \coloneqq \gX \otimes_\bZ \bF_p$ is its reduction modulo $p$, then the mod 2 Dyer--Lashof operation $Q^1_2 \colon H_{1,0}(\gX_{\bF_2}) \to H_{2,1}(\gX_{\bF_2})$ has a refinement to an operation
\[(Q^1_2)_\bZ \colon H_{1,0}(\gX) \lra H_{2,1}(\gX)\]
defined on integral homology, satisfying $(Q^1_2)_\bZ(x) \,\mr{mod}\, 2 = Q^1_2(x\, \mr{mod}\,2)$. 

By universality, to define this operation it is enough to define a class 
\[(Q^1_2)_\bZ(\sigma) \in H_{2,1}(\gE_\infty(S^{1,0}_\bZ\sigma)),\]
but this group is $H_1(\fS_2;\bZ)=\bZ/2$ and we take $(Q^1_2)_\bZ(\sigma)$ to be the unique non-zero class, which indeed reduces to $Q^1_2(\sigma)$ modulo 2. We may similarly define an integral refinement $(\beta Q^1_3)_\bZ$ of the mod 3 Dyer--Lashof operation $\beta Q^1_3 \colon H_{1,0}(\gX_{\bF_3}) \to H_{3,3}(\gX_{\bF_3})$ by choosing the unique element
\[(\beta Q^1_3)_\bZ(\sigma) \in H_{3,3}(\gE_\infty(S^{1,0}_\bZ\sigma)) = H_3(\fS_3 ; \bZ) \cong \bZ/2 \oplus \bZ/3,\]
of order 3 which reduces modulo 3 to $\beta Q^1_3(\sigma)$. Both of these constructions give chain-level representatives for these operations, by choosing chains in $\gE_\infty(S^{1,0}_\bZ\sigma)$ once and for all. They induce analogous operations with coefficients in any commutative ring $\bk$.

\begin{lemma}\label{lem:SNrelations}
Let $(\R, \bk)$ satisfy the Nesterenko--Suslin property.
\begin{enumerate}[(i)]
\item There is an element $a \in H_{1,1}(\gR_\bk)$ such that $(Q^1_2)_\bk(\sigma) = \sigma \cdot a$, and $2a=0$.

\item There is an element $b \in H_{2,3}(\gR_\bk)$ such that $(\beta Q^1_3)_\bk(\sigma) = \sigma \cdot b$, and $3b=0$.
\end{enumerate}
\end{lemma}
\begin{proof}
Writing $C_p = \langle t \mid t^p \rangle$ for the cyclic group of order $p$, by the definition above the class $(Q^1_2)_\bk(\sigma)$ lies in the image of the map on homology induced by
\[t \longmapsto \begin{bmatrix} 0 & 1 \\
	1 & 0 \end{bmatrix} \colon C_2  \lra \mr{GM}_2(\R) \subset \mr{GL}_2(\R)\]
Changing to the basis $\{e_1 + e_2 , e_2\}$, this is conjugate to $\begin{bsmallmatrix} 1 & 1 \\0 & -1 \end{bsmallmatrix}$. By definition of the Nesterenko--Suslin property the homomorphisms $\begin{bsmallmatrix} * & 0 \\0 & * \end{bsmallmatrix} \to \begin{bsmallmatrix} * & * \\0 & * \end{bsmallmatrix} \to \begin{bsmallmatrix} * & 0 \\0 & * \end{bsmallmatrix}$ both induce isomorphisms on $\bk$-homology, so the homomorphisms $t \mapsto \begin{bsmallmatrix} 1 & 1 \\0 & -1 \end{bsmallmatrix}$ and $t \mapsto \begin{bsmallmatrix} 1 & 0 \\0 & -1 \end{bsmallmatrix}$ are the same on $\bk$-homology. Thus we can take $a = [-1] \in H_1(\mr{GL}_1(\R);\bk)$.

Similarly, $(\beta Q^1_3)_\bk(\sigma)$ lies in the image of the map on homology induced by
	\[t \longmapsto \begin{bmatrix} 0 & 0 & 1 \\
	1 & 0 & 0\\ 0 & 1 & 0 \end{bmatrix} \colon C_3 \lra \mr{GM}_3(\R) \subset \mr{GL}_3(\R)\]
Changing to the basis $\{e_1+e_2+e_3, e_2, e_3\}$, this is conjugate to $\begin{bsmallmatrix} 1 & 0 & 1 \\0 & 0 & -1\\ 0 & 1 & -1 \end{bsmallmatrix}$. By the Nesterenko--Suslin property  again this homomorphism induces the same map as $t \mapsto \begin{bsmallmatrix} 1 & 0 & 0 \\0 & 0 & -1\\ 0 & 1 & -1 \end{bsmallmatrix}$ on homology. Thus we can take $b \in H_2(\mr{GL}_2(\R);\bk)$ to be the image of a generator of $H_3(C_3 ; \bk)$ under $t \mapsto \begin{bsmallmatrix}  0 & -1\\  1 & -1 \end{bsmallmatrix} \colon C_3 \to \mr{GL}_2(\R)$.
\end{proof}

\subsubsection{Some attaching maps}

We also record the following, concerning the classes $\rho_2$ and $\rho_3$ in Figure~\ref{fig:eInfHomology}.

\begin{proposition}\label{prop:RhoCellsLooselyAttached}
Letting $\gR'_\bZ \coloneqq \gR_\bZ \cup_{\sigma}^{E_\infty} D^{1,1} \sigma'$, the Hurewicz maps
\begin{align*}
	H_{2,2}(\gR'_\bZ) &\lra H_{2,2}^{E_\infty}(\gR'_\bZ) \cong \bZ/2\\
	H_{3,4}(\gR'_\bZ) &\lra H_{3,4}^{E_\infty}(\gR'_\bZ) \cong \bZ/3
\end{align*}
are surjective.
\end{proposition}

\begin{proof}
  We may obtain a chart for $H_{*,*}^{E_\infty}(\gR'_\bZ)$ from Figure~\ref{fig:eInfHomology} by removing the copy of $\bZ$ in degree $(1,0)$. Let $\gA' \coloneqq \gA \cup_{\sigma}^{E_\infty} D^{1,1} \sigma'$, recalling that $\gA = \gE_\infty(S^{1,0} \otimes \bZ[BA^\times])$ was defined in \eqref{eqn:first-approx}.
    Then the Hurewicz map gives a map of long exact sequences
\[\begin{tikzcd} 
	H_{2,2}(\gA') \dar \rar &[-12pt] H_{2,2}(\gR'_\bZ) \rar \dar &[-8pt] H_{2,2}(\gR'_\bZ, \gA') \rar \dar{\simeq} &[-8pt] H_{2,1}(\gA')=0 \dar   \\
	0=H_{2,2}^{E_\infty}(\gA') \rar & H_{2,2}^{E_\infty}(\gR'_\bZ) \rar & H_{2,2}^{E_\infty}(\gR'_\bZ, \gA') \rar & H_{2,1}^{E_\infty}(\gA')=0,
\end{tikzcd}\]
and an application of Proposition $E_k$.11.9 shows the indicated vertical map is an isomorphism, this gives the first case.

For the second case, it is enough to work 3-locally, with $\gR_{\bZ_{(3)}}$; we implicity 3-localise $\gA$ too. We will extend $i \colon \gA \to \gR_{\bZ_{(3)}}$ to an improved approximation which takes into account the pre-Bloch group.

Let $0 \to P_1 \to P_0 \to \mathfrak{p}(\R)_{(3)} \to 0$ be a free resolution of $\mathfrak{p}(\R)_{(3)}$ as a $\bZ_{(3)}$-module, and $P_\bullet$ be the simplical $\bZ_{(3)}$-module corresponding to the chain complex $P_1 \to P_0$. The isomorphism $\mathfrak{p}(\R)_{(3)} \cong H_{2,3}(\gR_{\bZ_{(3)}}, \gA)$ is then induced by a map
\[S^{2,3} \otimes(P_\bullet) \lra \mathrm{Cone}(i),\]
or equivalently a map $\alpha\colon S^{2,2} \otimes P_\bullet \to \gA$ along with a nullhomotopy in $\gR_{\bZ_{(3)}}$. We may form an extension of the map $i$ to
	\[\gB \coloneqq \gA \cup^{E_\infty}_{\alpha} (D^{2,3} \otimes P_\bullet) \stackrel{j}\lra \gR_{\bZ_{(3)}}.\]

By construction we have $H_{n,d}(\gR_{\bZ_{(3)}}, \gB)=0$ for $n=1$, and for $n \geq 2$ with $d \leq 3$. In particular $j_* \colon H_{2,3}(\gB) \to H_{2,3}(\gR_{\bZ_{(3)}})$ is surjective, so the class $b \in H_{2,3}(\gR_{\bZ_{(3)}})$ from Lemma \ref{lem:SNrelations} (ii) may be lifted to a cycle $b' \colon S^{2,3} \to \gB$. We then form the cycle
	\[(\beta Q^1_3)_\bZ(\sigma) - b' \cdot \sigma \colon S^{3,3} \lra \gB,\]
which is nullhomotopic in $\gR_{\bZ_{(3)}}$ by Lemma \ref{lem:SNrelations} (ii). Choosing such a nullhomotopy $H$ of this cycle in $\gR_{\bZ_{(3)}}$, and using the canonical nullhomotopy of the same cycle in $\gB' \coloneqq \gB \cup^{E_\infty}_\sigma D^{1,1}_\bZ \sigma'$ induced by $\sigma = \partial(\sigma')$, gives a diagram
	\[\begin{tikzcd} 
	D^{3,4} \ar{rrd} & &  \\
	S^{3,3}  \dar \rar \uar \drar[phantom, "H"]& \gB  \dar \rar & \gB' \rar \dar& Q^{E_\infty}_\bL(\gB') \dar\\
	D^{3,4} \rar & \gR_{\bZ_{(3)}} \dar \rar & \gR'_{\bZ_{(3)}} \dar\rar & Q^{E_\infty}_\bL(\gR'_{\bZ_{(3)}}) \dar\\
	 & \gN  \rar& \gN' \rar & Q^{E_\infty}_\bL(\gN').
	\end{tikzcd}\]
The two maps $D^{3,4} \to \gR'_{\bZ_{(3)}}$ agreeing on $S^{3,3}$ give an element $x_H \in \smash{H_{3,4}(\gR'_{\bZ_{(3)}})}$.  Here we write $\gN$ for the 3-localisation of the object defined in Section~\ref{sec:e-infty-homology}, with $\gN(n) = \bZ_{(3)}$ for $n > 0$, and $\gN' \coloneqq \gN \cup^{E_\infty} D^{1,1}$.

We will now show that the image of $x_H$ in  $H_{3,4}^{E_\infty}(\gR'_{\bZ_{(3)}})$ is nontrivial. The inclusion 
	\[\gE \coloneqq \gE_\infty(S^{1,0}) \lra \gA \lra \gB\]
is a split monomorphism.  The splitting is induced by collapsing $B\R^\times \to \{*\}$, which annihilates the attaching map $\alpha$ for the pre-Bloch group, so may be extended to  a map of pairs $(\gR_{\bZ_{(3)}}, \gB) \to (\gN, \gE)$. Consider the zig-zag of maps
	\[\begin{tikzcd}[column sep=1.7em] 
	H_{3,4}(\gN, \gE) \dar& H_{3,4}(\gR_{\bZ_{(3)}}, \gB)  \lar \rar \dar& H_{3,4}(\gR'_{\bZ_{(3)}}, \gB')\dar & H_{3,4}(\gR'_{\bZ_{(3)}}) \lar \dar\\
	H_{3,4}^{E_\infty}(\gN, \gE) & H_{3,4}^{E_\infty}(\gR_{\bZ_{(3)}}, \gB)  \lar \rar{\sim}& H_{3,4}^{E_\infty}(\gR'_{\bZ_{(3)}}, \gB') & H_{3,4}^{E_\infty}(\gR'_{\bZ_{(3)}}) \lar[swap]{\sim}
	\end{tikzcd}\]
The image of $x_H$ under the fourth vertical map corresponds under the isomorphisms in the bottom row to the image of the nullhomotopy $[H] \in H_{3,4}(\gR_{\bZ_{(3)}}, \gB)$ under the second vertical map, so it is enough to show this is non-trivial. To do this it is enough to show the image of $[H]$ in $H_{3,4}^{E_\infty}(\gN, \gE)$ is nontrivial. In the composition
	\[H_{3,4}(\gR_{\bZ_{(3)}}, \gB) \lra H_{3,4}(\gN, \gE) \lra H_{3,4}^{E_\infty}(\gN, \gE),\]
the second map is an isomorphism as in Remark \ref{rem:EInfHomOfN}, and the first map is non-trivial because the composition
	\[H_{3,4}(\gR_{\bZ_{(3)}}, \gB) \lra H_{3,4}(\gN, \gE) \overset{\partial}\lra H_{3,3}(\gE)\]
sends $[H]$ to $(\beta Q^1_3)_\bZ(\sigma) \neq 0$ by construction.
\end{proof}

\subsection{The Nesterenko--Suslin theorem}\label{sec:NS}

In this section we will explain how the $E_\infty$-homology calculations Section \ref{sec:EinfHomologyCalc} can be used to recover the following theorem of Nesterenko--Suslin \cite{SN}, and Guin \cite{Guin}.  The case where $\R$ is an infinite field was proved earlier by Suslin \cite[Theorem 3.4]{SuslinCharClass}.

To state it, recall that as in \cite[Section 3]{SN} the \emph{Milnor $K$-theory} $K^M_*(\R)$ of a ring $\R$ is the quotient of the graded tensor algebra $T^*_\bZ(\R^\times)$ by the homogeneous ideal generated by the \emph{Steinberg relations}
	\[x \otimes (1-x) \in T^2_\bZ(\R^\times) \quad \text{ for }\quad x,1-x \in \R^\times.\]
The $\bZ$-span of these relations contains $x \otimes y + y \otimes x$ for all $x, y \in \R^\times$, by \cite[Lemma 3.2]{SN}, so the ring $K^M_*(\R)$ is graded-commutative. Recall that we write $\tilde{\wedge}^n M \coloneqq [M^{\otimes n} \otimes \bZ^{-}]_{\fS_n}$, so that $\tilde{\wedge}^* M \coloneqq \bigoplus_{n \geq 0} \tilde{\wedge}^n M$ is the free graded-commutative $\bZ$-algebra on a $\bZ$-module $M$ placed in degree 1, we can therefore also define Milnor $K$-theory as the corresponding quotient of $\tilde{\wedge}^* \R^\times$. In the case of fields the following is \cite[Theorem 3.4]{SuslinCharClass}, for local rings it is \cite[Theorems 2.7, 3.25]{SN}, and in general it is a special case of \cite[Th{\'e}or{\`e}me 2]{Guin}.

\begin{theorem}[Suslin, Nesterenko--Suslin, Guin] \label{thm:NS}
Let $\R$ be a connected semi-local ring with all residue fields infinite. Then
\begin{align*}
	H_{d}(\mr{GL}_n(\R), \mr{GL}_{n-1}(\R);\bZ) &=0 \text{ for } d < n, \text{ and}\\
	H_{n}(\mr{GL}_n(\R), \mr{GL}_{n-1}(\R);\bZ) &\cong K_n^M(\R).
\end{align*}
\end{theorem}

Recall that $\overline{\gR}_\bZ$ is the associative algebra obtained by rectifying the unital $E_\infty$-algebra $\gR_\bZ^+$, and $\overline{\gR}_\bZ/\sigma$ is the cofiber of the stabilisation map. Thus, in our terminology, the Nesterenko--Suslin theorem is the vanishing of $H_{n,d}(\overline{\gR}_\bZ/\sigma)$ for $d < n$ and an isomorphism of $H_{n,n}(\overline{\gR}_\bZ/\sigma)$ with $n$th Milnor $K$-theory. Let us make explicit what this isomorphism will be.

We write $D_n \coloneqq H_{n,n}(\overline{\gR}_\bZ)$ for the $n$th diagonal homology of the associative algebra $\overline{\gR}_\bZ$, so that $D_* \coloneqq \bigoplus_{n \geq 0} D_n$ is a graded-commutative ring. Similarly, we write $M_n \coloneqq H_{n,n}(\overline{\gR}_\bZ/\sigma)$ for the $n$th diagonal homology of the $\overline{\gR}_\bZ$-module $\overline{\gR}_\bZ/\sigma$, so that $M_* \coloneqq \bigoplus_{n \geq 0} M_n$ is a graded $D_*$-module. The quotient map gives a map of $D_*$-modules $D_* \to M_*$. As the ring $D_*$ is graded-commutative the identification $\R^\times \overset{\sim}\to H_{1,1}(\overline{\gR}_\bZ) = D_1$ extends to a ring homomorphism $\tilde{\wedge}^* \R^\times \to D_*$, and hence to a $\tilde{\wedge}^* \R^\times$-module homomorphism
	\[\phi\colon \tilde{\wedge}^* \R^\times \lra D_* \lra M_*.\]
We will show that this map is surjective, and that its kernel is the $\tilde{\wedge}^* \R^\times$-submodule generated by the Steinberg relations.

\begin{proof}[Proof of Theorem \ref{thm:NS}]
	In Section \ref{sec:e-infty-homology} we have discussed the map $\epsilon \colon \gR_\bZ \to \gN$ of $E_\infty$-algebras induced by the augmentations $B\mr{GL}_n(\R) \to \{*\}$ (these go through for $\R$ connected semi-local in exactly the same way). Theorem \ref{thm:proof-of-conjecture-rank-leq-3} shows that $\smash{H^{E_2}_{n,d}(\gN, \gR_\bk)}=0$ for $d<2n$ and $d \leq 3$, and also for $d < 2n-1$, and the argument of Corollary \ref{cor:EInfHomNrelR} shows that $H^{E_\infty}_{n,d}(\gN, \gR_\bZ)$ vanishes in the same range of bidegrees. This is depicted in Figure \ref{fig:eInfRelHomology}.  (Conjecture~\ref{conj:double-steinberg-local} would imply vanishing for $d < 2n$ for all $d \geq 0$, and that the slope-two line through the origin in the figure continues past $d = 3$.)  That the entry $(n,d) = (2,4)$ is $\mathfrak{p}(\R)$ will be established later and is not needed right now (see Remark~\ref{remark:entry-two-four}).

\begin{figure}[h]
	\begin{tikzpicture}
	\begin{scope}
	\clip (-2,-1) rectangle ({2.5*4+2},7.5);
	\draw (0,0)--(10.5,0);
	\draw (0,0) -- (0,7.5);
	\foreach \s in {0,...,7}
	{
		\draw [dotted] (-1.5,\s)--(10.5,\s);
		\node [fill=white] at (-1.5,\s) [left] {\tiny $\s$};
	}
	\foreach \s in {0,...,4}
	{
		\draw [dotted] ({2.5*\s},-0.5)--({2.5*\s},7.5);
		\node [fill=white] at ({2.5*\s},-.5) {\tiny $\s$};
	}
	
	\draw [very thick,Mahogany,densely dotted] (0,0) -- (6*1.25,6) -- (7*1.25,6) -- (9*1.25,8);

	\node [fill=white] at (2.5,2) {$\R^\times$};
	\node [fill=white] at (2.5,3) {$H_2(\R^\times;\bZ)$};
	\node [fill=white] at (2.5,4) {$H_3(\R^\times;\bZ)$};
	\node [fill=white] at (2.5,5) {$H_4(\R^\times;\bZ)$};
	\node [fill=white] at (2.5,6) {$H_5(\R^\times;\bZ)$};
	\node [fill=white] at (2.5,7) {$H_6(\R^\times;\bZ)$};
	\node [fill=white] at (5,4) {$\mathfrak{p}(\R)$};  
	\node [fill=white] at (5,5) {?};
	\node [fill=white] at (5,6) {?};
	\node [fill=white] at (5,7) {?};
	\node [fill=white] at (7.5,6) {?};
	\node [fill=white] at (7.5,7) {?};
	\node [fill=white] at (10,7) {?};

	\node at (-.5,-.5) {$\nicefrac{d}{n}$};
	\end{scope}
	\end{tikzpicture}
	\caption{The $E_\infty$-homology of the pair $(\gN,\gR_\bZ)$, which vanishes below the dotted line. On the line $d=n+1$ it just has $\R^\times$ in bidegree $(1,2)$.}
	\label{fig:eInfRelHomology}
\end{figure}

We may therefore construct a relative CW-$E_\infty$-algebra $\gR_\bZ \hookrightarrow \gC \overset{\sim}\to \gN$ by attaching $(n,d)$-cells only in the necessary bidegrees. This has a skeletal filtration $\mathrm{sk}(\gC)$ and the spectral sequence (see Theorem $E_k$.10.10) of the filtered object $\overline{\mathrm{sk}(\gC)}/\sigma$ then has the form
	\[E^1_{n,p,q} = H_{n,p+q,p}\left(\overline{\gR}_\bZ/\sigma[0] \otimes E_\infty^+(\bigoplus S^{n_\alpha, d_\alpha, d_\alpha}x_\alpha)\right) \Longrightarrow H_{n,p+q}(\overline{\gN}/\sigma).\]
Here it may be helpful to recall from Section \ref{sec:notat-other-recoll} that for fixed $n$ the grading conventions of our spectral sequences are as those for the homological Serre spectral sequence.

Since $\overline{\gN}/\sigma \simeq \bZ[0,0]$, the spectral sequence has $E^\infty_{0,0,0} = \bZ$ and $E^\infty_{n,p,q} = 0$ for $(n,p,q) \neq (0,0,0)$.  All the cells satisfy $n_\alpha \geq 1$ and $d_\alpha \geq n_\alpha+1$, as shown in Figure \ref{fig:eInfRelHomology}, and the groups we are interested in appear as $E^1_{n,0,d} = H_{n,d}(\overline{\gR}_\bZ/\sigma)$.

The vanishing part of the theorem is the following.

\vspace{1ex}

\noindent \textbf{Claim}. $H_{n,d}(\overline{\gR}_\bZ/\sigma) = 0$ for $d < n$.

\begin{proof}[Proof of Claim]
Suppose for a contradiction that $H_{n,d}(\overline{\gR}_\bZ/\sigma) \neq 0$ for some $d < n$, and let $d$ be minimal with this property. Then $E^1_{n,0,d} \neq 0$, so consider differentials $d^r \colon E^r_{n,r,d-r+1} \to E^r_{n,0,d}$ going into the first column.

The K{\"u}nneth exact sequence expresses $E^1_{n,r,d-r+1}$ as an extension of
	\[\bigoplus_{\substack{n'+n''=n \\ d'+d''=d+1}}H_{n', d'}(\overline{\gR}_\bZ/\sigma) \otimes H_{n'', d'', r}(\gE_\infty^+(\bigoplus S^{n_\alpha, d_\alpha, d_\alpha}x_\alpha))\]
by
	\[\bigoplus_{\substack{n'+n''=n \\ d'+d''=d}}\mr{Tor}^\bZ_1\left(H_{n', d'}(\overline{\gR}_\bZ/\sigma) , H_{n'', d'', r}(\gE_\infty^+(\bigoplus S^{n_\alpha, d_\alpha, d_\alpha}x_\alpha))\right),\]
and $E^r_{n,r,d-r+1}$ is a subquotient of $E^1_{n,r,d-r+1}$. We will show that these are zero.

In the first case consider the term
	\[H_{n', d'}(\overline{\gR}_\bZ/\sigma) \otimes H_{n'', d'', r}(\gE_\infty^+(\bigoplus S^{n_\alpha, d_\alpha, d_\alpha}x_\alpha))\]
with $n'+n''=n$ and $d'+d''=d+1$. The tridegrees of the $x_\alpha$ ensure that the right-hand factor can be non-trivial only if $r \geq 2$, $d'' \geq n''+1$, and $d'' \geq r$. But in this case
	\[d'-n' = d+1-d'' - (n-n'') = (d-n) + (n''+1-d'')<0\]
so $d'<n'$, and $d' = d+1-d'' \leq d+1-r < d$, so the left-hand factor must be zero, by our assumption that $d$ was minimal. 

Similarly, for the term
	\[\mr{Tor}^\bZ_1\left(H_{n', d'}(\overline{\gR}_\bZ/\sigma) , H_{n'', d'', r}(\gE_\infty^+(\bigoplus S^{n_\alpha, d_\alpha, d_\alpha}x_\alpha))\right)\]
with $n'+n''=n$ and $d'+d''=d$, if the right-hand factor is non-trivial then
	\[d'-n' = d-d'' - (n-n'') = (d-n) + (n''-d'')<-1\]
so again $d' < n'$, and $d' = d-d'' \leq d-r < d$, so the left-hand factor must again be trivial.

Thus the domains of differentials arriving at $E^r_{n,0,d}$ are zero for all $r$, so $E^\infty_{n,0,d} \neq 0$ and hence $H_{n,d}(\overline{\gN}/\sigma) \neq 0$, a contradiction.
\end{proof}

For the remaining part we will analyse the differentials in this spectral sequence. The only differentials entering $E^1_{n,0,n} = H_{n,n}(\overline{\gR}_\bZ/\sigma)$ are of the form $d^r \colon E^r_{n, r,n-r+1,} \to E^r_{n,0,n}$. By the K{\"u}nneth theorem, using the vanishing just established and the fact that $d_\alpha \geq n_\alpha+1$, we have
	\[E^1_{n,r,n-r+1} = \bigoplus_{a+b=n} H_{a,a}(\overline{\gR}_\bZ/\sigma) \otimes H_{b, b+1,r}(\gE_\infty^+(\bigoplus S^{n_\alpha, d_\alpha, d_\alpha}x_\alpha)).\]
The only cells satisfying $d_\alpha=n_\alpha+1$ are those for $(n_\alpha,d_\alpha)=(1,2)$, so this is only nonzero for $r=2$, where we can write $E^1_{n,2,n-1} = H_{n-1,n-1}(\overline{\gR}_\bZ/\sigma) \otimes E^1_{1,2,2}$. Thus the only differential entering $E^1_{n,0,n} = H_{n,n}(\overline{\gR}_\bZ/\sigma)$ is
\begin{equation}
	d^2 \colon E^2_{n,2,n-1} = H_{n-1,n-1}(\overline{\gR}_\bZ/\sigma) \otimes E^2_{1,2,0} \lra E^2_{n,0,n} = H_{n,n}(\overline{\gR}_\bZ/\sigma).\label{eq:31}
\end{equation}

Note that the differential
	\[d^1 \colon E^1_{1,3,0} = \bigoplus_{(n_\alpha, d_\alpha)=(1,3)} \bZ \lra E^1_{1,2,0} = \bigoplus_{(n_\alpha, d_\alpha)=(1,2)} \bZ\]
is that of the relative $E_\infty$-cellular chain complex of $(\gN, \gR_\bZ)$ in these degrees, giving that $E^2_{1,2,0} \cong H^{E_\infty}_{1,2}(\gN, \gR_\bZ) = \R^\times$.  Therefore the differential~\eqref{eq:31} gives a map
	\[H_{n-1,n-1}(\overline{\gR}_\bZ/\sigma) \otimes \R^\times \lra H_{n,n}(\overline{\gR}_\bZ/\sigma),\]
which must be surjective for all $n \geq 1$, since $E^\infty_{n,0,n}= 0$. By the multiplicative properties of the spectral sequence (module structure over $H_{*,*}(\overline{\gR}_\bZ)$) this proves the generation part of the theorem.

\vspace{.5em}

For the relations part of the theorem, we first note that the long exact sequence for the triple $\mr{GM}_1(\R) \leq \mr{GM}_2(\R) \leq \mr{GL}_2(\R)$ together with Suslin's formula \eqref{eq:AttMap} shows that the map
	\[\tilde{\wedge}^2 \R^\times = H_2(\mr{GM}_2(\R),\mr{GM}_1(\R)) \lra H_2(\mr{GL}_2(\R),\mr{GL}_1(\R)) = H_{2,2}(\overline{\gR}_\bZ/\sigma)\]
annihilates $x \tilde{\wedge} (1-x)$ for all $x \in \R^\times \setminus \{1\}$, and therefore the kernel of the surjective left $\tilde{\wedge}^* \R^\times$-module map
	\[\tilde{\wedge}^* \R^\times \lra \bigoplus_{n \geq 0} H_{n,n}(\overline{\gR}_\bZ/\sigma)\]
contains the $\tilde{\wedge}^* \R^\times$-submodule generated by the Steinberg relations $x \tilde{\wedge} (1-x)$. It remains to show that the kernel is no larger.

To do so, consider the $E_\infty$-algebra $\gR' \coloneqq {\gR_\bZ} \cup^{E_\infty}_\sigma D^{1,1}\sigma'$. The universal property of module quotients provides an $\overline{\gR}_\bZ$-module map
\[\overline{\gR}_\bZ/\sigma \lra \overline{\gR}' = \overline{{\gR_\bZ} \cup^{E_\infty}_\sigma D^{1,1}\sigma'}.\]
Let us write $M(\bZ/2)$ for the (simplicial abelian group corresponding to the) chain complex $(\bZ \smash{\xrightarrow{2}} \bZ)$ and similarly for $M(\bZ/3)$.  Proposition \ref{prop:RhoCellsLooselyAttached} shows that we may choose maps $\rho_2'\colon S^{2,2} \otimes M(\bZ/2) \to \gR'$ and $\rho_3' \colon S^{3,4} \otimes M(\bZ/3) \to \gR'$ representing non-trivial $E_\infty$-homology classes.  We then define $\gR' \to \gR''$ by taking pushout with $D^{2,3} \otimes M(\bZ/2)$
 along $\rho'_2$ and cone off $\rho'_3$ in a similar way.  The $E_\infty$-homology of $\gR''$ then looks like Figure~\ref{fig:eInfHomology}, except the entries $(1,0)$, $(2,2)$, and $(3,4)$ have been set to zero.

Recalling $\gA = \gE_\infty(S^{1,0} \otimes \bZ[B\R^\times])$ and defining
\begin{equation*}
	\gA'' \coloneqq \gA \cup_{\sigma}^{E_\infty} D^{1,1} \sigma',
\end{equation*}
the map $\gA \to \gR$ extends to a map $\gA'' \to \gR''$. In the long exact sequence
\[\cdots \lra H^{E_\infty}_{n,d}(\gA'') \lra H^{E_\infty}_{n,d}(\gR'') \lra H^{E_\infty}_{n,d}(\gR'',\gA'') \lra H^{E_\infty}_{n,d-1}(\gA'') \lra \cdots\]
the terms $\smash{H^{E_\infty}_{n,d}(\gA'')}$ vanish unless $n=1$, in which case the left map is an isomorphism. We conclude that 
\begin{equation*}
H^{E_\infty}_{n,d}(\gR'',\gA'') \cong \begin{cases} 0 & \text{if $n=1$,} \\
	H^{E_\infty}_{n,d}(\gR'') & \text{otherwise.}\end{cases}
\end{equation*}
Inspecting Figure \ref{fig:eInfHomology}, we see that $(n,d) = (2,3)$ is the only non-zero entry in bidegrees $d \leq n+1$, implying that $\gR''$ may be obtained, up to equivalence, from $\gA''$ by attaching cells of bidegree $(n,d)$ having $d > n+1$, after attaching $(2,3)$-cells along the connecting homomorphism
\begin{equation}\label{eq:33}
  H^{E_\infty}_{2,3}(\gR'', \gA'') \overset{\cong}\longleftarrow H_{2,3}(\gR'', \gA'') \overset{\partial}\lra H_{2,2}(\gA'') = \tilde{\wedge}^2 \R^\times.
\end{equation}
The isomorphism follows from the Hurewicz theorem for $E_k$-homology (Corollary $E_k$.11.12). To understand the map $\partial$, consider the commutative diagram
\[\begin{tikzcd}
	H_{2,3}^{E_\infty}(\gR,\gA) \dar{\cong} & \lar[swap]{\cong} H_{2,3}(\gR, \gA) \rar \dar & H_{2,2}(\gA) \dar \\
	H_{2,3}^{E_\infty}(\gR'',\gA'') & \lar[swap]{\cong} H_{2,3}(\gR'', \gA'') \rar{\partial} & H_{2,2}(\gA'')
\end{tikzcd}\]
where the top-left horizontal map is an isomorphism by the discusison in Section~\ref{sec:determ-figure-reff} around equation \eqref{eq:HurewiczInGrading2}, the bottom-left horizontal map is an isomorphism by the Hurewicz theorem as just discussed, and the left vertical map is an isomorphism by comparison of the long exact sequences.
Arguing as in Section~\ref{sec:PreBloch} then gives the identification
\[\begin{tikzcd}H_{2,3}(\gR'', \gA'')  \rar{\partial} \dar{\cong} & H_{2,2}(\gA'') \dar{\cong} & \lar H_{2,2}(\gA) \dar{\cong} \\
	H_3(\mr{GL}_2(\R),\mr{GM}_2(\R)) \rar & H_2(\mr{GM}_2(\R),\mr{GM}_1(\R)) & \lar H_2(\mr{GM}_2(\R))\end{tikzcd}\]
    under which the connecting homomorphism~\eqref{eq:33} becomes identified with the second component of \eqref{eq:AttMap}.
    
Since all cells of $\gR''$ are on or above the line $d = n$, it is easy to calculate the graded commutative ring $\oplus_{n \geq 0} H_{n,n}(\gR'')$: it is generated by the cells with $d=n$ and subject only to the relations coming from attaching maps of cells with $d = n+1$.  By inspecting Figure~\ref{fig:eInfHomology} again, and by the above analysis of the attaching map, we therefore get a ring isomorphism
	\[\bigoplus_{n \geq 0} H_{n,n}(\overline{\gR}'') \cong \frac{\bigoplus_{n \geq 0} \tilde{\wedge}^n \R^\times}{(x \tilde{\wedge} (1-x) \text{ for } x \in \R^\times \setminus \{1\})}\]
which is a presentation of Milnor $K$-theory. The composition
	\[\tilde{\wedge}^* \R^\times \lra \bigoplus_{n \geq 0} H_{n,n}(\overline{\gR}_\bZ/\sigma) \lra \bigoplus_{n \geq 0} H_{n,n}(\overline{\gR}'') = K_*^M(\R)\]
has kernel the submodule generated by the Steinberg relations, so the first map does too.
\end{proof}

\subsection{One degree higher}\label{sec:NSplus1}

The $E_\infty$-homology calculations in Section \ref{sec:EinfHomologyCalc} may also be used to address the relative rational homology groups $H_{n+1}(\mr{GL}_{n}(\R),\mr{GL}_{n-1}(\R);\bQ)$, in the degree above that addressed by the Nesterenko--Suslin theorem. Our answer is described in terms of the third Harrison homology of the (graded) commutative ring $K_*^M(\R)_\bQ$; this is Theorem \ref{thm:Harrison}.

\begin{theorem}\label{thm:NSPlusOne}
Let $\R$ be a connected semi-local ring with all residue fields infinite. The $\tilde{\wedge}^* (\R^\times)_\bQ$-module structure on $\bigoplus_{n \geq 0} H_{n,n+1}(\overline{\gR}_\bQ/\sigma)$ descends to a $K_*^M(\R)_\bQ$-module structure. Furthermore there is a natural homomorphism of graded $\bQ$-vector spaces
  \begin{equation*}
    \bigoplus_{n \geq 0} \mr{Harr}_{3}(K_*^M(\R)_\bQ)_n \lra \bQ \otimes_{K_*^M(\R)_\bQ} \bigoplus_{n \geq 0}  H_{n,n+1}(\overline{\gR}_\bQ/\sigma),
  \end{equation*}
  which is an isomorphism for $n \geq 5$. If $\R$ is an infinite field then this map is an isomorphism for $n \geq 4$.
\end{theorem}

If $\R$ is a field of positive characteristic then it is a consequence of conjectures of Parshin and Beilinson that the quadratic algebra $K_*^M(\R)_\bQ$ is Koszul (see \cite[Section 0.6]{PositselskiAT}). When $K_*^M(\R)_\bQ$ is Koszul, $\mr{Harr}_{3}(K_*^M(\R)_\bQ)_n$ is supported in grading $n=3$. Assuming this Koszulness property gives the following.

\begin{corollary}\label{cor:NSPlusOneKoszul}
  If $K_*^M(\R)_\bQ$ is Koszul then $\bigoplus_{n \geq 0} H_{n,n+1}(\overline{\gR}_\bQ/\sigma)$ is generated as a $K_*^M(\R)_\bQ$-module in grading
  $n \leq 4$ (grading $n \leq 3$ if $\R$ is a field).
\end{corollary}

\begin{proof}
Letting $\gR' \coloneqq {\gR}_\bQ \cup^{E_\infty}_\sigma D^{1,1} \sigma'$, the natural map
	\[\overline{\gR}_\bQ/\sigma \lra \overline{\gR}'\]
\sloppy is a weak equivalence. This may be seen by considering the induced map of cell-attachment spectral sequences (see Section $E_k$.10.3.2) in the categories of $\overline{\gR}_\bQ$-modules and $E_\infty$-algebras respectively which, together with the identification $H_{*,*}(\gE_\infty^+(S^{1,1}_\bQ\sigma')) \simeq \Lambda_\bQ[\sigma']$ where $\sigma'$ has bidegree $(1,1)$, both have $E^1$-pages given by $H_*(\gR_\bQ) \otimes \Lambda_\bQ[\sigma']$ with $d^1(x \otimes \sigma') = x \cdot \sigma$, and degenerate at $E^2$.

Now $\gR'_\bQ$ is an $E_\infty$-algebra having $E_\infty$-homology only in bidegrees $(n,d)$ such that $d \geq n$, and by Theorem \ref{thm:NS} having
	\[\bigoplus_{n \geq 0} H_{n,n}(\overline{\gR}') = K_*^M(\R) \otimes \bQ\]
as a graded ring. This in particular shows that the $\tilde{\wedge}^* (\R^\times)_\bQ$-module structure on
	\[\bigoplus_{n \geq 0} H_{n,n+1}(\overline{\gR}_\bQ/\sigma) = \bigoplus_{n \geq 0} H_{n,n+1}(\overline{\gR}')\]
descends to a $K_*^M(\R)_\bQ$-module structure.

We may attach $E_\infty$-$(n,d)$-cells with $d \geq n+2$ to $\gR'$ to form an $E_\infty$-algebra $\gK$ having
	\[H_{n,d}(\gK) = \begin{cases}
	K_n^M(\R) \otimes \bQ & \text{if } n=d, n>0,\\
	0 & \text{else}.\end{cases}\]
Let $\gI$ denote the homotopy fibre of the $\overline{\gR}'$-module map $\overline{\gR}' \to \overline{\gK}$. We then have that $H_{n,d}(\gI)=0$ for $d \leq n$, and on the line above this we have
\[H_{n,n+1}(\gI) \overset{\sim}\lra H_{n,n+1}(\overline{\gR}').\]
The fibration sequence $\gI \to \overline{\gR}' \to \overline{\gK}$ of $\overline{\gR}'$-modules is also a cofibration sequence, so there is a cofibration sequence
\begin{equation}\label{eq:CofSeqKoszul}
	B(\bQ, \overline{\gR}', \gI) \lra B(\bQ, \overline{\gR}', \overline{\gR}') \simeq \bQ \lra B(\bQ, \overline{\gR}', \overline{\gK}).
\end{equation}
We can compute the leftmost term by the bar spectral sequence
	\[E^1_{n,p,q} = \mathrm{Tor}_p^{H_{*,*}(\overline{\gR}')}(\bQ[0,0], H_{*,*}(\gI))_{n,q} \Longrightarrow H_{n,p+q}^{\overline{\gR}'}(\gI)\]
which along the line $p+q=n+1$ yields an identification
	\[\bigoplus_{n \geq 0} H_{n,n+1}^{\overline{\gR}'}(\gI) = \bQ \otimes_{K^M_*(\R)_\bQ} \left(\bigoplus_{n \geq 0} H_{n,n+1}(\gI)\right) = \bQ \otimes_{K^M_*(\R)_\bQ} \left(\bigoplus_{n \geq 0} H_{n,n+1}(\overline{\gR}')\right)\]
and combining this with the cofibration sequence \eqref{eq:CofSeqKoszul} and the discussion above gives an identification
	\[\bQ \otimes_{K^M_*(\R)_\bQ} \left(\bigoplus_{n \geq 0} H_{n,n+1}(\overline{\gR}_\bQ/\sigma)\right) \cong \bigoplus_{n \geq 0} H_{n,n+2}^{\overline{\gR}'}(\overline{\gK}).\]

To access this latter term we use Theorem $E_k$.15.6, which provides a strongly convergent spectral sequence
\begin{equation}\label{eq:CanMultSS}
	E^1_{n,p,q} = \mr{S}^*[(-1)_* H^{E_\infty}_{*,*}(\gK, \gR')]_{n,p+q,p} \Longrightarrow H_{n,p+q}^{\overline{\gR}'}(\overline{\gK})
\end{equation}
with differentials $d^r \colon E^r_{n,p,q} \to E^r_{n,p-r,q+r-1}$ and $\mr{S}^*[-]$ denoting the free graded-commutative algebra.
As $\gK$ is obtained from $\gR'$ by attaching $E_\infty$-$(n,d)$-cells with $d \geq n+2$, we have $H^{E_\infty}_{n,d}(\gK, \gR')=0$ for $d < n+2$. Thus $E^1_{n,-1,p}=0$ for $p-1 < n+2$, and by taking products we find that
\begin{equation}\label{eq:E1Vanish}
	E^1_{n,p,q}=0 \text{ for } 3p+q < n.
\end{equation}
Furthermore, along the line $d=n+2$ we have an exact sequence
	\[\cdots \lra H^{E_\infty}_{n,n+2}(\gR') \lra H^{E_\infty}_{n,n+2}(\gK) \lra H^{E_\infty}_{n,n+2}(\gK, \gR')\lra H^{E_\infty}_{n,n+1}(\gR') \lra \cdots.\]
In characteristic zero $E_\infty$-homology agrees with Harrison hyperhomology \cite{HarrisonCoh,QuillenCoh, Fresse} up to a shift of degrees, so there is for any $\gA \in \Alg_{E_\infty}(\cat{sMod}_\bQ^\bN)$ a hyperhomology spectral sequence
$${{}^{\prime}E^2_{n,p,q}} = \mr{Harr}_p(H_{*,*}(\gA))_{n, q} \Longrightarrow H^{E_\infty}_{n,p+q-1}(\gA)$$
with differentials $d^r : {{}^{\prime}E^r_{n,p,q}} \to {{}^{\prime}E^r_{n,p-r,q+r-1}}$. For $\gA=\gK$, the fact that $H_{*,*}(\gK)$ is supported in diagonal bidegrees, and along the diagonal is given by $K_*^M(\R)_\bQ$, means that this spectral sequence collapses and gives
\[H^{E_\infty}_{n,n+2}(\gK) \cong \mr{Harr}_{3}(K_*^M(\R)_\bQ)_n.\] On the other hand $H^{E_\infty}_{n,n+2}(\gR') = H^{E_\infty}_{n,n+1}(\gR') = 0$ as long as $n \geq 5$, by Figure \ref{fig:eInfHomology}. Furthermore, if $\R$ is a field then this vanishing holds for $n \geq 4$, by Corollary \ref{cor:einfty-hom-indec}. This provides a map
	\[\mr{Harr}_{3}(K_*^M(\R)_\bQ)_n \lra H^{E_\infty}_{n,n+2}(\gK, \gR') = E^1_{n,-1,n+3}\]
which is an isomorphism for $n \geq 5$ ($n \geq 4$ if $\R$ is a field). There are no possible differentials entering this position, and differentials leaving this position go to $E^r_{n,-1-r,n+3+r-1}$, but these groups are zero by \eqref{eq:E1Vanish}.  Therefore the edge homomorphism
\begin{equation*}
	H_{n,n+2}^{\overline{\gR}'}(\overline{\gK}) \twoheadrightarrow E^\infty_{n,-1,n+3} \subset E^1_{n,-1,n+3}
\end{equation*}
is surjective.  A similar argument shows that $E^1_{n,q,n+2-q}=0$ for $q<-1$, which proves that this edge homomorphism is injective.  We therefore obtain
	\[\mr{Harr}_{3}(K_*^M(\R)_\bQ)_n \lra H^{E_\infty}_{n,n+2}(\gK, \gR') \stackrel{\cong}{\longleftarrow} H_{n,n+2}^{\overline{\gR}'}(\overline{\gK})\]
which is an isomorphism in the claimed range of degrees.
\end{proof}

\subsection{Rings with vanishing rational $K_2$}\label{sec:TrivK2Rat}

If $\R$ is a field or is a semi-local ring with infinite residue fields, then $K_2^M(\R) = K_2(\R)$ (for fields this is Matsumoto's theorem \cite[\S 11]{MilnorKThy}, and for semi-local rings with infinite residue fields it follows from \cite[Theorem 7.1, 8.4]{vanderKallenK2} or \cite[Corollary 4.3]{SN}). Thus if $K_2(\R)_\bQ=0$ then $K_2^M(\R)_\bQ=0$ and so the algebra $K^M_*(\R)_\bQ$ is Koszul. 

\begin{example}\label{exam:vanishing-k2}
For $\R =\bF$ a field, the assumption $K_2(\bF)_\bQ = 0$ is satisfied for $\bF$ a finite field, a number field (using localisation and then Borel's calculation that $K_2(\mathcal{O}_\bF) \otimes \bQ=0$), or $\bF_q(t)$, and hence by taking colimits also for the algebraic closures $\bar{\bQ}$ or $\bar{\bF}_p$. If $\bF$ has infinite transcendence degree over $\bF_p$ or $\bQ$, then $K_2(\bF)_\bQ \neq 0$. For these and more examples see \cite[III.\S 6]{WeibelK}.\end{example}

If $\R$ is a connected semi-local ring with all residue fields infinite and such that $K_2(\R)_\bQ=0$, then Corollary \ref{cor:NSPlusOneKoszul} implies that $H_{n+1}(\mr{GL}_n(\R), \mr{GL}_{n-1}(\R);\bQ)=0$ as long as $n \geq 5$.
 Here we will show that under this hypothesis one in fact obtains a much stronger homological stability range, of slope greater than 1; this is Theorem~\ref{thm:stab-special-assumptions}~(i).

\begin{theorem}\label{thm:K2vanish}
If $\R$ is a connected semi-local ring with all residue fields infinite and such that $K_2(\R)_\bQ=0$, then
	\[H_d(\mr{GL}_n(\R), \mr{GL}_{n-1}(\R);\bQ)=0\]
in degrees $d < \tfrac{4n-1}{3}$.
\end{theorem}

\begin{proof}
We use the map $\gR' \to \gK$ constructed in the proof of Theorem \ref{thm:NSPlusOne}, where 
	\[H_{n,d}(\gR') = H_{n,d}(\overline{\gR}_\bQ/\sigma)=H_d(\mr{GL}_n(\R), \mr{GL}_{n-1}(\R);\bQ).\]
This map may be obtained as a relative CW-$E_\infty$-algebra by attaching $E_\infty$-$(n,d)$-cells to $\gR'$ with $d \geq n+2$, so $H^{E_\infty}_{n,d}(\gK, \gR')=0$ for $d < n+2$. Furthermore, we have $H^{E_\infty}_{1,0}(\gR')=0$ by construction, and we have $H^{E_\infty}_{n,d}(\gR')=0$ for $d < 2n-2$ and for $(n,d)=(2,2)$ or $(3,4)$, by Figure \ref{fig:eInfHomology}. Under the assumption $K_2(\R)_\bQ=0$, as we have discussed above, the algebra $K_*^M(\R)_\bQ$ is Koszul so $H^{E_\infty}_{n,d}(\gK)=0$ for $d \neq 2n-1$. Combining these estimates, using the long exact sequence on $E_\infty$-homology for the pair $(\gK, \gR')$, we see that $H^{E_\infty}_{n,d}(\gK, \gR')=0$ for $d-1 < \tfrac{4}{3}n$.
 
We now consider the spectral sequence associated to the skeletal filtration of Corollary $E_k$.10.19
\[E^1_{n,p,q} = H_{n,p+q,p}\left(\gR'[0] \otimes \gE_\infty^+(\bigoplus_\alpha S_\bQ^{n_\alpha, d_\alpha, d_\alpha}x_\alpha)\right) \Longrightarrow H_{n,p+q}(\gK)\]
with $d_\alpha \geq n_\alpha+2$ and $d_\alpha-1 \geq \tfrac{4}{3}n_\alpha$, and differential $d^r \colon E^r_{n,p,q} \to E^r_{n,p-r,q+r-1}$. The target of this spectral sequence is non-trivial only in bidegrees $(0,0)$ (where it is $\bQ$) and $(1,1)$ (where it is $(\R^\times)_\bQ$).

Suppose for a contradiction that $H_{n,d}(\gR')$ is nonzero, with $n>1$ and $\tfrac{d-1}{n-1} < \tfrac{4}{3}$, contributing to the zeroeth column; without loss of generality we may suppose that $d$ is minimal with this property. As $n>1$ this cannot survive the spectral sequence, so must be the target of a nontrivial $d^r$-differential starting at
	\[E^1_{n,r,d-r+1} = \bigoplus_{\substack{n'+n'' = n\\ d'+d''=d+1}}H_{n',d'}(\gR') \otimes H_{n'', d'',r}\left(\gE_\infty^+(\bigoplus_\alpha S_\bQ^{n_\alpha, d_\alpha, d_\alpha}x_\alpha)\right).\]
The tridegrees of the $x_\alpha$ ensure that the right-hand factor can be non-trivial only if $r \geq 2$, $d'' \geq r$, and $d''-1 \geq \tfrac{4}{3} n''$. Now
\begin{align*}
	d'-1 = (d-1)-(d''-1) &< \tfrac{4}{3}(n-1) - (d''-1) = \tfrac{4}{3}(n'-1+n'') - (d''-1)
\end{align*}
so
	\[\tfrac{d'-1}{n'-1} < \tfrac{4}{3}+ \tfrac{1}{n'-1}(\tfrac{4}{3} n'' - (d''-1)) \leq \tfrac{4}{3},\]
and $d' < d$ as $d'' \geq r \geq 2$. Under these conditions the left-hand factor vanishes, as we supposed that $d$ was minimal; this is a contradiction.
\end{proof}

\subsection{Algebraically closed fields and conjectures of Mirzaii and Yagunov}\label{sec:MirzYagConj}

If $\bF$ is an algebraically closed field then $\bF^\times$ is a divisible group. This means that $\bF^\times \otimes_\bZ \bZ/p$ vanishes for all primes $p$, which has implications for the $E_\infty$-homology calculations in Section \ref{sec:EinfHomologyCalc} when working with $\bZ/p$-coefficients. The group $\bF^\times$ being divisible implies that Milnor $K$-theory is divisible, and so by the Nesterenko--Suslin theorem gives
	\[H_{n}(\mr{GL}_n(\bF), \mr{GL}_{n-1}(\bF);\bZ/p)=0\]
for all $n>0$. We will show in Corollary \ref{cor:alg-closed-2} below that the same vanishing holds in the next homological degree up, as long as $n>3$. We will deduce this from the following much more general homological stability range, which implies Theorem \ref{thm:stab-special-assumptions} (\ref{enum:p-div}):

\begin{theorem}\label{thm:alg-closed-1}
	Let $p$ be a prime number and $\R$ be a connected semi-local ring with infinite residue fields such that $\R^\times \otimes_\bZ \bZ/p = 0$. Then
		\[H_{d}(\mr{GL}_n(\R), \mr{GL}_{n-1}(\R);\bZ/p)=0\]
	in degrees $d < \tfrac{3}{2}n$. 
	
	If in addition $A$ is a field and $\mathfrak{p}(\R) \otimes_\bZ \bZ/p = 0$, then this holds in degrees $d < \tfrac{5}{3}n$.
\end{theorem}

\begin{proof}
We proceed as in the proof of Theorem \ref{thm:NS}, considering the map $\epsilon \colon \gR_{\bZ} \to \gN$ and its base change to $\bZ/p$. 	By applying the universal coefficient theorem to Figure \ref{fig:eInfRelHomology}, and using that $A^\times \otimes_\bZ \bZ/p=0$, the $E_\infty$-homology of the pair $(\gN, \gR_{\bZ/p})$ satisfies $H_{n,d}^{E_\infty}(\gN, \gR_{\bZ/p})=0$ for $d-1 < \tfrac{3}{2} n$.	As in the proof of Theorem \ref{thm:NS} we construct a relative CW-$E_\infty$-algebra $\gR_{\bZ/p} \hookrightarrow \gC \overset{\sim} \to \gN$ only having cells in the necessary degrees and skeletal filtration $\mathrm{sk}(\gC)$. The filtered object $\overline{\mathrm{sk}(\gC)}/\sigma$ has a spectral sequence
	\[E^1_{n,p,q} = H_{n,p+q,p}(\overline{\gR}_{\bZ/p}/\sigma[0] \otimes \gE_\infty^+(\bigoplus_\alpha S^{n_\alpha, d_\alpha, d_\alpha}x_\alpha)) \Longrightarrow H_{n,p+q}(\overline{\gN}/\sigma),\]
and $H_{*,*}(\overline{\gN}/\sigma) = \bZ/p[0,0]$. The tridegrees of the $x_\alpha$ satisfy $d_\alpha-1 \geq \tfrac{3}{2} n_\alpha$. 
	
Suppose for a contradiction that $H_{n, d}(\overline{\gR}_{\bZ/p}/\sigma)$ is nonzero, with $n>0$ and $\tfrac{d}{n} < \tfrac{3}{2}$; without loss of generality we may suppose that $d$ is minimal with this property. As $n>0$ this cannot survive the spectral sequence, so must be the target of a $d^r$-differential starting at 
	\[E^1_{n,r, d-r+1} = \bigoplus_{\substack{n'+n'' = n\\ d'+d''=d+1}}H_{n',d'}(\overline{\gR}_{\bZ/p}/\sigma) \otimes H_{n'', d'',r}(\gE_\infty^+(\bigoplus_\alpha S^{n_\alpha, d_\alpha, d_\alpha}x_\alpha)).\]
The tridegrees of the $x_\alpha$ ensure that the right-hand factor is non-trivial only if $r \geq 2$, $d'' \geq r$, and $d''-1 \geq \tfrac{3}{2} n''$. Now
	\[d' = d+1-d'' < \tfrac{3}{2}n + 1 - d'' = \tfrac{3}{2}n' + (\tfrac{3}{2} n'' + 1 - d'') \leq \tfrac{3}{2} n'\]
and $d' < d$ as $d'' \geq r \geq 2$. Under these conditions the left-hand factor vanishes, as we supposed $d$ was minimal; this is a contradiction.

Suppose now that $A$ is a field and that $\mathfrak{p}(\R) \otimes_\bZ \bZ/p = 0$. By Corollary \ref{cor:EInfHomNrelR} we have $H_{n,d}^{E_\infty}(\gN, \gR_{\bZ/p})=0$ for $d < 2n$, and we also have $H_{1,2}^{E_\infty}(\gN, \gR_{\bZ/p})= A^\times \otimes_\bZ \bZ/p=0$. In rank $n=2$ we have the following:

\vspace{1ex}
	
\noindent \textbf{Claim.} We have $H_{2,d}^{E_\infty}(\gN, \gR_{\bZ/p})=0$ for $d \leq 4$.
	
\begin{proof}[Proof of Claim]
We already know that $H_{2,d}^{E_\infty}(\gN, \gR_{\bZ})=0$ for $d \leq 3$. As in Remark \ref{rem:EInfHomOfN} we have
	\[H_{2,d}^{E_\infty}(\gN) = H_{2,d}^{E_\infty}(\gN, \gE_\infty(S^{1,0})) \overset{\sim}\lla H_{2,d}(\gN, \gE_\infty(S^{1,0})) \overset{\sim}\lra \tilde{H}_{d-1}(\fS_2 ; \bZ).\]
The long exact sequence for the pair is then
	\[H_{2,4}^{E_\infty}(\gR_\bZ) \lra \bZ/2 \lra H_{2,4}^{E_\infty}(\gN, \gR_\bZ) \lra \mathfrak{p}(\R) \lra 0\]
and we will show that the first map is surjective.  Under the assumption $\mathfrak{p}(\R) \otimes_\bZ \bZ/p = 0$, the claim then follows from the resulting isomorphism $H_{2,4}^{E_\infty}(\gN, \gR_\bZ) \cong \mathfrak{p}(\R)$ together with the Universal Coefficient Theorem.                
		
For this, as in Section \ref{sec:EinfHomologyCalc} let $\gA = \gE_\infty(S^{1,0} \otimes \bZ[B\R^\times])$, and also write $\gE = \gE_\infty(S^{1,0})$, and then consider the commutative square
\begin{equation*}
	\begin{tikzcd}
	\gA \dar \rar& \gR_\bZ \dar\\
	\gE \rar& \gN,
	\end{tikzcd}
\end{equation*}
whose left-hand vertical map is induced by $B\R^\times \to \{*\}$ and so is a split epimorphism. As $\gA$ only has $E_\infty$-cells in rank 1, the map $H_{2,*}^{E_\infty}(\gR_\bZ) \to H_{2,*}^{E_\infty}(\gR_\bZ, \gA)$ is an isomorphism; similarly with $\gN$ and $\gE$. Consider the diagram
\begin{equation*}
	\begin{tikzcd}
	H_{2,4}^{E_\infty}(\gR_\bZ, \gA) \dar & H_{2,4}(\gR_\bZ, \gA) \lar[swap]{\sim} \dar \rar& H_{2,3}(\gA) \dar \rar &  H_{2,3}(\gR_\bZ)\\
	H_{2,4}^{E_\infty}(\gN, \gE) & H_{2,4}(\gN, \gE) \lar[swap]{\sim} \rar{\sim}& H_{2,3}(\gE)
	\end{tikzcd}
\end{equation*}
where the left-hand isomorphisms are obtained as in Remark \ref{rem:EInfHomOfN}, using Proposition $E_k$.11.9.
		
All groups in the bottom row are isomorphic to $H_3(B\fS_2;\bZ) = \bZ/2$, and we need to show that the leftmost vertical map is surjective; it is therefore enough to show that the middle vertical map is surjective. 
		
Consider the splitting $\gE \to \gA$, which gives a class $\alpha \in H_{2,3}(\gA)$ mapping to a generator of $H_{2,3}(\gE)=\bZ/2$.  The class $\alpha$ is, by construction, given by the homomorphism
	\begin{equation*}
	t \longmapsto \begin{bsmallmatrix} 0 & 1 \\
	1 & 0 \end{bsmallmatrix} \colon C_2 \lra \mr{GM}_2(\R)
	\end{equation*}
evaluated on the generator of $H_3(BC_2;\bZ) = \bZ/2$. As in the proof of Lemma \ref{lem:SNrelations} this homomorphism induces the same as $t \mapsto \begin{bsmallmatrix} 1 & 0 \\ 0 & -1 \end{bsmallmatrix}$ on homology when considered as a homomorphism to $\mr{GL}_2(\R)$. Thus there is a class $\alpha' \in H_{3}(\mr{GM}_1(\R);\bZ)$ such that $\alpha - \sigma \cdot \alpha'=0 \in \mr{Ker}(H_{2,3}(\gA) \to H_{2,3}(\gR_\bZ))$, so $\alpha - \sigma \cdot \alpha'$ lifts to $H_{2,4}(\gR_\bZ, \gA)$. But under the map $H_{2,3}(\gA) \to H_{2,3}(\gE)$ the class $\alpha - \sigma \cdot \alpha'$ is sent to a generator, as $\alpha'$ is sent to to 0.
\end{proof}

In particular we have $H_{n,d}^{E_\infty}(\gN, \gR_{\bZ/p})=0$ for $d-1 < \tfrac{5}{3} n$. The same spectral sequence argument as above, with $\tfrac{3}{2}$ replaced by $\tfrac{5}{3}$, gives the claimed vanishing range.
\end{proof}

\begin{remark}\label{remark:entry-two-four}
The proof of the isomorphism $H_{2,4}^{E_\infty}(\gN, \gR_\bZ) \cong \mathfrak{p}(\R)$ given above applies for any connected semi-local ring with infinite residue fields, and shows that the entry $(n,d) = (2,4)$ of Figure~\ref{fig:eInfRelHomology} is as depicted.
\end{remark}

The assumption that $\R^\times \otimes_\bZ \bZ/p=0$ is in particular satisfied when $\R$ is an algebraically closed field $\bF$. In that case it is a theorem of Dupont and Sah \cite[Theorem 5.1]{DupontSah} that the pre-Bloch group $\mathfrak{p}(\bF)$ is divisible, so we also have $\mathfrak{p}(\bF) \otimes_\bZ \bZ/p=0$. This implies Theorem \ref{thm:stab-special-assumptions} (\ref{enum:alg-closed}). Combining a number of deep results in algebraic $K$-theory shows that $\mathfrak{p}(\bF) \otimes_\bZ \bZ/p=0$ under much weaker assumptions than being algebraically closed:

\begin{corollary}\label{cor:roots-of-unity}
If $\bF$ is an infinite field satisfying (i) $\bF^\times \otimes \bZ/p = 0$, and (ii) the polynomial $x^p-1$ splits into linear factors, then
	\[H_{d}(\mr{GL}_n(\bF), \mr{GL}_{n-1}(\bF);\bZ/p)=0\]
in degrees $d < \tfrac{5}{3}n$.
\end{corollary}

\begin{proof}	
We verify that $\mathfrak{p}(\bF) \otimes_\bZ \bZ/p = 0$ under the above assumptions. There are two cases: either (a) $\mr{char}(\bF) = p$, or (b) $\bF$ contains $p$th roots of unity. From the tautological exact sequence
	\[0 \lra B(\bF) \lra \mathfrak{p}(\bF) \lra \tilde{\wedge}^2 \bF^\times \lra K^M_2(\bF) \lra 0,\]
and the exact sequence of  \cite[Theorem 5.2]{SuslinK3}
	\[0 \lra \mr{Tor}_1^{\bZ}(\bF^\times,\bF^\times)^\sim \lra K_3(\bF)^\mr{ind} \lra B(\bF) \lra 0,\]
it follows that it suffices to prove that both $K_3(\bF)^\mr{ind} \otimes \bZ/p$ and the $p$-torsion in $K^M_2(\bF)$ vanish.
	
In case (a), $K_3(\bF)^\mr{ind} \otimes \bZ/p = 0$ by \cite[Theorem 8.5]{SuslinMerkujev} and the $p$-torsion in $K^M_2(\bF)$ vanishes by \cite{Izhboldin}. In case (b), we will have $\mr{char}(\bF) \neq p$. Then on \cite[p.~257]{Kolster} we find there is an exact sequence
	\[0 \lra K_3(\bF)^\mr{ind} \otimes \bZ/p \lra \mu_p \otimes \bF^\times \lra \Ker\left[p \colon K^M_2(\bF) \to K^M_2(\bF)\right] \lra 0\]
under the assumption that $\mr{char}(\bF) \neq p$ and $\mu_p \subset \bF$. Since $\mu_p \cong \bZ/p$, the middle term vanishes and hence so do the two outer terms.
\end{proof}

\begin{remark}\label{rem:k-theory-r} If $\mr{char}(\bF) \neq p$ and $\bF$ does not contain $p$th roots of unity, write $\bE \coloneqq \bF(\zeta_p)$ and $G \coloneqq \mr{Gal}(\bE/\bF)$. We then have \cite[Proposition 1.2]{Kolster} \cite{Merkurev}
	\[0 \lra K_3(\bF)^\mr{ind} \otimes \bZ/p \lra (\mu_p \otimes \bE^\times)^G \lra \Ker\left[p \colon K^M_2(\bF) \to K^M_2(\bF)\right] \lra 0.\]
So when $\mr{char}(\bF) \neq p$ we may replace hypothesis (ii) in Corollary \ref{cor:roots-of-unity} by the assumption that $\bE^\times \otimes \bZ/p = 0$.
\end{remark}

Since $n+1<\tfrac{5}{3}n$ when $n>1$, we deduce:

\begin{corollary}\label{cor:alg-closed-2}
If $\bF$ is an algebraically closed field then
\[H_{n+1}(\mr{GL}_n(\bF), \mr{GL}_{n-1}(\bF);\bZ/p)=0\]
for all $n>1$ and all primes $p$. 
\end{corollary}

In \cite[Section 4]{mirzaiiclosed}, Mirzaii proved this statement for $2< n \leq 4$, and related the vanishing of these groups to divisibility of certain higher pre-Bloch groups $\mathfrak{p}_n(\bF)$ suggested by Loday \cite[Section 4.4]{lodaycomp}. He shows \cite[Theorem 3.4]{mirzaiiclosed} that the conclusion of Corollary \ref{cor:alg-closed-2} implies that these higher pre-Bloch groups satisfy
	\[\mathfrak{p}_n(\bF) \otimes \bZ/p = \begin{cases}
	\bZ/p & \text{$n$ odd,}\\
	0 & \text{$n$ even,}
	\end{cases}\]
which resolves \cite[Conjecture 3.5]{mirzaiiclosed}.

A different notion of higher pre-Bloch groups has been defined by Yagunov \cite{yagunov}, and are called $\wp^n(\bF)$ and $\wp^n(\bF)_{cl}$. By \cite[Remark 2.3]{mirzaiiclosed}, there are identifications \[\wp^n(\bF)_{cl} = \begin{cases}
	\Ker\left[\text{an epimorphism } \mathfrak{p}_n(\bF) \to \bZ\right] & \text{$n$ odd,}\\
	\mathfrak{p}_n(\bF) & \text{$n$ even,}
	\end{cases}\] 
whence it follows from the above that $\wp^n(\bF)_{cl} \otimes \bZ/p=0$ for all primes $p$ so that $\wp^n(\bF)_{cl}$ is a divisible group. The group $\wp^n(\bF)$ is defined in relation to a coefficient ring $\bk$, but if $\bF$ is algebraically closed and $2 \in \bk^\times$ then a certain map $\wp^n(\bF)_{cl} \otimes \bk \to \wp^n(\bF)$ is surjective; thus $\wp^n(\bF)$ is a divisible $\bk$-module (i.e.~for all $x \in \wp^n(\bF)$ and integer $n \geq 1$ there exists $y \in \wp^n(\bF)$ such that $ny=x$). This resolves Yagunov's Conjecture 0.2 (at least for coefficient rings $\bk$ in which 2 is invertible, but see \cite[Remark 4.1]{yagunov}).

Finally, recall from Example \ref{exam:vanishing-k2} that the rational $K_2$ of the algebraic closures $\bar{\bQ}$ and $\bar{\bF}_p$ vanishes, so both Theorem \ref{thm:K2vanish} and \ref{thm:alg-closed-1} apply to these fields, from which we obtain the following \emph{integral} homological stability statement. 

\begin{corollary}
If $\bF$ is $\bar{\bQ}$ or $\bar{\bF}_p$ then
	\[H_{d}(\mr{GL}_n(\bF), \mr{GL}_{n-1}(\bF);\bZ)=0\]
in degrees $d < \tfrac{4n-1}{3}$.
\end{corollary}

\begin{proof}
By Theorem \ref{thm:K2vanish} these groups are torsion for $d < \tfrac{4n-1}{3}$. By the portion 
	\[\cdots \lra H_{n,d+1}(\gR_{\bZ/p}) \overset{\beta}\lra H_{n,d}(\gR_{\bZ}) \overset{p}\lra H_{n,d}(\gR_{\bZ}) \lra \cdots\]
of the Bockstein sequence and Theorem \ref{thm:alg-closed-1}, the $p$-torsion subgroup must vanish as long as $d+1 < \tfrac{5n}{3}$
\end{proof}

\begin{remark}The homology groups of $\mr{GL}_n(\bar{\bF}_p)$ with $\bZ/\ell$-coefficients for $\gcd(\ell,p) = 1$ were completely computed by Quillen \cite[Theorem 3]{quillenfinite}.\end{remark}

\subsection{The ``weak'' injectivity conjecture of Suslin and Mirzaii}
\label{sec:weak-inject-conj}

For a connected semi-local ring with infinite residue fields, and a coefficient ring $\bk$,  the inclusion of diagonal matrices into all matrices leads to a homomorphism
	\[H_1(\mr{GL}_{1}(\R);\bk)^{\otimes n} \to H_{n}(\mr{GL}_{n}(\R);\bk) \to H_{n}(\mr{GL}_{n}(\R),\mr{GL}_{n-1}(\R);\bk) \cong K_n^M(\R) \otimes_{\bZ} \bk,\]
which is induced from the canonical quotient homomorphism $(\R^\times)^{\otimes n} \to K^M_n(\R)$.  It is in particular surjective, implying by the long exact sequence that the surjection $H_{n-1}(\mr{GL}_{n-1}(\R)) \to H_{n-1}(\mr{GL}_n(\R))$ is an isomorphism.  In one degree above, we obtain an exact sequence
\begin{equation}\label{eq:32}
	H_n(\mr{GL}_{n-1}(\R);\bk) \lra H_n(\mr{GL}_n(\R);\bk) \lra K_n^M(\R)\otimes_\bZ \bk \lra 0.
\end{equation}
If $(n-1)!$ is invertible in $\bk$, this sequence may be explicitly split using the composition $K_n^M(\R) \to K_n(\R) = \pi_n(B\mr{GL}(\R)^+) \to H_n(B\mr{GL}(\R) ^+) \cong H_n(B\mr{GL}_n(\R))$ of the map from Milnor $K$-theory to $K$-theory, followed by the Hurewicz map.  Indeed, the composition $K^M_n(\R) \to H_n(\mr{GL}_n(\R)) \to H_n(\mr{GL}_n(\R),\mr{GL}_{n-1}(\R)) \cong K^M_n(\R)$ is multiplication by $(-1)^{n-1}(n-1)!$, see \cite{SuslinCharClass}, \cite[Theorem 4.1a]{SN}.

In this section we complete a programme of Mirzaii \cite{MirzaiiInj} to prove that when $\R = \bF$ is an infinite field and $\bk$ is a field in which $(n-1)!$ is invertible, then the first map in~\eqref{eq:32} is injective, leading to the direct sum decomposition
\begin{equation}\label{eq:36}
  H_n(\mr{GL}_n(\bF);\bk) \cong (\bk \otimes_\bZ K^M_n(\bF)) \oplus H_n(\mr{GL}_{n-1}(\bF);\bk).
\end{equation}
For $\bk = \bQ$, Suslin has asked more generally whether the stabilisation maps
\[H_i(\mr{GL}_{n-1}(\bF);\bQ) \lra H_i(\mr{GL}_n(\bF);\bQ)\]
might always be injective (see \cite[Problem 4.13]{Sah}, \cite[Remark 7.7]{borelyang}, \cite[Conjecture 2]{DeJeu}, \cite[Conjecture 1]{MirzaiiInj}).  As explained above, this is the case for $i < n$ by \cite[Theorem 3.4(c)]{SuslinCharClass}.

\begin{theorem}\label{thm:WeakInj}
If $\bF$ is an infinite field and $\bk$ is a field in which $(n-1)!$ is invertible then the stabilisation map
\[H_n(\mr{GL}_{n-1}(\bF);\bk) \lra H_n(\mr{GL}_n(\bF);\bk)\]
is injective.
\end{theorem}

This is straightforward for $n=1,2$, was proven for $n=3$ by Sah (a consequence of \cite[Remark 3.19]{Sah}) and by Elbaz-Vincent (a consequence of the proof of  \cite[Theorem 1.22]{E-Z}), and for $n=4$ by Mirzaii \cite[Theorem 3]{MirzaiiInj}. The cases $n > 4$ are new.

\begin{remark}
For infinite fields whose rational Milnor $K$-theory is Koszul, we may also deduce this for $\bk=\bQ$ from Corollary \ref{cor:NSPlusOneKoszul}. It is equivalent to show that the connecting homomorphism
\[\partial_n \colon H_{n+1}(\mr{GL}_n(\bF),\mr{GL}_{n-1}(\bF);\bQ) \lra H_n(\mr{GL}_{n-1}(\bF);\bQ)\]
is zero. In our notation these maps are $\partial_n \colon H_{n, n+1}(\overline{\gR}_\bQ/\sigma) \to H_{n-1,n}(\overline{\gR}_\bQ)$, which assemble into to a map 
\[\partial \colon \bigoplus_{n \geq 0} H_{n,n+1}(\overline{\gR}_\bQ/\sigma) \lra \bigoplus_{n \geq 0} H_{n-1,n}(\overline{\gR}_\bQ)\]
of $\tilde{\wedge}^* \bF^\times_\bQ$-modules. The maps $\partial_n$ are known to be zero for $n \leq 3$ by the aforementioned work of Sah \cite{Sah} and Elbaz-Vincent \cite{E-Z}, but under the Koszulness hypotheses the domain is generated as a $\tilde{\wedge}^* \bF^\times_\bQ$-module in degrees $n \leq 3$ by Corollary \ref{cor:NSPlusOneKoszul}, so the map $\partial$ is zero.
\end{remark}

Let us recall Conjecture 2 of \cite{MirzaiiInj}, which is what we shall prove.  Consider the diagram of groups
\begin{equation}\label{eq:35}
  \begin{tikzcd}
    \bF^\times \times \bF^\times \times \mr{GL}_{k-1}(\bF) \rar[shift left=.5ex] \rar[shift left=-.5ex] &
    \bF^\times \times \mr{GL}_{k-1}(\bF) \rar["a"] & \mr{GL}_k(\bF),
  \end{tikzcd}
\end{equation}
where $a$ denotes the standard block sum map, and the two parallel arrows are given by $(x,y,X) \mapsto (x,a(y,X))$ and $(x,y,X) \mapsto (y,a(x,X))$, respectively.  Now \cite[Conjecture 2]{MirzaiiInj} can be phrased as this diagram becoming a coequalizer diagram upon applying $H_k(-;\bk)$.  In other words, the sequence
\tikzcdset{scale cd/.style={every label/.append style={scale=#1},
    cells={nodes={scale=#1}}}}
\begin{equation}\label{eq:34}
\begin{tikzcd}[scale cd=0.9] 
H_k((\bF^\times)^2 \times \mr{GL}_{k-2}(\bF);\bk) \rar{(\bF^\times \times a)_* \circ (\mr{id}-\tau_*)}  \ar[draw=none]{d}[name=X, anchor=center]{}&[40pt]   H_k(\bF^\times \times \mr{GL}_{k-1}(\bF);\bk)
\ar[rounded corners,
to path={ -- ([xshift=2ex]\tikztostart.east)
	|- (X.center) \tikztonodes
	-| ([xshift=-2ex]\tikztotarget.west)
	-- (\tikztotarget)}]{dl}[pos=0.8, swap]{a_*} \\
  H_k(\mr{GL}_k(\bF);\bk) \rar &  0 &   
\end{tikzcd}
\end{equation}
should be exact, where $\tau \colon \bF^\times \times \bF^\times \times \mr{GL}_{k-1}(\bF) \to \bF^\times \times \bF^\times \times \mr{GL}_{k-1}(\bF)$ is the homomorphism which swaps the first two factors.

\sloppy It was shown in \cite[Proposition 4]{MirzaiiInj} that exactness of this sequence together with injectivity of the maps $H_{k-1}(\mr{GL}_{k-2}(\bF);\bk) \to H_{k-1}(\mr{GL}_{k-1}(\bF);\bk)$ and $H_{k-2}(\mr{GL}_{k-3}(\bF);\bk) \to H_{k-2}(\mr{GL}_{k-2}(\bF);\bk)$ implies injectivity of 
\[H_{k}(\mr{GL}_{k-1}(\bF);\bk) \lra H_{k}(\mr{GL}_{k}(\bF);\bk),\]  We shall show that~\eqref{eq:34} is exact for all $k \geq 3$ when $\max(6,(k-1)!)$ is invertible in $\bk$, from which Theorem~\ref{thm:WeakInj} then follows by induction (except when $n=3 = \mr{char}(\bk)$).

In the rest of this section we shall use both letters $k$ and $n$ for the rank grading (first index), following the convention that $k$ is used for the rank in which we are proving exactness of~\eqref{eq:34}, and $n$ as a generic letter in spectral sequences etc.

\begin{remark}\label{rem:Mirzaii-remarks}
  Let us comment on which parts of \cite{MirzaiiInj} we use, referring also to \cite{Mirzaii15}\footnote{We thank Behrooz Mirzaii for drawing our attention to this reference.} for some corrections and improvements. In particular, we recommend replacing both \cite[Definition 2]{MirzaiiInj} and \cite[Proposition 3]{MirzaiiInj} by \cite[Corollary 2.2]{Mirzaii15}, which states and proves an explicit formula for the composition
  \[\begin{tikzcd} K^M_n(\bF) \rar & K_n(\bF) \rar{\text{Hurewicz}} &[15pt] H_n(B\mr{GL}(\bF)^+) & \lar[swap]{\cong} H_n(B\mr{GL}_n(\bF))\\[-15pt]
    \{a_1, \dots, a_n\} \arrow[|->]{rrr} & & & {[a_1,\dots, a_n]}\end{tikzcd}\]
  The element $[a_1, \dots, a_n] \in H_n(\mr{GL}_n(\bF))$ is defined using an explicit group homomorphism $\bZ^n \to \mr{GL}_n(\bF)$ given by the diagonal matrices $A_{i,n}$ stated in op.cit.,\ and the fundamental class of $(S^1)^n \simeq B\bZ^n$.

  In addition to the general strategy, the only part of \cite{MirzaiiInj} which we use here is its Proposition 4.  Using \cite[Corollary 2.2]{Mirzaii15} as suggested, the proof of \cite[Proposition 4]{MirzaiiInj} then does not use any other parts of \cite{MirzaiiInj}, in fact not even the spectral sequences mentioned in the same section (the symbols $\smash{{\mathfrak{d}'}^1_{1,n}}$ and $\smash{\beta_2^{(n)}}$ denote the same homomorphism).  Let us explain the strategy, since it is quite simple.  Assuming that applying  $H_k(-;\bk)$ to \eqref{eq:35} gives a presentation of $H_k(\mr{GL}_k(\bF);\bk)$ as a quotient of
  \begin{equation*}
    H_k(\bF^\times \times \mr{GL}_{k-1}(\bF);\bk) \cong \bigoplus_{i + j =k} H_i(\bF^\times;\bk) \otimes_\bk H_j(\mr{GL}_{k-1}(\bF);\bk),
  \end{equation*}
  Mirzaii uses the induction hypothesis and the resulting direct sum decompositions~\eqref{eq:36} of $H_{k-1}(\mr{GL}_{k-1}(\bF);\bk)$ and $H_{k-2}(\mr{GL}_{k-2}(\bF);\bk)$ to define a retraction $\varphi$ of $H_k(\bF^\times \times \mr{GL}_{k-1}(\bF);\bk)$ onto the K\"unneth summand with $(i,j) = (0,k)$, such that $\varphi$ coequalizes the two homomorphisms from $H_k(\bF^\times \times \bF^\times \times \mr{GL}_{k-2}(\bF);\bk)$ from~\eqref{eq:35}.  A retraction with this property supplies a left inverse to the restriction of $a_*$ to the $(0,k)$-summand, so this restriction is injective.
\end{remark}

As the left-hand map in~\eqref{eq:34} is anti-invariant with respect to swapping the two $\bF^\times$-factors, we may replace its domain by the anti-coinvariants
\[[H_k((\bF^\times)^2 \times \mr{GL}_{k-2}(\bF);\bk) \otimes \bk^{-}]_{\fS_2}.\]
Note that the field $\bk$ is always required to contain $\tfrac{1}{2}$, so we may write this as
\[H_k(\mr{GM}_2(\bF) \times \mr{GL}_{k-2}(\bF);\bk^{-}),\]
where $\bk^-$ is the coefficient module given by $\mr{GM}_2(\bF) \times \mr{GL}_{k-2}(\bF) \to \mr{GM}_2(\bF) \to \fS_2$ and the sign action. 

\begin{lemma}
Modifying Mirzaii's complex in this way, it is isomorphic up to sign to the portion
\[E^1_{k,2,k} \overset{d^1}\lra E^1_{k,1,k} \overset{d^1}\lra E^1_{k,0,k} \lra 0\]
of the $E^1$-page of the spectral sequence of the filtration of the bar construction
\[B(\bk, \overline{\gE_\infty(S^{1,0} \otimes \bk[B\bF^\times])}, \overline{\gR}_\bk)\]
given by putting $\bk$ and $\overline{\gR}_\bk$ in filtration 0 and putting $S^{1,0} \otimes \bk[B\bF^\times]$ in filtration 1, i.e.\ the filtered object $B(0_* \bk, \overline{\gE_\infty(1_* S^{1,0} \otimes \bk[B\bF^\times])}, 0_*\overline{\gR}_\bk)$.
\end{lemma}

\begin{proof}
This filtered object has associated graded
\[B(\bk, \overline{\gE_\infty(S^{1,0,1} \otimes \bk[B\bF^\times])}, \bk) \otimes 0_*\overline{\gR}_\bk \simeq \gE^+_\infty(S^{1,1,1} \otimes \bk[B\bF^\times]) \otimes 0_*\overline{\gR}_\bk,\]
where the equivalence is by Theorem $E_k$.13.8. In second grading $q \geq 0$ (i.e.~third index) the object $\gE^+_\infty(S^{1,1,1} \otimes \bk[B\bF^\times])$ is equivalent to
\[(S^{1,1} \otimes \bk[B\bF^\times] )^{\otimes q} \otimes_{\fS_q} E\fS_q \simeq S^{q,q} \otimes \left(\bk^- \otimes_{\mr{GM}_q(\bF)} E\mr{GM}_q(\bF) \right),\]
which gives
\begin{equation}
\begin{aligned}
  E^1_{n,p,q} &= H_{n,p+q,p}(\gE^+_\infty(S^{1,1,1} \otimes \bk[B\bF^\times]) \otimes 0_*\overline{\gR}_\bk) \\
  &= H_{q}(\mr{GM}_p(\bF) \times \mr{GL}_{n-p}(\bF) ; \bk^-)
\end{aligned}\label{eq:39}
\end{equation}
with differentials $d^r \colon E^r_{n,p,q} \to E^r_{n,p-r,q+r-1}$. This identifies the claimed groups with the terms in (our modification of) Mirzaii's sequence. It remains to show that the $d^1$-differential of this spectral sequence is identified with Mirzaii's maps.

The differential $d^1 \colon E^1_{n,1,p} \to E^1_{n,0,p}$ may be analysed as follows. Let us write $B_\bullet \coloneqq B_\bullet(0_* \bk, \overline{\gE_\infty(1_* S^{1,0} \otimes \bk[B\bF^\times])}, 0_*\overline{\gR}_\bk)$, $B = |B_\bullet|$, and $F_q B_\bullet$ for the $q$th stage of the filtration. Let us describe the equivalence $\frac{F_1 B}{F_0 B} \simeq S^{1,1} \otimes \bk[B\bF^\times] \otimes \overline{\gR}_\bk$.  It is induced by the evident map
\begin{equation}
  S^{1,0} \otimes \bk[B\bF^\times] \otimes \overline{\gR}_\bk \lra F_1 B_1\label{eq:37}
\end{equation}
in the following way.  The map~\eqref{eq:37} has the property that its composition with either face map $d_0, d_1 \colon F_1 B_1 \to F_1 B_0$ lands in $F_0 B_0 \subset F_1 B_0$.  Such a map can be regarded as a map of pairs of simplicial objects
\begin{equation*}
  (\Delta^1 _\bullet, \partial \Delta^1_\bullet) \otimes S^{1,0} \otimes \bk[B\bF^\times] \otimes \overline{\gR}_\bk \lra (F_1 B_\bullet,F_0 B_\bullet)
\end{equation*}
which upon realising and passing from pairs to their quotients gives the map $S^{1,1} \otimes \bk[B\bF^\times] \otimes \overline{\gR}_\bk \lra \frac{F_1 B}{F_0 B}$.  In this situation the connecting map
\begin{equation*}
  \frac{F_1 B}{F_0 B} \simeq \mathrm{Cofib}\big(F_0 B \to F_1 B) \lra S^{0,1} \otimes F_0 B
\end{equation*}
is homotopic to the composition of~\eqref{eq:37} with $d_0 - d_1 \colon B_1 \to B_0$. The map $d_0$ is zero (as the augmentation $\overline{\gE_\infty(S^{1,0} \otimes \bk[B\bF^\times])} \to \bk$ is trivial on $S^{1,0} \otimes \bk[B\bF^\times]$), and the map $d_1$ is that induced by $\bF^\times \times \mr{GL}_{n-1}(\bF) \to \mr{GL}_n(\bF)$, i.e.\ the map $a_*$. Under the identification~\eqref{eq:39} the differential $d^1 \colon E^1_{n,1,q} \to E^1_{n,0,q}$ therefore becomes $-a_*$, which for $n = q = k$ is (minus) the map in \eqref{eq:34}.

The formula just proved may be stated more concisely as
\begin{equation*}
  d^1((S^{1,1,1} \otimes x) \otimes r) = - a_*(x \otimes r),
\end{equation*}
for $x \in H_i(\bF^\times;\bk)$ and $r \in H_{n,q-i}(\overline{\gR}_\bk)$.  By multiplicativity of the spectral sequence we deduce
\begin{multline*}
  d^1((S^{1,1,1} \otimes x)  (S^{1,1,1} \otimes y)  \otimes r) =\\
  (-1)^{i} (S^{1,1,1} \otimes y) \otimes (- a_*(x \otimes r)) + (-1)^{i+1} (S^{1,1,1} \otimes x) \otimes (- a_*(y \otimes r))
\end{multline*}
for $x \in H_i(\bF^\times;\bk)$ and $y \in H_j(\bF^\times;\bk)$ and $r \in H_{n,q-i-j}(\overline{\gR}_\bk)$.  The isomorphism~\eqref{eq:39} now identifies the differential $d^1 \colon E^1_{n,2,q} \to E^1_{n,1,q}$ with $(\bF^\times \times a)_* \circ (\mr{id}-\tau_*)$, up to (degree-dependent) sign.  Setting $(n,q) = (k,k)$ this is first map in~\eqref{eq:34}.
\end{proof}

To prove exactness of~\eqref{eq:34} we must show that $E^2_{k,0,k}$ and $E^2_{k,1,k}$ vanish.  Under the stronger assumption $\mr{char}(\bk) = 0$, the argument goes as follows.  In this case, the natural maps
\begin{equation}\label{eq:Approx}
  \mr{S}^*_\bk(S^{1,1,1} \otimes H_*(\bF^\times;\bk)) \lra H_{*,*,*}(E^+_\infty(S^{1,1,1} \otimes \bk[B\bF^\times])),
\end{equation}
and
\begin{equation}\label{eq:ChainMap}
\begin{tikzcd}
\mr{S}^*_\bk(S^{1,1,1} \otimes H_*(\bF^\times;\bk)) \otimes H_{*,0,*}(\overline{\gR}_\bk) \dar \\  H_{*,*,*}(E^+_\infty(S^{1,1,1} \otimes \bk[B\bF^\times]))\otimes H_{*,0,*}(\overline{\gR}_\bk) \cong E^1_{*,*,*}
\end{tikzcd}
\end{equation}
are isomorphisms in all degrees.  Furthermore, as in the proof above, the Leibniz rule determines the entire $d^1$-differential since the $E^1$-page is multiplicatively generated by $E^1_{*,p,*}$ with $p \leq 1$.  As a trigraded differential algebra, the domain of~\eqref{eq:ChainMap} may be regarded as the Koszul resolution
\begin{equation}\label{eq:KoszulResolution}
  (\mathrm{S}^*_\bk(S^{1,1,1} \otimes H_*(\bF^\times;\bk)) \otimes \mr{S}^*_\bk(S^{1,0,0} \otimes H_*(\bF^\times;\bk)), d) \lra \bk[0,0,0],
\end{equation}
whose differential is determined by $d(S^{1,1,1} \otimes v) = S^{1,0,0} \otimes v$ and the Leibniz rule, tensored with $H_{*,*,0}(\overline{\gR}_\bk)$ over $\mathrm{S}^*_\bk(S^{1,0,0} \otimes H_*(\bF^\times;\bk))$.  When $\mr{char}(\bk) = 0$ the $E^1$-page therefore calculates $\mr{Tor}$ of $\bk$ and $H_{*,*}(\overline{\gR}_\bk)$ over $\mathrm{S}^*_\bk(S^{1,0} \otimes H_*(\bF^\times;\bk))$, giving an isomorphism
\begin{equation}\label{eq:TorCompare}
  \mathrm{Tor}_p^{\mathrm{S}^*_\bk(S^{1,0} \otimes H_*(\bF^\times;\bk))}(\bk, H_{*,*}(\overline{\gR}_\bk))_{n,q} \lra E^2_{n,p,q}.
\end{equation}
Under the assumption $\mr{char}(\bk) = 0$, the proof is concluded in two steps.  The first (Proposition~\ref{prop:for-clarity} below) is to show that $E^2_{k,r,k-r+1}$ and $E^2_{k,r+1,k-r+1}$ vanish for all $r \geq 2$, using Suslin's theorem (our Theorem~\ref{thm:NS}) to estimate the $\mr{Tor}$ groups.  Since no non-zero differentials can exit $E^2_{n,p,q}$ for $p \leq 1$ this shows that $E^2_{k,0,k} = E^\infty_{k,0,k}$ and $E^2_{k,1,k} = E^\infty_{k,1,k}$. This spectral sequence converges to the homology of the bar construction $B(\bk, \overline{\gE_\infty(S^{1,0} \otimes \bk[B\bF^\times])}, \overline{\gR}_\bk)$, and the second step (Proposition~\ref{prop:MirzaiiFinalStep} below) is to use our estimates on cells in a relative cell structure on $\gE_\infty(S^{1,0} \otimes \bk[B\bF^\times]) \to \overline{\gR}_\bk$ to see that this bar construction has no homology in bidegrees $(k,k)$ and $(k,k+1)$.

Under the weaker assumption that $(k-1)!$ is invertible in $\bk$ the structure of the argument is the same, but the maps~\eqref{eq:Approx} and~\eqref{eq:ChainMap} will only be isomorphisms in a range and the~\eqref{eq:KoszulResolution} will only be a resolution in a range.  Before proceeding we make these ranges more precise.

\begin{lemma}\label{lem:AlmostFree}
  The map \eqref{eq:Approx} is an isomorphism in tridegrees $(n,p,q)$ when $n!$ is invertible in $\bk$, for any $p$ and $q$.  If $k \geq 4$ and $(k-1)!$ is invertible in $\bk$, then it is also an isomorphism when $n = k$ and $p+q < \tfrac{8k}{5}$.
\end{lemma}

\begin{proof}
In rank $n$ this map is
\[\mathrm{S}^n_\bk(S^{1} \otimes H_*(\bF^\times;\bk)) \lra H_*((S^1 \otimes \bk[B\bF^\times])^{\otimes n} \hcoker \fS_n)\]
which is identified with the edge homomorphism for the homotopy orbits spectral sequence
\[E^2_{s,t} = H_s(\fS_n ; H_t((S^1 \otimes \bk[B\bF^\times])^{\otimes n})) \Longrightarrow H_{s+t}((S^1 \otimes \bk[B\bF^\times])^{\otimes n} \hcoker \fS_n).\]
We must therefore show that the columns with $s>0$ do not contribute in these ranges of degrees.  If $n!$ is invertible in $\bk$ this is simply because $\fS_n$ has no higher homology with $\bk$-module coefficients.
 
If $n = k \geq 4$ and $(k-1)!$ is invertible in $\bk$,
consider $H_*((S^1 \otimes \bk[B\bF^\times])^{\otimes k}) = (S^1 \otimes H_*(\bF^\times;\bk))^{\otimes k}$ as a graded $\fS_k$-representation, recalling that this action involves the Koszul sign rule. In degree $t=k+\ell$ it is given by
\[\bigoplus_{\substack{n_0 + n_1 + \cdots + n_\ell = k \\ n_1 + 2n_2 + \cdots + \ell n_\ell = \ell}} \left(\mathrm{Ind}^{\fS_k}_{\fS_{n_0} \times \cdots \times \fS_{n_\ell}} \bigotimes_{i =0}^{\ell} ((\bk^{-})^{\otimes i+1} \otimes H_i(\bF^\times;\bk)^{\otimes n_i})\right).\]
By Shapiro's lemma we have
\begin{align*}&H_s\left(\fS_k ; \mathrm{Ind}^{\fS_k}_{\fS_{n_0} \times \cdots \times \fS_{n_\ell}} \bigotimes_{i =0}^{\ell} ((\bk^{-})^{\otimes i+1} \otimes H_i(\bF^\times;\bk)^{\otimes n_i}) \right) \\
&\qquad= H_s\left(\fS_{n_0} \times \cdots \times \fS_{n_\ell}; \bigotimes_{i = 0}^{\ell} ((\bk^{-})^{\otimes i+1} \otimes H_i(\bF^\times;\bk)^{\otimes n_i}) \right).\end{align*}
Unless $n_i=k$ for some $i$ (the remaining $n_j$'s being zero), the group $\fS_{n_0} \times \cdots \times \fS_{n_\ell}$ has order dividing a power of $(k-1)!$ and so this homology vanishes for $s>0$. The terms with $n_i=k$ and $i>0$ contribute to degree at least $2k > \frac{3k}5$. The remaining term is given by $n_0=k$, which contributes $H_\ell(\fS_k;\bk^{-})$ to degree $k + \ell$.

If $k$ is not a prime then $|\fS_k| = k!$ divides a power of $(k-1)!$ and so these groups all vanish. Similarly if $\bk$ has characteristic coprime to $k$. The remaining case is that $k=p$ a prime and $\bk$ has characteristic $p \geq 5$. By \cite[Proposition B]{hausmann} $H_\ell(\fS_{p};\bk^{-})=0$ for $\ell < p-2$, so in particular for $\ell < \frac{3k}{5}$.
\end{proof}

\begin{remark}
This lemma is false if $k=3$ and $\mr{char}(\bk) = 3$ because of an extra class due to $H_1(\fS_3 ; \bF_3^{-}) = \bF_3$.
\end{remark}

\begin{lemma}\label{lem:AlmostRes}
If $(k-1)!$ is invertible in $\bk$ then the augmentation~\eqref{eq:KoszulResolution} induces an isomorphism on homology in tridegrees $(n,p,d)$ with $d < 2k-1$, and an epimorphism for $d=2k-1$.
\end{lemma}
\begin{proof}
By writing $H_*(\bF^\times;\bk)$ as a sum of 1-dimensional graded $\bk$-vector spaces it is enough to show that
\begin{align*}
\epsilon \colon \left(\Lambda^*_\bk(\bk[2i+1]) \otimes \mathrm{Sym}^*_\bk(\bk[2i]), d\right) &\lra \bk\\
\epsilon \colon \left(\mr{Sym}^*_\bk(\bk[2i+2]) \otimes \Lambda^*_\bk(\bk[2i+1]), d\right) &\lra \bk
\end{align*}
are homology isomorphisms in degrees $* < 2k-1$ for all $i \geq 0$. The first is in fact a homology isomorphism in all degrees, and the second in degrees $* < k(2i+2)-1$ under our assumption that $(k-1)!$ is a unit in $\bk$.
\end{proof}

By comparing~\eqref{eq:KoszulResolution} with a free resolution, we see that the domain of~\eqref{eq:ChainMap} calculates $\mr{Tor}$ in bidegrees $(n,p,q)$ with $q < 2k-1$ if $(k-1)!$ is invertible in $\bk$, and hence taking homology of~\eqref{eq:ChainMap} induces a map~\eqref{eq:TorCompare} in such bidegrees.

\begin{corollary}\label{cor:AlmostEpi}
  If $k!$ is invertible in $\bk$ with $k \geq 3$ then~\eqref{eq:TorCompare} is an isomorphism for $q \leq 2k$ and $n \leq k$. If $k \geq 4$ and $(k-1)!$ is invertible in $\bk$, then~\eqref{eq:TorCompare} is an isomorphism for $n=k$ and $p+q < \tfrac{8k}{5}-1$, and is an epimorphism for $n=k$ and $p+q < \tfrac{8k}{5}$. 
\end{corollary} 

\begin{proof}
  In the former case Lemma \ref{lem:AlmostFree} shows that the map of chain complexes \eqref{eq:ChainMap} is an isomorphism for ranks $n \leq k$.  In the latter case it shows
  that~\eqref{eq:ChainMap} is an isomorphism for ranks $n < k$, and that in rank $n=k$ it is an isomorphism on all K{\"u}nneth summands apart from
\[
\mathrm{S}^k_\bk\left(S^{1,1,1} \otimes H_*(\bF^\times;\bk)\right) \otimes H_{0,0,*}(\overline{\gR}_\bk) \to H_{k,*,*}\left(E^+_\infty(S^{1,1,1} \otimes \bk[B\bF^\times])\right)\otimes H_{0,0,*}(\overline{\gR}_\bk)\]
on which it is an isomorphism only in homological degrees (third index) $* < \tfrac{8k}{5}$. The result follows by taking homology.
\end{proof}

As we have described, $E^2_{k,0,k}$ and $E^2_{k,1,k}$ potentially receive $d^r$-differentials starting from $E^r_{k,r,k-r+1}$ and $E^r_{k,r+1,k-r+1}$ respectively, with $r \geq 2$. Under the assumptions that $k \geq 3$ and that $\max(6,(k-1)!)$ is invertible in $\bk$, Corollary \ref{cor:AlmostEpi} shows (using $k+2 \leq \tfrac{8k}{5}$ when $k \geq 4$) that the map \eqref{eq:TorCompare} surjects onto these groups, so to rule out such differentials it is enough to show that
\[\mathrm{Tor}_{r+\epsilon}^{\mathrm{S}^*_\bk(S^{1,0} \otimes H_*(\bF^\times;\bk))}(\bk, H_{*,*}(\overline{\gR}_\bk))_{k,k-r+1}=0\]
for $\epsilon=0,1$ and $r \geq 2$. For clarity, we will prove the following stronger result.

\begin{proposition}\label{prop:for-clarity}
We have $\mathrm{Tor}_p^{\mathrm{S}^*_\bk(S^{1,0} \otimes H_*(\bF^\times;\bk))}(\bk, H_{*,*}(\overline{\gR}_\bk))_{n,q}=0$ for $q < n$ and all $p \geq 0$.
\end{proposition}

\begin{proof}
  We shall deduce this from the theorem of Suslin \cite{SuslinCharClass}, Theorem~\ref{thm:NS} above.

  Let us abbreviate
\begin{align*}
A &\coloneqq \mathrm{S}^*_\bk(S^{1,0} \otimes H_*(\bF^\times;\bk))\\
B &\coloneqq H_{*,*}(\overline{\gR}_\bk),
\end{align*}
and consider these as commutative ring objects in $\cat{Ch}_\bk^{\bN}$ (we will call these cdgas), which happen to have zero differential. As they have zero differential they admit a further weight grading, defined to coincide with the homological grading; we will denote the weight $w$ piece by $(-)^w$.  The $\mr{Tor}$-group in the proposition is then
\begin{equation*} 
  H_{n,p+q}(\bk \otimes^\bL_A B)^{q},
\end{equation*}
which we wish to show vanishes when $q< n$. The number $p$ is not significant for the rest of the argument so we shall change notation and show $H_{n,d}(\bk \otimes^\bL_{A} B)^{w} = 0$ for $w < n$.

We shall do this by manipulating $A$ and $B$ as cdgas. First, define
\[A_0 \coloneqq \mathrm{S}_\bk^*(S^{1,1}\sigma') \otimes A \quad \text{ with differential determined by } d(\sigma') = \sigma\]
and $B_0 \coloneqq A_0 \otimes_A B = \mathrm{S}_\bk^*(S^{1,1}\sigma') \otimes B$, whose differential is then given by the same formula. The weight gradings on $A$ and $B$ extend to these by declaring $\sigma'$ to have weight 0. If $P \overset{\sim}\to \bk$ is a cofibrant resolution as a right $A_0$-module, then it is also cofibrant as a right $A$-module, and we have
\[\bk \otimes^\bL_{A_0} B_0 \simeq P \otimes_{A_0} B_0 = P \otimes_{A_0} (A_0 \otimes_{A} B) \cong P \otimes_A B \simeq \bk \otimes^\bL_A B.\]
Thus $H_{n,d}(\bk \otimes^\bL_A B)^{w}= H_{n,d}(\bk \otimes^\bL_{A_0} B_0)^{w}$; we may show the latter vanishes.  

We have
\[A = \mathrm{S}^*_\bk(S^{1,0} \otimes H_*(\bF^\times;\bk)) = \mathrm{S}^*_\bk(S^{1,0}\sigma) \otimes \mathrm{S}^*_\bk(S^{1,0} \otimes \tilde{H}_*(\bF^\times;\bk)),\]
a free module over $\mathrm{S}^*_\bk(S^{1,0}\sigma) = \bk[\sigma]$. As the chain complex $\mathrm{S}_\bk^*(S^{1,1}\sigma') \otimes \mathrm{S}_\bk^*(S^{1,0}\sigma)$ with the differential $d(\sigma') = \sigma$ is quasi-isomorphic to $\bk$, we find that
\[A_0 \simeq \mathrm{S}^*_\bk(S^{1,0} \otimes \tilde{H}_*(\bF^\times;\bk)).\]
This cdga has zero differential, and vanishes in bidegrees $(n,d)$ with $d < n$. Furthermore $H_{n,d}(A_0)$ is concentrated in weight $d$.

We do not know that $B$ is a free $\mathrm{S}^*_\bk(S^{1,0}\sigma)$-module, but
we have a long exact sequence
\[H_{n-1,d}(\overline{\gR}_\bk) \overset{\sigma \cdot }\lra H_{n,d}(\overline{\gR}_\bk) \lra H_{n,d}(B_0) \lra H_{n-1,d-1}(\overline{\gR}_\bk) \overset{\sigma \cdot }\lra H_{n,d-1}(\overline{\gR}_\bk).\]
which shows that $H_{n,d}(B_0)$ is a sum of parts of weights $d$ and $d-1$. It follows from the theorem of Suslin (Theorem~\ref{thm:NS}) that $H_{n,d}(B_0)=0$ for $d<n$, and that $H_{n,n}(\overline{\gR}_\bk) \to H_{n,n}(B_0)$ is surjective so the latter has weight precisely $n$. That is, for $d = n$ there is no weight $d-1$ part.

Therefore  $H_{n,d}(A_0)$ and $H_{n,d}(B_0)$ both vanish for $d < n$ and are both concentrated in weights $\geq n$. These properties are preserved by tensor product, so for $s \geq 0$ the object
\begin{equation*}
  (H_{*,*}(A_0)^{\otimes s} \otimes H_{*,*}(B_0))
\end{equation*}
is concentrated in bidegrees $(n,d)$ with $d \geq n$, and in weights $\geq n$.  As $s$ varies, these tensor powers form the $E^1$ page of a spectral sequence converging to the homology of $\bk \otimes^\bL_{A_0} B_0 \simeq \bk \otimes^\bL_{A} B$.  Therefore $H_{n,d}(\bk \otimes^\bL_{A} B)^w = 0$ for $w < n$, as claimed.
\end{proof}

It follows that $E^2_{k,0,k}$ and $E^2_{k,1,k}$ must survive the spectral sequence, and so must be subquotients of $H_{k,d}(B(\bk, \overline{\gE_\infty(S^{1,0} \otimes \bk[B\bF^\times])}, \overline{\gR}_\bk))$ with $d=k$ or $d=k+1$ respectively.  Exactness of~\eqref{eq:34} for $k \geq 3$ when $\max(6,(k-1)!)$ is invertible in $\bk$ is completed by showing that these groups vanish.  In fact we prove the following stronger statement, which finishes the argument as we have assumed that $(k-1)!$ is invertible in $\bk$.

\begin{proposition}\label{prop:MirzaiiFinalStep}
  If $6 \in \bk^\times$, then $H_{n,d}(B(\bk, \overline{\gE_\infty(S^{1,0} \otimes \bk[B\bF^\times])}, \overline{\gR}_\bk))=0$ for $d < \tfrac{3}{2}n$.
\end{proposition}

\begin{proof}
Recall that we write $\gA_\bk = \gE_\infty(S^{1,0} \otimes \bk[B\bF^\times])$. From Figure \ref{fig:eInfHomology} we see that $H^{E_\infty}_{n,d}(\gR_\bk,\gA_\bk)=0$ for $\frac{d}{n} < \tfrac{3}{2}$, using the assumption that 2 and 3 are invertible in $\bk$. We can therefore apply Theorem $E_k$.15.9 to $\gA_\bk \to \gR_\bk$ with $\rho(n) = \tfrac{3}{2} n$, $\gM = \bk$, and $\mu(n) = \tfrac{3}{2} n$, to conclude that $H_{n,d}(B(\bk, \overline{\gA}_\bk, \overline{\gR}_\bk))=0$ for $d < \tfrac{3}{2}n$ as required.
\end{proof}

\begin{proof}[Proof of Theorem~\ref{thm:WeakInj}]
  This follows by induction when~\eqref{eq:34} is exact, as outlined above Remark~\ref{rem:Mirzaii-remarks}.  The base cases are $n=1$ and $n=2$, which are easy.  (This argument does not cover $n=3$ when $\mr{char}(\bk) = 3$.)
\end{proof}

\section{Further applications}\label{sec:AppFurther} Having discussed applications to homological stability, we finish with three applications of a different nature.

\subsection{Improved range for the Milnor conjecture for $\mr{GL}_n$}  \label{sec:milnor} Let $G$ be a Lie group, then we can compute its homology both as a topological group and as a discrete group. These are given by the homology of the classifying spaces $BG^\mr{top}$ and $BG$ respectively. (We use a superscript $\mr{top}$ to denote the Lie group as the topological group, as we have been using $\mr{GL}_n(-)$ as notation for the discrete group in the entire paper.) The canonical continuous homomorphism $G \to G^\mr{top}$ induces a map on homology with $\bk$-coefficients
\[H_*(BG;\bk) \lra H_*(BG^\mr{top};\bk).\]

Milnor has shown \cite[Theorem 1]{milnor-lie} that if $G^\mr{top}$ has finitely many components and $\bk$ is finite then this map is split surjective, and has conjectured that it is an isomorphism for any Lie group when $\bk$ is finite \cite[\S 1]{milnor-lie}. Suslin has proved that this for $\mr{GL}_n(\bC)$ as well as $\mr{GL}_n(\bR)$, in the limit as $n \to \infty$: the maps
\begin{align*}H_*(B\mr{GL}_\infty(\bR);\bk) &\lra H_*(B\mr{GL}_\infty(\bR)^\mr{top};\bk), \\ H_*(B\mr{GL}_\infty(\bC);\bk) &\lra H_*(B\mr{GL}_\infty(\bC)^\mr{top};\bk),\end{align*}
are isomorphisms for all degrees $*$ when $\bk$ is finite \cite[Corollary 4.8]{suslinlocal}. Using homological stability results this implies that the Milnor conjecture is true for $\mr{GL}_n(\bC)$ and $\mr{GL}_n(\bR)$ in a range of degrees going to infinity with $n$. For example, the results of Nesterenko and Suslin imply it is true in the range $* < n$. Our Theorem \ref{thm:alg-closed-1} implies that for all primes $p$ the map
\[H_d(\mr{GL}_n(\bC);\bZ/p) \lra H_d(\rm{GL}_\infty(\bC);\bZ/p)\]
is an isomorphism for $d < \tfrac{5}{3}n-1$ so we deduce that the Milnor conjecture holds for $\mr{GL}_n(\bC)$ in this range of degrees too. We obtain the same result for $\bR$ for odd primes $p$, using Remark \ref{rem:k-theory-r}.

\begin{remark}Morel has proven that the Milnor conjecture is equivalent to several other conjectures in \'etale cohomology and $\bA^1$-homotopy theory, and the same relations hold in a range  \cite{morel}.\end{remark}

\subsection{Vanishing results for Steinberg homology}\label{sec:steinberg-vanishing}

Conjectures of Church, Farb and Putman \cite{churchfarbputman} as well as applications in algebraic $K$-theory such as \cite{quillenintegers} involve the homology of $\mr{GL}_n(\R)$ with coefficients in the Steinberg module $\mr{St}(\R^n)$. Combining Theorem \ref{thmcor:BlockvsFlag} and Theorem \ref{thm:split-building-ek}, we see that if $(\R, \bk)$ satisfies the Nesterenko--Suslin property then there is an isomorphism
\begin{equation}\label{eqn:st-e1} H^{E_1}_{n,d}(\gR_\bk) \cong H_{d-n+1}(\mr{GL}_n(\R);\mr{St}(\R^n) \otimes \bk),\end{equation}
so the homology of the Steinberg module can be interpreted as $E_1$-homology of $\gR_\bk$. 

The following vanishing result for $E_1$-homology follows from our vanishing result for $E_2$-homology:

\begin{theorem}\label{thm:VanishSteinbergHomology}
If $\R$ is a connected semi-local ring with all residue fields infinite, then $H^{E_1}_{n,d}(\gR_\bZ) = 0$ for $d<\tfrac{3}{2}(n-1)$.
\end{theorem}

Under the isomorphism \eqref{eqn:st-e1} this corresponds to the vanishing of the homology groups $H_{d}(\mr{GL}_n(\R);\mr{St}(\R^n))$  for $d< \tfrac{1}{2}(n-1)$, which is Theorem~\ref{thm:SteinbergHomology}. In the infinite field case this recovers the main result of \cite{ashputmansam} for general linear groups (we did the case of finite fields in \cite[Section 7]{e2cellsIII}).  See \cite[Section 6]{millernagpalpatzt} for another argument in the case of fields. The result for connected semi-local rings with all residue fields infinite is new.

\begin{proof}[Proof of Theorem \ref{thm:VanishSteinbergHomology}]
By Theorem \ref{thm:E2Homology}, or more generally Theorem \ref{thm:Solomon--Tits-local} combined with Theorem \ref{thmcor:BlockvsFlag}, under the given assumptions we have that $H_{n,d}^{E_2}(\gR_\bZ)=0$ for $d< 2(n-1)$. Thus we also have $\smash{H_{n,d}^{E_\infty}(\gR_\bZ)}=0$ for $d< 2(n-1)$, by Theorem $E_k$.14.4. With $\bQ$-coefficients we can then apply Theorem $E_k$.14.7, giving that $\smash{H^{E_1}_{n,d}(\gR_\bQ)} = 0$ for $d<\tfrac{3}{2}(n-1)$ as claimed. To obtain the claimed result we will explain how to adapt this argument to $\bZ_{(p)}$-coefficients, spelling out and refine the claims made in Remark $E_k$.14.8.

Note that as $S^{0,0} \oplus \Sigma Q^{E_1}_\bL(\gR_\bZ) \simeq B^{E_1}(\gR^+_\bZ)$, it is equivalent to show that $H_{n,d}(B^{E_1}(\gR^+_\bZ))=0$ for $d-1 < \tfrac{3}{2}(n-1)$.

First suppose that $p \geq 5$. As in the proof of Theorem $E_k$.14.6, by finding a suitable CW-approximation to $\gR_\bZ$ and filtering by skeleta we can reduce to the case of $\gE_\infty(X)$ where $X$ only has $(n,d)$-cells with $d \geq 2(n-1)$. In fact we can arrange that $X = S^{1,0}_\bZ \sigma \oplus \tilde{X}$ where $\tilde{X}$ only has $(n,d)$-cells satisfying $d \geq 2(n-1)$ and $d \geq n$. By Theorem $E_k$.13.8 we have\[B^{E_1}(\gE_\infty^+(X)) \simeq \gE_\infty^+(\Sigma X) \simeq \gE_\infty^+(S^{1,1}_\bZ \bar{\sigma}) \otimes \gE_\infty^+(\Sigma\tilde{X}),\]
and $\Sigma X$ only has $(n,d)$-cells satisfying $d \geq \tfrac{3}{2} n$, so $\gE_\infty^+(\Sigma\tilde{X})$ does too. On the other hand inspecting the description of the homology of free $E_\infty$-algebras (see Section $E_k$.16.3) we see that $H_{n,d}(\gE_\infty^+(S^{1,1} \bar{\sigma});\bF_p)=0$ for $d-1 < \tfrac{3}{2}(n-1)$; this is where we use that $p \geq 5$. The integral homology of $\gE_\infty^+(S^{1,1} \bar{\sigma})$ has finite type, so this implies that $H_{n,d}(\gE_\infty^+(S^{1,1} \bar{\sigma});\bZ_{(p)})=0$ for $d-1 < \tfrac{3}{2}(n-1)$. By the cell-attachment spectral sequence (Corollary $E_k$.10.17), attaching further $E_\infty$-$(n,d)$-cells with $d \geq \tfrac{3}{2}n$ cannot create any homology in these degrees, giving the required vanishing range for $\bZ_{(p)}$-homology.

For $p=3$ this argument does not quite work, because the classes $(\beta Q^1_3(\bar{\sigma}))^k$ of bidegree $(3k,4k)$ have $\tfrac{4k-1}{3k-1} \geq \tfrac{3}{2}$ as soon as $k \geq 2$. To fix this we shall use the fact that we know, from Lemma \ref{lem:SNrelations}, the relation $(\beta Q^1_3)_\bZ(\sigma) = \sigma \cdot b$ in the homology of $\gR_{\bZ}$. Thus we can obtain $\gR_{\bZ}$ from
\[\gE_\infty(S^{1,0} \sigma \oplus S^{2,4} b) \cup^{E_\infty}_{(\beta Q^1_3)_\bZ(\sigma) - \sigma \cdot b} D^{3,5}\rho_3\]
by attaching $(n,d)$-cells satisfying $d \geq 2(n-1)$ and $d \geq n$. As above, it is then enough to show that the $\bZ_{(3)}$-homology of 
\begin{equation}\label{eq:SmallModelAt3}
B^{E_1}(\gE_\infty^+(S^{1,0} \sigma \oplus S^{2,4} b) \cup^{E_\infty}_{(\beta Q^1_3)_\bZ(\sigma) - \sigma \cdot b} D^{3,5}\rho_3)
\end{equation}
vanishes in bidegrees $(n,d)$ with $d-1 < \tfrac{3}{2}(n-1)$. As this has finite type it is enough to show its $\bF_3$-homology vanishes in this range of degrees. We have
\begin{equation}\label{eq:BarOnce}
B^{E_1}(\gE_\infty^+(S^{1,0} \sigma \oplus S^{2,4} b)) \simeq \gE_\infty^+(S^{1,1} \bar{\sigma} \oplus S^{2,5} \bar{b})
\end{equation}
and \eqref{eq:SmallModelAt3} is obtained from this by attaching an $E_\infty$-cell along the map
\[\bar{\rho}_3 \colon S^{3,5} = \Sigma S^{3,4} \xrightarrow{\Sigma \rho_3}  \Sigma \gE_\infty^+(S^{1,0} \sigma \oplus S^{2,4} b) \lra B^{E_1}(\gE_\infty^+(S^{1,0} \sigma \oplus S^{2,4} b)),\]
which we must determine.

\vspace{1ex}

\noindent \textbf{Claim}. Under the equivalence \eqref{eq:BarOnce} we have $\bar{\rho}_3 = (\beta Q^1_3)_\bZ(\bar{\sigma})$ up to units.

\vspace{1ex}

Here $(\beta Q^1_3)_\bZ(\bar{\sigma}) \in H_{3,5}(\gE_\infty^+(S^{1,1} \bar{\sigma}))$ denotes the unique class which reduces to $\beta Q^1_3(\bar{\sigma})$ modulo 3. 

\begin{proof}[Proof of claim]
For any $E_1$-algebra $\gA$, the map
\[s \colon H_{n,d}(\gA^+) \overset{\sim}\lra H_{n,d+1}(\Sigma \gA^+) \lra H_{n,d+1}(B^{E_1}(\gA^+))\]
annihilates classes which are $E_1$-decomposable, that is, decomposable with respect to the product on homology induced by the multiplication of the $E_1$-algebra structure. Thus $\bar{\rho}_3 = s(\rho_3) = s ( (\beta Q^1_3)_\bZ(\sigma))$, and we must show that this is $(\beta Q^1_3)_\bZ(\bar{\sigma})$ up to units. It is enough to check this after passing to $\bF_3$-coefficients, where we can consider the  bar spectral sequence (Theorem $E_k$.14.1) for ${B}^{E_1}(\gE_\infty^+(S^{1,0}\sigma))$. By the description of the homology of free $E_\infty$-algebras (see Section $E_k$.16.2) we have that $H_{*,*}(\gE_\infty^+(S^{1,0}\sigma);\bF_3)$ is given by
\[\mr{S}^*_{\bF_3}\left[Q^I_3(\sigma) \, \middle\vert \, \text{$I = (\epsilon_1,s_1,\ldots,\epsilon_r,s_r)$ admissible, $e(I)+\epsilon_1>0$}\right],\]
a free graded-commutative algebra. 

Thus the bar spectral sequence starts with the tensor product of exterior algebras on classes $sQ^I_3(\sigma)$ when $Q^I_3(\sigma)$ is as above and has even degree, and divided power algebras on classes $sQ^I_3(\sigma)$ when $Q^I_3(\sigma)$ is as above and has odd degree. It converges to $H_{*,*}({B}^{E_1}(\gE_\infty^+(S^{1,0}\sigma));\bF_3)$. (The notation $s$ used here is consistent with that above, as the map $\Sigma \gA^+ \to B^{E_1}(\gA^+)$ is the inclusion of the 1-skeleton of the reduced bar construction.) The class $s \beta Q^1_3(\sigma)$ has tridegree $(3,4,1)$, so total bidegree $(3,5)$, and up to units is the unique class of this total bidegree. As 
\[\beta Q^1_3(\bar{\sigma}) \neq 0 \in H_{3,5}(\gE_\infty^+(S^{1,1} \bar{\sigma});\bF_3) = H_{3,5}({B}^{E_1}(\gE_\infty^+(S^{1,0}\sigma));\bF_3)\]
this must indeed be equal to $s \beta Q^1_3(\sigma)$ up to units.
\end{proof}

Thus \eqref{eq:SmallModelAt3} is equivalent to
\[\gE_\infty^+(S^{1,1} \bar{\sigma} \oplus S^{2,5} \bar{b}) \cup^{E_\infty}_{(\beta Q^1_3)_\bZ(\bar{\sigma})} D^{3,6} \bar{\rho}_3,\]
and we must show that its $\bF_3$-homology vanishes in bidegrees $(n,d)$ with $d-1 < \tfrac{3}{2}(n-1)$. It is enough to show that the homology of 
\[\gE_\infty^+(S^{1,1} \bar{\sigma}) \cup^{E_\infty}_{(\beta Q^1_3)_\bZ(\bar{\sigma})} D^{3,6} \bar{\rho}_3\]
has this vanishing line, because adding the trivially attached $E_\infty$-$(2,5)$-cell $\bar{b}$ has the effect of applying $- \otimes \gE_\infty^+(S^{2,5} \bar{b})$, which preserves this vanishing line. To show this we use the cell-attachment spectral sequence, which starts with
\[E^1_{n,p,q} = H_{n,p+q,p}(\gE_\infty^+(S^{1,1,0} \bar{\sigma} \oplus S^{3,6,1} \bar{\rho}_3);\bF_3) = \mr{S}^*_{\bF_3}[Q^I_3(\bar{\sigma}), Q^J_3(\bar{\rho}_3)],\]
where $I$ and $J$ are admissible and satisfy the appropriate excess condition. The $d^1$-differential satisfies $d^1 (\bar{\rho}_3) = \beta Q^1_3(\bar{\sigma})$ and by Theorem $E_k$.16.8 it vanishes on all other generators. By graded-commutativity we have $\bar{\sigma}^2=0$. We can therefore write $(E^1_{*,*,*}, d^1)$ as 
\[\left(\Lambda_{\bF_3}[\bar{\sigma}, \bar{\rho}_3] \otimes \bF_3[ \beta Q^1_3(\bar{\sigma})], d \bar{\rho}_3 = \beta Q^1_3(\bar{\sigma})\right) \otimes \mr{S}^*_{\bF_3}[\text{classes of slope $\geq \tfrac{3}{2}$}].\]
Thus  $E^2_{*,*,*} = \Lambda_{\bF_3}[\bar{\sigma}] \otimes \mr{S}^*_{\bF_3}[\text{classes of slope $\geq \tfrac{3}{2}$}]$ which vanishes in bidegrees $(n,d)$ with $d-1 < \tfrac{3}{2}(n-1)$ as required.

For $p=2$ we proceed in a similar way, using the relation $(Q^1_2)_\bZ(\sigma) = \sigma \cdot a$ in the homology of $\gR_{\bZ}$. This is imposed by a cell $\rho_2$, and the analogue of the Claim above is that $\bar{\rho}_2 = \bar{\sigma}^2$, which may be shown in the same way. As above we reduce to showing that the $\bF_2$-homology of $\gE_\infty^+(S^{1,1} \bar{\sigma}) \cup^{E_\infty}_{\bar{\sigma}^2} D^{2,3} \bar{\rho}_2$ vanishes in bidegrees $(n,d)$ with $d-1 < \tfrac{3}{2}(n-1)$. The $E^1$-page of the cell attachment spectral sequence is now
\[(\bF_2[\bar{\sigma}, \bar{\rho}_2], d^1 \bar{\rho}_2 = \bar{\sigma}^2) \otimes \bF_2[\text{classes of slope $\geq \tfrac{3}{2}$}]\]
so the $E^2$-page is $\bF_2[\bar{\sigma}, \bar{\rho}_2^2]/(\bar{\sigma}^2)  \otimes \bF_2[\text{classes of slope $\geq \tfrac{3}{2}$}]$. As $\bar{\rho}_2^2$ also has slope $\geq \tfrac{3}{2}$, the $E^2$-page vanishes in bidegrees $(n,d)$ with $d-1 < \tfrac{3}{2}(n-1)$ as required.
\end{proof}

\begin{remark}
We believe that the unit in the Claim is equal to 1. When $\gA$ is an $E_k$-algebra,  the map
\[s \colon H_{n,d}(\gA^+;\bF_p) \overset{\sim}\lra H_{n,d+1}(\Sigma \gA^+;\bF_p) \lra H_{n,d+1}(B^{E_1}(\gA^+);\bF_p)\]
should commute with those Dyer--Lashof operations which are defined for $E_{k-1}$-algebras, cf.\ \cite[Lemma 7.2]{BokstedtTHHZ}, \cite[Proposition 5.9]{AngeltveitRognes}. 
\end{remark}

\begin{remark}In fact, one can try to compute the Steinberg homology with rational coefficients when $\R=\bF$ is a number field, using the results of Borel and Yang \cite{borelyang}. This is analogous to the computations we did for Steinberg homology of finite fields in \cite[Section 7]{e2cellsIII}, but seems to be algebraically more involved.
\end{remark}

\subsection{A question of Quillen and comparing rank filtrations}\label{sec:QuillenQuestion}

Finally, we pick up a thread from Remark \ref{rem:ns-other}. In \cite[p.203]{QuillenChar} Quillen has asked---when translated to the language of this paper---whether the pair $(\bZ, \bQ)$ might satisfy the Nesterenko--Suslin property.

\begin{theorem}
The pair $(\bZ, \bQ)$ does not satisfy the Nesterenko--Suslin property.
\end{theorem}

Rognes has pointed out to us that this may be proved by elementary means by directly computing the rational homology of the group  $U \leq \mr{GL}_3(\bZ)$ of upper triangular matrices, and finding that $H_3(U;\bQ) \neq 0$ (whereas the diagonal matrices have trivial rational homology). This calculation is given in \cite[Lemma 6.8]{RognesDeg5}.

We shall give a proof based on $E_k$-homology, which might illuminate the role of the Nesterenko--Suslin property in earlier sections.

\begin{proof}
If it did, then by Theorem \ref{thmcor:BlockvsFlag} we would have a $\bQ$-homology equivalence
  \[\tilde{D}^{1}(\bZ^n) \hcoker \mr{GL}_n(\bZ) \lra {D}^{1}(\bZ^n) \hcoker \mr{GL}_n(\bZ)\]
for each $n$. The map $V \mapsto V \otimes \bQ$ induces a homeomorphism 
${D}^{1}(\bZ^n) \cong {D}^{1}(\bQ^n)$, so by the Solomon--Tits theorem (Theorem \ref{thm:SolomonTits}) and Lemma~\ref{lem:splittingcomplex-vs-building}, ${D}^{1}(\bZ^n)$ is also a wedge of $n$-spheres, with top homology $\mr{St}(\bZ^n) \cong \widetilde{H}_{n}({D}^{1}(\bZ^n))$.  Considering the $E_\infty$-algebra $\gR \simeq \bigsqcup_{n>0} B\mr{GL}_n(\bZ)$ constructed by the method of Section \ref{sec:contructing-r}, this would give an isomorphism as in \eqref{eqn:ek-homology-buildings}
\[H^{E_1}_{n,d}(\gR_\bQ) \cong H_{d-(n-1)}(\mr{GL}_n(\bZ) ; \mr{St}(\bZ^n)\otimes \bQ).\]
We will compute the two sides for $n=3$ by independent means, and show that this leads to a contradiction.

It is well known that $\tilde{H}_*(\mr{SL}_2(\bZ);\bQ)=0$, and Soul{\'e} has shown \cite[Theorem 4]{Soule} that $\tilde{H}_*(\mr{SL}_3(\bZ);\bQ)=0$ too. The same vanishing then holds for $\mr{GL}_2(\bZ)$ and $\mr{GL}_3(\bZ)$. Computing $E_1$-homology using the bar spectral sequence shows that $H^{E_1}_{3,*}(\gR_\bQ)=0$. 

On the other hand, $\mr{SL}_3(\bZ)$ is a virtual duality group of dimension $\binom{3}{2}=3$, and $\mr{St}(\bZ^3)$ is its rational dualising module. Using the splitting $\mr{GL}_3(\bZ) \cong \mr{SL}_3(\bZ) \times \bZ^\times$ where the second factor is given by $\begin{bsmallmatrix} -1 & 0 & 0 \\0 & -1 & 0\\ 0 & 0 & -1 \end{bsmallmatrix}$ which acts trivially on $\mr{St}(\bZ^3)$, $\mr{GL}_3(\bZ)$ is also a virtual duality group of dimension $3$, with the same rational dualising module $\mr{St}(\bZ^3)$. Thus $H_{3}(\mr{GL}_3(\bZ) ; \mr{St}(\bZ^3)\otimes \bQ) \cong H^{0}(\mr{GL}_3(\bZ);\bQ) = \bQ$, which gives the required contradiction.
\end{proof}

The discussion in this proof shows that Rognes' stable rank filtration of $K_*(\bZ)$ differs significantly from our split stable rank filtration. Firstly, it is shown in \cite{RognesDeg5} that a generator of $K_5(\bZ) \otimes \bQ = \bQ$ has filtration 3 where it represents a non-zero element in the associated graded with respect to the stable rank filtration. On the other hand, as we have seen in the proof above the homology groups $H_{n,d}(\gR_\bQ) = H_d(\mr{GL}_n(\bZ);\bQ)$ vanish for $n \leq 3$ and $d > 0$, and from this we can calculate the $E_\infty$-homology of $\gR_\bQ$ in this range, by iterated bar constructions:
\[H_{n,d}^{E_\infty}(\gR_\bQ) = \begin{cases}
\bQ & \text{ if $(n,d) = (1,0)$,}\\
0 & \text{ otherwise, for $n \leq 3$.}
\end{cases}\]
Using the fact that the $E_\infty$-homology is the homology of the associated graded of the split stable rank filtration, as in Remark \ref{rem:RkFiltComparison}, it follows that $K_5(\bZ) \otimes \bQ$ has filtration $> 3$ with respect to the split stable rank filtration.

In contrast, as we mentioned in Remark \ref{rem:ns-other}, Quillen has shown that $(\bZ[\tfrac{1}{p}], \bQ)$ satisfies the Nesterenko--Suslin property, and so as in Remark \ref{rem:RkFiltComparison} the split and ordinary stable rank filtrations of $K_*(\bZ[\tfrac{1}{p}]) \otimes \bQ$ coincide.  The localisation sequence in $K$-theory implies that the map $K_*(\bZ) \otimes \bQ \to K_*(\bZ[\tfrac{1}{p}]) \otimes \bQ$ is an isomorphism for $*>1$, and by functoriality of the stable rank filtration a generator of $K_5(\bZ[\frac1p]) \otimes \bQ$ is contained in filtration 3.  From the vanishing of $H_i(\mr{SL}_2(\bZ[\frac1p]);\bQ)$ for $i > 2$ (see e.g.\ \cite[Section II.1.4]{SerreTrees}) it may be deduced in a way similar to the above that $F_2 K_5(\bZ[\frac1p]) \otimes \bQ = 0$.  In the associated graded for the (split) stable rank filtration on $K_5(\bZ[\frac1p])\otimes \bQ$, a generator therefore shows up in filtration degree 3.  In particular, the isomorphism $K_5(\bZ)\otimes \bQ \to K_5(\bZ[\frac1p])\otimes \bQ$ is an isomorphism of filtered vector spaces with regards to the stable rank filtrations, but not with regards to the split stable rank filtrations.

\let\oldaddcontentsline\addcontentsline
\renewcommand{\addcontentsline}[3]{}

\bibliographystyle{amsalpha}
\bibliography{../../biblio}

\vspace{1cm}

\let\addcontentsline\oldaddcontentsline

\end{document}